\documentclass[11pt, reqno]{amsart}
\usepackage{amsmath, amssymb, amsthm, graphicx,marvosym}
\usepackage{cancel}
\usepackage{cite}
\usepackage[nobysame]{amsrefs}[11pts]
\usepackage{bm}
\usepackage{color}
\usepackage{amsfonts}
\usepackage{mathrsfs}

\newtheorem{theorem}{Theorem}[section]
\newtheorem{definition}{Definition}[section]
\newtheorem{lemma}{Lemma}[section]
\newtheorem{remark}{Remark}[section]

\newtheorem{corollary}{Corollary}[section]
\numberwithin{equation}{section}
\numberwithin{figure}{section}

\makeatletter \@addtoreset{equation}{section} \makeatother

\def\tilde{\widetilde}

\newcommand{\fr}{\frac}
\newcommand{\ef}{\eqref}

\newcommand{\beq}{\begin{equation}}
	\newcommand{\eeq}{\end{equation}}

\newcommand{\dv}{{\mathrm {div}}}

\usepackage{amsmath}
\allowdisplaybreaks

\begin{document}

	\title[Degenerate compressible Navier-Stokes equations]{Local-in-time well-posedness for the regular solution to the 2D full compressible Navier-Stokes equations with degenerate viscosities and heat conductivity}

\author{Yue Cao}
\address[Y. Cao]{School of Mathematics, East China University of Science and Technology, Shanghai, 200237, China.}
\email{\tt cao\_yue12@ecust.edu.cn}

\author{Xun Jiang}
\address[X. Jiang]{School of Mathematical Sciences, Shanghai Jiao Tong University, Shanghai, 200240, China.}
\email{\tt xunjiang@sjtu.edu.cn}




\date{\today}

\begin{abstract}
	
	This paper considers  the  two-dimensional  Cauchy problem of the full compressible Navier-Stokes equations with far-field vacuum in $\mathbb{R}^2$, where the viscosity and heat-conductivity coefficients  depend on the absolute temperature $\theta$ in the form of $\theta^\nu$ with $\nu>0$. 
    Due to the appearance of the vacuum, the momentum equation are both degenerate in the time evolution and spatial dissipation, which makes the study on the well-posedness challenged. By establishing some new singular-weighted (negative powers of the density $\rho$) estimates of the solution, we establish the local-in-time well-posedness of the regular solution with far-field vacuum in terms of $\rho$, the velocity $u$ and the entropy $S$. 
	    

\end{abstract}
\maketitle

\section{Introduction}

The full  compressible Navier-Stokes equation (\textbf{FCNS}) describes the movement of viscous heat-conductive fluid, which is given by:
\begin{equation}\label{1}
	\left\{\begin{aligned}
		&\rho_t+\text{div}(\rho u)=0,\\[5pt]
		&(\rho u)_t+\text{div}(\rho u\otimes u)+\nabla P=\text{div}\mathbb{T},\\[5pt]
		&(\rho \mathcal{E})_t+\text{div}(\rho \mathcal{E} u+Pu)=\text{div}(u\mathbb{T})+\text{div}(\kappa\nabla\theta).
	\end{aligned}\right.
\end{equation}
In \eqref{1}, $t\geq 0$ and $x=(x_1,x_2)^\top\in \mathbb{R}^2$ are the time and the spatial coordinate, respectively, $\rho\geq 0$ the density, $u=(u^{(1)},u^{(2)})^{\top}$ the fluid velocity, $\theta$ the absolute temperature, $\mathcal{E}=e+\frac{1}{2}|u|^2$  the specific total energy,  $e$ the specific internal energy. $P$ is the pressure, for polytropic fluids,  the  equation of state satisfies 
\begin{equation}\label{2}
	P=R\rho\theta=(\gamma-1)\rho e=Ae^{S/c_v}\rho^\gamma,\quad e=c_v\theta=\frac{R}{\gamma-1} \theta,
\end{equation}
$R>0$ is the gas constant, $A$ is a   positive constant, $c_v>0$ is  the specific heat at unit volume,  $\gamma>1$ is the adiabatic exponent and $S$ is the entropy. $\mathbb{T}$ is the viscosity stress tensor given by
\begin{equation}\label{3}
	\mathbb{T}=2\mu D(u)+\lambda \text{div}u \mathbb{I}_2,
\end{equation}
where $D(u)=\frac{\nabla u+(\nabla u)^\top}{2}$ is the deformation tensor, $\mathbb{I}_2$ is the $2\times 2$ identity matrix, $\mu$ is the shear viscosity coefficient, and $\lambda+\mu$ is the bulk viscosity coefficient. $\kappa$ denotes the coefficient of heat conductivity. 

For constant viscous heat-conductive flow, there is a lot of literature on the well-posedness and behavior of solutions to the compressible Navier-Stokes equation. 
When the initial data are away from vacuum, as demonstrated in \cites{itaya1, KA, nash, serrin, tani,ksone} and so on, the local well-posedness of classical solutions to \ef{1} can be established through the standard symmetric hyperbolic-parabolic structure that satisfies Kawashima's condition. 

When the initial vacuum is allowed, some new difficulties arise.  For example, the degeneracy of the time evolution in both the momentum and energy equations made it difficult to characterize the behavior of $u$, $\theta$ and $S$ near the vacuum domain. 
  For the constant viscous flow with vacuum, 
  Salvi and Stra$\rm{\check{s}}$kraba \cites{sal} first introduce the compatibility conditions and obtain the local-in-time existence of strong solutions to the isentropic flow, which has been extended to the non-isentropic flow by the work of Cho and Kim \cites{cho4}. For the global-in-time existence of classical solutions, we refer to Huang, Li and Xin \cites{HLX} for the isentropic flow in $\mathbb{R}^3$, 
and Huang-Li \cite{arma2018}, Wen-Zhu \cites{WH},  for the non-isentropic flow in $\mathbb{R}^3$. 
Some other results can be refer to \cites{FM2,jensen, ls15, song,mat} and the references cited therein.


It was pointed that Xin-Zhu \cite{XinZ1} and Duan-Xin-Zhu \cite{DXZ1} reveal a phenomenon: the classical solutions with vacuum to the constant viscous heat-conductive flow cannot preserve the conservation of momentum. In fact, the assumption of constant viscosity leads to a nonphysical scenario, that is, the vacuum exerts an external force on the fluid at the vacuum boundary.  Thus it is more appropriate to adopt degenerate viscosity coefficients to the viscous heat-conductive flows when modeling the vacuum problems.
  Besides, according to the theory of gas dynamics,   \ef{1} can be derived from the Boltzmann equations through the Chapman-Enskog expansion \cites{chap} to the first order,  and under some proper physical assumptions, $\mu$, $\lambda$ and $\kappa$ display as 
\begin{equation}
	\label{eq:1.5g}
	\begin{split}
		\mu(\theta)=&\alpha \theta^{\mathrm{b}},\quad  \lambda(\theta)=\beta \theta^{\mathrm{b}}, \quad \kappa(\theta)=\hbar \theta^{\mathrm{b}},
	\end{split}
\end{equation}
for
$\mathrm{b}=\frac12+\frac{2}{\Upsilon} \in \big[\frac12,\infty\big)$, where for Maxwellian molecules, the rigid elastic spherical molecules and the ionized gas,
$\mathrm{b} = 1, \frac12, \frac52$,  respectively.
%
%
%
For the isentropic flow, according to Boyle and Gay-Lussac law,  the dependence on $\theta$ is reduced to the dependence on $\rho$, i.e.
\begin{equation}\label{15r}
\mu(\rho)=\alpha \rho^\mathrm{c},\quad \lambda(\rho)=\beta \rho^\mathrm{c}.
\end{equation}

For isentropic flow with density-dependent viscosities as in \eqref{15r}, the appearance of vacuum will cause the degenerate of both time evolution and spatial dissipation, and make it hard to establish the well-posedness theory. 
 Bresch and Desjardins \cites{DB1} introduced the BD entropy identity,
which helps to establish the well-posedness of weak solutions to the compressible Navier-Stokes equations, especially the study of weak solutions of the case with density-dependent viscosities, see \cites{VaYu2,GZH,LiX,VaYu1}.  When $\text{c}=1$ in \eqref{15r}, Li-Pan-Zhu \cites{LPZ1}  established the local-in-time existence of classical solutions without using the BD entropy, see also Li-Pan-Zhu \cites{LPZ2} for the case of $1<\text{c}\le\frac{\gamma+1}{2}$; when  $0<\text{c}<1$ in \eqref{15r}, 
 Xin-Zhu \cite{XinZ1} obtain the local-in-time existence of classical solutions with conserved total mass and momentum in $\mathbb{R}^3$. For the global existence of classical solutions, we refer to Cao-Li-Zhu \cite{CLZ} for the case in $\mathbb{R}$, Huang-Li \cite{lh22} for the case in $\mathbb{R}^2$ and Xin-Zhu \cites{XZ2} for the case in $\mathbb{R}^3$. 

For the degenerate non-isentropic flow with viscosity and heat-conductivity in \eqref{eq:1.5g}, due to the appearance of the temperature in the viscous stress tensor and the heat flux vector, the difficulty of the problem will reach a new level. More precisely, for the isentropic flow, we can compare the density power to obtain the degeneracy of the time evolution and spatial dissipation, and then obtain the principal mathematical structure of the momentum equations. However, for the non-isentropic flow, the degeneracy of the time evolution is characterized by the power of the density, and the degeneracy of spatial dissipation is characterized by the power of the temperature, thus one needs to find a uniform quantity to evaluate the degeneracy. According to Boyle and Gay-Lussac law, 
\[
\theta=R^{-1}Ae^{S/c_v}\rho^{\gamma-1},
\]
thus, one can rewrite the momentum equation with density and entropy, and the viscosity stress tensor term can be divided into two terms: one term with derivative on the viscosity coefficients, and one term with the derivative on the entropy. Then if we divide the momentum equation by $\rho$, the second term will become a more singular term than the isentropic case, and this makes the study on the non-isentropic case harder than the isentropic case. 
Duan-Xin-Zhu \cites{DXZ1} 
 investigate the case with zero heat conductivity ($\kappa=0$), and establish the local-in-time existence of classical solutions in $\mathbb{R}^3$, where the solution has conserved total mass, momentum and energy, moreover, $S$ maintains uniform upper and lower bounds throughout the spatial domain.  Recently,
 Duan-Xin-Zhu \cites{DXZ2} extend the result to the case with positive heat conductivity ($\kappa>0$). 
 However, due to the difference of interpolating estimates and  singular-weighted estimates in $\mathbb{R}^3$ and $\mathbb{R}^2$, the theory in Duan-Xin-Zhu \cites{DXZ2} for $\mathbb{R}^3$ can not be directly apply to the $\mathbb{R}^2$ cases.
This paper aims at address the local-in-time existence of classical solutions to the Cauchy problem of \ef{1} with degenerate viscosity and far-field vacuum in $\mathbb{R}^2$,
where 
\begin{equation}\label{4}
	\mu(\theta)=\alpha\theta^\nu,\quad \lambda(\theta)=\beta\theta^\nu,\quad \kappa=\hbar \theta^{\nu}
\end{equation}
with constants $\alpha,\beta,\hbar, \nu$ satisfying
\begin{equation}\label{5}
	\alpha>0,\quad  \alpha+\beta\geq 0,\quad \hbar>0, \quad \text{and} \quad 0<\delta=(\gamma-1)\nu <1.
\end{equation}
Then  according to \eqref{2} and \eqref{4}, one can rewrite    \eqref{1} to the system of $(\rho,u,S)$ as
\begin{equation}\label{8}\left\{\begin{aligned}
		\displaystyle
		&\rho_t+\text{div}(\rho u)=0,\\[1pt]
		\displaystyle
		&(\rho u)_t+\text{div}(\rho u\otimes u)+\nabla P
		=A^\nu R^{-\nu}\text{div}\big(\rho^{\delta}e^{\frac{S}{c_v}\nu}Q(u)\big),\\[1pt]
		\displaystyle
		&P \big(S_t+u\cdot \nabla S\big)- \digamma \rho^{\delta+\gamma-1} e^{\frac{S}{c_v}\nu}\Delta e^{\frac{S}{c_v}}\\
		&\quad=A^{\nu}R^{1-\nu}\rho^{\delta}e^{\frac{S}{c_v}\nu}H(u)+\digamma \rho^{\delta}e^{\frac{S}{c_v}(\nu+1)} \Delta \rho^{\gamma-1}+\Lambda(\rho, S),
	\end{aligned}\right.\end{equation}
where $\digamma=R\hbar \big(\frac{A}{R}\big)^{\nu+1}$, and 
\begin{equation}\label{QH}
	\begin{split}
		Q(u)=&\alpha\big(\nabla u+(\nabla u)^\top\big)+\beta\text{div}u\mathbb{I}_2,\quad H(u)=2\alpha |D(u)|^2+\beta (\text{div}u)^2,\\
		\Lambda(\rho, S)=&2 \digamma \rho^{\delta} e^{\frac{S}{c_v}\nu}\nabla \rho^{\gamma-1}\cdot \nabla e^{\frac{S}{c_v}} +\digamma \nabla( \rho^{\delta} e^{\frac{S}{c_v}\nu})\cdot \nabla( e^{\frac{S}{c_v}} \rho^{\gamma-1} ).
	\end{split}
\end{equation}
We  will establish the local-in-time existence of $(\rho,u,S)$  to the Cauchy problem    $\eqref{8}$  with \eqref{2}, 
\eqref{QH},  and  
\begin{align}
	(\rho,u,S)|_{t=0}=\big(\rho_0(x)>0,u_0(x),S_0(x)\big)\quad  &\text{for} \quad x\in\mathbb{R}^2, \label{6}\\[2pt]
	(\rho,u,S)(t,x)\rightarrow(0,0,\bar{S})\quad  \text{as} \quad |x|\rightarrow\infty\quad &\text{for}\quad t\geq 0,  \label{7}
\end{align}
where $\bar{S}$ is some constant. 

In this paper, we adopt the following simplified notations for the standard homogeneous and inhomogeneous Sobolev spaces(see Galdi \cite{gandi}):
\begin{equation*}\begin{split}
		& |f|_p=\|f\|_{L^p(\mathbb{R}^2)},\quad  \|f\|_s=\|f\|_{H^s(\mathbb{R}^2)}, \quad \|f\|_{m,p}=\|f\|_{W^{m,p}(\mathbb{R}^2)},\\[1pt]   
		& D^{k,r}=\{f\in L^1_{loc}(\mathbb{R}^2): |f|_{D^{k,r}}=|\nabla^kf|_{r}<\infty\},\quad   |f|_{D^{k,r}}=\|f\|_{D^{k,r}(\mathbb{R}^2)},   \\[1pt]
		& D^k=D^{k,2},\quad \int  f   =\int_{\mathbb{R}^2}  f \text{d}x,\quad X([0,T]; Y)=X\big([0,T]; Y(\mathbb{R}^2)\big), \\[1pt]
		&  \|f\|_{X_1 \cap X_2}=\|f\|_{X_1}+\|f\|_{X_2},\quad \|(f,g)\|_X=\|f\|_{X}+\|g\|_{X}.
	\end{split}
\end{equation*}
The \text{Lam$\acute{\text{e}}$} operator $L$ acting on $u$ is defined as:
$$Lu\triangleq-\alpha \Delta u -(\alpha+\beta) \nabla \text{div} u.$$
  
Before stating the main results,  we first introduce a proper class of solutions called regular solutions to the Cauchy problem \eqref{8} with \eqref{2} and \eqref{QH}-\eqref{7}.
\begin{definition}\label{def21}  Let $T>0$ be a finite constant. $(\rho,u,S)$  is called a regular solution to the Cauchy problem \eqref{8} with \eqref{2} and \eqref{QH}-\eqref{7}   in $[0,T]\times\mathbb{R}^2$ if
	\begin{enumerate}
    \item   $(\rho,u,S)$ satisfies this problem in the sense of distributions; 
		\item $\rho>0,\ \  \rho^{\gamma-1}\in C([0,T]; L^p\cap D^1\cap D^3), \ \ \nabla\rho^{\delta-1}\in C([0,T]; L^q\cap D^1 ),\\
		\quad    \rho^{\frac{\delta-1}{2}} \nabla^3\rho^{\delta-1}\in C([0,T];L^2)$;\\
		\item $ u\in C([0,T];H^3)\cap L^2([0,T];H^4),\quad  \rho^{\delta-1}\nabla u \in C([0,T];L^2),$\\
		\quad $    \rho^{\frac{3(\delta-1)}{2}}\nabla^2 u\in  L^\infty([0,T];H^1),\quad \rho^{\delta-1}\nabla^4 u \in L^2([0,T];L^2)  $;\\
		\item $S-\bar{S}\in C([0,T];H^3),\quad  e^{\frac{S}{c_v}}-e^{\frac{\bar{S}}{c_v}}\in C([0,T];H^3)$,
	\end{enumerate}
	for some constants $p,q\in (2,\infty)$.
\end{definition}

Now, the main result in this paper can be stated as follows.
\begin{theorem} \label{th21} Assume that  
	\begin{equation}\label{can1}
		\gamma>1,\quad 0<\delta=(\gamma-1)\nu <1, \quad \gamma+\delta \leq 2, \quad\alpha>0,\quad \alpha+\beta\geq 0.
	\end{equation}
If $(\rho_0,u_0,S_0)$ satisfy
	\begin{equation}\label{2.7}\begin{aligned}
			&\rho_0>0,\quad  \rho_0^{\gamma-1}\in  L^p\cap D^1\cap D^3,\quad \nabla\rho_0^{\delta-1}\in L^q\cap D^{1},\quad \rho_0^\frac{\delta-1}{2}\nabla^3\rho_0^{\delta-1}\in L^2,\\
			& \nabla \rho_0^{\frac{(\delta-1)}{2}}\in L^4,  
   \quad  \nabla \rho_0^{\frac{2\gamma+\delta-3}{2}}\in L^2,\quad u_0\in H^3,\quad S_0-\bar{S}\in  H^3,
		\end{aligned}
	\end{equation}
	for some  $p, q\in (2,\infty)$	and the   compatibility conditions:
	\begin{equation}\label{2.8}\begin{aligned}
			u_0=\rho_0^{\frac{1-\delta}{2}}g_1, \ \nabla u_0=\rho_0^{1-\delta}g_2,\  Lu_0=\rho_0^{\frac{3(1-\delta)}{2}}g_3, \ \nabla (\rho^{\frac{3(1-\delta)}{2}}Lu_0)=\rho_0^{\frac{1-\delta}{2}} g_4,\\
			\nabla e^{\frac{S_0}{c_v}}_0=\rho_0^{\frac{1-\delta}{2}}g_5,\ \  \Delta e^{\frac{S_0}{c_v}}_0=\rho_0^{1-\delta} g_6, \ \  \nabla (\rho^{\delta-1}_0 \Delta e_0^{\frac{S_0}{c_v}})=\rho_0^{\frac{1-\delta}{2}} g_7,\\
	\end{aligned}\end{equation}
	for  $g_i\in L^2(i=1,\cdots,7)$, then the Cauchy problem \eqref{8} with \eqref{2} and \eqref{QH}-\eqref{7} admits  a unique strong solution $(\rho,u,S)$ in $[0,T_*]\times \mathbb{R}^2$  for some  time $T_*>0$ with 
\begin{equation}\label{2.9}
\begin{aligned}
& \nabla \rho^{\delta-1} \in C\left(\left[0, T_*\right] ; L^q\cap D^1 \cap D^2\right), \quad \rho_t^{\gamma-1} \in C\left(\left[0, T_*\right] ; H^2\right),\\
&\rho_{t t}^{\gamma-1} \in C\left(\left[0, T_*\right] ; L^2\right)\cap L^2(\left[0, T_*\right] ; D^1),\quad \nabla \rho_t^{\delta-1} \in C\left(\left[0, T_*\right] ; H^1\right),\\
&\nabla \rho_{t t}^{\delta-1} \in L^2\left(\left[0, T_*\right] ; L^2\right), \quad
  u_t \in C\left(\left[0, T_*\right] ; H^1\right)\cap L^2\left(\left[0, T_*\right] ; D^2\right), \\ &\rho^{\delta-1} \nabla u_t\in L^\infty([0,T_*];L^2), \quad (\rho^{\delta-1} \nabla^2 u_t,u_{tt})\in L^2([0,T_*];L^2),\\
  &t^{\frac{1}{2}} u_t \in L^{\infty}\left(\left[0, T_*\right] ; D^2\right) \cap L^2\left(\left[0, T_*\right] ; D^3\right), \quad t^{\frac{1}{2}} \rho^{\delta-1} \nabla^2 u_t\in L^{\infty}\left(\left[0, T_*\right] ; L^2\right)  \\&( t^{\frac{1}{2}} \rho^{\delta-1} \nabla^3 u_t,  t^{\frac{1}{2}} \rho^{\frac{\delta-1}{2}} \nabla u_{t t}) \in L^2\left(\left[0, T_*\right] ; L^2\right), \\
& t^{\frac{1}{2}} u_{t t} \in L^{\infty}\left(\left[0, T_*\right] ; L^2\right) \cap L^2\left(\left[0, T_*\right] ; D^1\right), \\
& ( S-\bar{S},\rho^{\frac{\delta-1}{2}} \nabla S,\rho^{\delta-1} \nabla^2 S, \rho^{\delta-1} \nabla^3 S) \in L^{\infty}\left(\left[0, T_*\right] ; L^2\right), \\
& S_t \in C\left(\left[0, T_*\right] ; H^1\right) \cap L^2\left(\left[0, T_*\right] ; D^2\right), \quad \rho^{\frac{\delta-1}{2}} S_t \in L^{\infty}\left(\left[0, T_*\right] ; D^1\right), \\
& \rho^{\frac{\delta-1}{2}} S_t \in L^2\left(\left[0, T_*\right] ; D^2\right), \quad S_{t t} \in L^2\left(\left[0, T_*\right] ; L^2\right)\\&
(t^{\frac{1}{2}} \rho^{\frac{\delta-1}{2}} \nabla^2 S_t, t^{\frac{1}{2}} \rho^{\frac{1-\delta}{4}} S_{t t}) \in L^{\infty}\left(\left[0, T_*\right] ; L^2\right),
\; t^{\frac{1}{2}} \rho^{\frac{\delta-1}{4}} \nabla S_{t t}\in L^2\left(\left[0, T_*\right] ; L^2\right).
\end{aligned}
\end{equation}
 %
%
%
%
%
	%
	Moreover,    $(\rho,u,S)$ is a classical solution to  the  problem \eqref{8} with \eqref{2} and \eqref{QH}-\eqref{7}  and  $\left(\rho, u, \theta=A R^{-1} \rho^{\gamma-1} e^{S / c_v}\right)$ is a classical solution to the problem  \eqref{1}-\eqref{3} with \eqref{4} and \eqref{6}-\eqref{7} in $\left(0, T_*\right] \times \mathbb{R}^2$.

\end{theorem}

\begin{remark} $\eqref{2.7}$-$\eqref{2.8}$ identifies a class of admissible initial data that provide unique solvability to  the problem \eqref{8} with \eqref{2} and \eqref{QH}-\eqref{7}. Such initial  data include
	\begin{equation*}
		\rho_0(x)=\frac{1}{(1+|x|^2)^{\varkappa}},\quad u_0(x)\in C^3_0(\mathbb{R}^2),\quad S_0=\bar{S}+f(x),
	\end{equation*}
	for some $f(x)\in H^3$, where $\fr{3+\delta}{2}<\gamma+\delta\leq 2$, 
	\begin{equation*}
		\hspace{2mm}\frac{1}{p(\gamma-1)}<\varkappa<\frac{1-2/q}{2(1-\delta)}.
	\end{equation*}
\end{remark}

\begin{remark}\label{canshu}
	We make some comments on \ef{can1} for the physical parameters: 
	
	First,  $\gamma+\delta\le 2$ is necessary to obtain the estimates of the density-related variables like $n=\rho^{2-\gamma-\delta}$. Moreover, the condition $\gamma+\delta\le 2$ allows more physical models compared to the condition $4\gamma+3\delta\le 7$ in $\mathbb{R}^3$.
	
	Second, $0<\delta=\nu(\gamma-1)<1$ includes many physical fluids, for example,  taking $p=6$, then the example of the initial data includes the air ($\gamma=1.4,\ \nu=0.72$ or $\nu=0.5$ for high temperature), the carbon dioxide ($\gamma=\fr43,\ \nu=0.95$), the hydrogen ($\gamma=1.4,\ \nu=0.69$) etc. (See Chapter 2, Section 2.1, pages 128-129 of \cite{tlt}  and Chapter 19, pages 638-639 of \cite{wuli}).  
\end{remark}

\begin{remark}\label{rxiangrong}
	We remark that \eqref{2.8} are required in proving Theorem \ref{th21}, especially in deducing the singular-weighted estimates of the velocity and the entropy.
	In particular,
	in order to show that
	\[
	\rho^{\frac{\delta-1}{2}}u,\  \rho^{\delta-1} \nabla u, \ \rho^{\frac{\delta-1}{2}}u_t, \ \rho^{\delta-1} \nabla u_t ,\ \rho^{\frac{\delta-1}{2}}\nabla e^{\frac{S}{c_v}} , \ (e^{\frac{S}{c_v}})_t,\ \rho^{\frac{\delta-1}{2}}\nabla ( e^{\frac{S}{c_v}} )_t
	\]
	and then
	\[
	\rho^{\frac{3(1-\delta)}{2}}\nabla^2u, \ \ \rho^{\frac{3(1-\delta)}{2}}\nabla^3u,\ \  \rho^{\delta-1}\nabla^2 e^{\frac{S}{c_v}}, \ \ \rho^{\delta-1}\nabla^3 e^{\frac{S}{c_v}}
	\]
	are all belong to $ L^\infty([0,T];L^2)$ for some constant $T>0$, we need the compatibility conditions \eqref{2.8}.
\end{remark}

In rest of this paper, we will introduce some basic facts in \S 2; and then the reformulation and main strategy for the Cauchy problem \eqref{8} with \eqref{2} and \eqref{QH}-\eqref{7} will be given in \S 3; the linearization of the reformulated problem \eqref{2.3} is stated in \S 4; the uniform a priori estimates of the solutions to the linearized Cauchy problem \eqref{ln} is established in \S 5; at last, we will prove the local-in-time existence of the Cauchy problem \eqref{8} with \eqref{2} and \eqref{QH}-\eqref{7} in \S 6.

\section{Preliminaries }


Now, we list some basic facts that will be frequently used in proving the local-in-time existence.
			The first one is the Gagliardo-Nirenberg inequality in $\mathbb{R}^2$.
			\begin{lemma}\cite{oar}\label{lem2as}\
				Let $f\in L^{q_1}\cap D^{1,r}(\mathbb{R}^2)$ for $1 \leq q_1,  r \leq \infty$.  Suppose   that  real numbers $\xi$ and $q_2$,  and  natural numbers $m$, $i$  and $j$ satisfy
				$$\frac{1}{{q_2}} = \frac{j}{d} + \left( \frac{1}{r} - \frac{i}{2} \right) \xi + \frac{1 - \xi}{q_1} \quad \text{and} \quad 
				\frac{j}{i} \leq \xi \leq 1.
				$$
				Then $u\in D^{j,{q_2}}(\mathbb{R}^2)$, and  there exists a constant $C$ depending only on $i$, $j$, $q_1$, $r$ and $\xi$ such that
				\begin{equation}\label{33}
					\begin{split}
						\| \nabla^{j} f \|_{L^{{q_2}}} \leq C \| \nabla^{i} f \|_{L^{r}}^{\xi} \| f \|_{L^{q_1}}^{1 - \xi}.
					\end{split}
				\end{equation}
				Moreover, if $j = 0$, $ir < 2$ and $q_1 = \infty$, then it is necessary to make the additional assumption that either $f$ tends to zero at infinity or that $f$ lies in $L^s(\mathbb{R}^2)$ for some finite $s > 0$;
				if $1 < r < \infty$ and $i -j -2/r$ is a non-negative integer, then it is necessary to assume also that $\xi \neq 1$.
				
				More specifically, in $\mathbb{R}^2$, it holds that
				\begin{equation*}
					\begin{split}
						|f|_3\le &C|\nabla f|_2^{\frac13} |f|_2^{\frac23},\quad |f|_4\le C |\nabla f|_2^{\frac12} |f|_2^{\frac12},\quad |f|_6\le C |\nabla f|_2^{\frac23} |f|_2^{\frac13},\\
						|f|_{\infty}\le& C|\nabla f|_3^{\frac34} |f|_2^{\frac14}\le C |\nabla^2 f|_2^{\frac14}  |\nabla f|_2^{\frac12} |f|_2^{\frac14},\quad   |f|_{\infty}\le C|\nabla f|_3^{\frac23} |f|_3^{\frac13},\\
						|f|_\infty\le& C|\nabla^2 f|_2^{\frac13} |f|_4^{\frac23}, \quad  |f|_2 \le  C|\nabla f|_1.
					\end{split}
				\end{equation*}
			\end{lemma}


			The second one gives some compactness results. 
			\begin{lemma}\cite{jm}\label{aubin} Let $X_0\subset X\subset X_1$ be three Banach spaces.  Suppose that $X_0$ is compactly embedded in $X$ and $X$ is continuously embedded in $X_1$. Then it holds that				
				\begin{enumerate}
					\item[(1)] If $f$ is bounded in $L^r([0,T];X_0)$ for $1\leq r < +\infty$, and $\frac{\partial f}{\partial t}$ is bounded in $L^1([0,T];X_1)$, then $f$ is relatively compact in $L^r([0,T];X)$;
					\item[(2)] If $f$ is bounded in $L^\infty([0,T];X_0)$  and $\frac{\partial f}{\partial t}$ is bounded in $L^r([0,T];X_1)$ for $r>1$, then $f$ is relatively compact in $C([0,T];X)$.
				\end{enumerate}
			\end{lemma}
			
			
			The third is used to obtain the time-weighted estimates of $u$.
			\begin{lemma}\cite{bjr}\label{bjr}
				If $f(t,x)\in L^2([0,T]; L^2)$, then there exists a sequence $s_k$ such that
				$$
				s_k\rightarrow 0 \quad \text{and}\quad s_k |f(s_k,x)|^2_2\rightarrow 0 \quad \text{as} \quad k\rightarrow\infty.
				$$
			\end{lemma}


			At last, one has the following regularity theory for  \text{Lam$\acute{\text{e}}$}  problem:
			\begin{equation}\label{elleq}
				-\alpha \Delta u -(\alpha+\beta) \nabla \text{div} u =f, \quad u\to 0\quad \text{as}\quad |x|\to \infty.
			\end{equation}
			
			\begin{lemma}\label{ell}
				\cite{lame} If $u\in D^{1,r} (\mathbb{R}^2)$ with $1<r<\infty$ is a weak solution to \eqref{elleq}, then 
				\[
				|u|_{D^{k+2,r}}\le C |f|_{D^{k,r}},
				\]
				where $C$ depends only on $\alpha, \beta$ and $r$.
			\end{lemma}

			\section{Reformulation and main strategy}
            
            \subsection{Reformulation}
            In order to prove Theorems \ref{th21}, we apply the density and entropy related functions 
			\begin{equation}\label{2.1}\begin{aligned}
					&\phi=\frac{A\gamma}{\gamma-1}\rho^{\gamma-1}, \quad l=e^{\frac{S}{c_v}}
			\end{aligned}\end{equation}
		to reformulate the Cauchy problem \eqref{8} with \eqref{2} and \eqref{QH}-\eqref{7} as 
			\begin{equation}\label{2.3}\left\{\begin{aligned}
					\displaystyle
					&  \phi_t+u\cdot\nabla\phi+(\gamma-1)\phi \text{div}u=0,\\[5pt]
					\displaystyle
					&  u_t+u\cdot \nabla u+a_1\phi\nabla l+l\nabla\phi+a_2l^\nu\phi^{2\iota} Lu
					=a_2\phi^{2\iota}\nabla l^\nu\cdot Q(u)+a_3l^\nu\psi  \cdot  Q(u),\\[5pt]
					\displaystyle
					& l_t+u\cdot\nabla l-a_4\phi^{2\iota}l^\nu\Delta l =a_5l^\nu n\phi^{4\iota}H(u)+a_6l^{\nu+1}\text{div}\psi+\Theta(\phi,l,\psi),\\
                    	&(\phi,u,l)|_{t=0}=(\phi_0,u_0,l_0)
					=\Big(\frac{A\gamma}{\gamma-1}\rho_0^{\gamma-1}(x),u_0(x),e^{\frac{S_0(x)}{c_v}}\Big)\quad  \text{for} \quad x\in\mathbb{R}^2,\\
                    	&(\phi,u,l)\rightarrow (0,0,\bar{l}) \hspace{2mm} \text{as} \hspace{2mm}|x|\rightarrow \infty \quad \text{for} \quad t\geq 0,
				\end{aligned}\right.\end{equation}
			where  $\psi=\frac{\delta}{\delta-1}\nabla\rho^{\delta-1}$,  $n=\rho^{2-\delta-\gamma}$, $$ \Theta(\phi,l,\psi)=a_7l^{\nu+1}\phi^{-2\iota}\psi\cdot\psi+a_8l^\nu\nabla l\cdot\psi+a_9l^{\nu-1}\phi^{2\iota}\nabla l\cdot \nabla l,$$  constants $\bar{l}=e^{\frac{\bar{S}}{c_v}}>0$, $\iota=\frac{\delta-1}{2(\gamma-1)}$, and
				\begin{align*}
					a_1=&\frac{\gamma-1}{\gamma},\ \   a_2=a\Big(\frac{A}{R}\Big)^\nu,\ \ a_3=\Big(\fr{A}{R}\Big)^\nu,\ \    a_4=\digamma\frac{a}{Ac_v},\ \ a_5=\frac{A^{\nu-1}a^2(\gamma-1)}{R^\nu},\\
					 a_6=&\digamma\frac{(\gamma-1)}{Ac_v\delta},\ \  a_7=\digamma\frac{\gamma(\gamma-1)}{aAc_v\delta^2},\ \ 
					a_8=2\digamma\frac{1+\nu}{Ac_v\nu},\ \  a_9=\digamma\frac{a\nu}{Ac_v},\ \ a=\Big(\frac{A\gamma}{\gamma-1}\Big)^{\fr{1-\delta}{\gamma-1}}. 
				\end{align*}

		 
We will first give the definition of regular solution and establish the local-in-time existence to the Cauchy problem \eqref{2.3}.
					\begin{definition}\label{strongsolution}
   $(\phi,u,l)$  is a strong solution to the Cauchy problem \ef{2.3} in $[0,T_*]\times \mathbb{R}^2$, if it satisfies \ef{2.3} in the sense of distributions and meets the equations in \eqref{2.3} a.e. $(t,x)\in (0,T_*]\times \mathbb{R}^2$.
\end{definition}
			
			\begin{theorem}\label{3.1} Under $\ef{can1}$, for some $p,q\in (2,\infty)$ and $\iota=\frac{\delta-1}{2(\gamma-1)}$, if $(\phi_0,u_0,l_0)$ satisfies the regularities:
				\begin{equation}\label{a}\begin{aligned}
						&\phi_0>0,\quad \phi_0\in L^p\cap D^1\cap D^3,\quad \nabla \phi^{2\iota}_0\in L^q\cap D^1,\quad  \phi_0^{2\iota}\nabla^3\phi^{2\iota}_0\in L^2, \\
						& \nabla\phi_0^{\frac{1}{2}\iota}\in L^4, \quad  \nabla \phi_0^{\iota+1}\in L^2,\quad  u_0\in H^3,\quad l_0-\bar{l}\in  H^3 ,\quad \inf_{x\in \mathbb{R}^2} l_0>0,
				\end{aligned}\end{equation}
				 and the compatibility conditions:
				\begin{equation}\label{2.8*}\begin{aligned}
						 u_0=\phi_0^{-\iota}g_1,   \quad \nabla u_0=\phi_0^{-2\iota}g_2,  \quad   L u_0=\phi_0^{-3\iota}g_3, \quad 
						\nabla(\phi_0^{3\iota} Lu_0)=\phi_0^{-\iota}g_4,\\
						\nabla  l_0 =\phi_0^{-\iota}g_5, \quad 
						\Delta l_0=\phi_0^{- 2\iota}g_6, \quad  \nabla(\phi_0^{2\iota}\Delta l_0)=\phi_0^{-\iota}g_7,
				\end{aligned}\end{equation}
				for  some $g_i (i=1, \cdots, 7)\in L^2$, then there exist a time $T_*>0$ and a unique strong solution $(\phi,u,l)$  in $[0,T_*]\times \mathbb{R}^2$   to the Cauchy problem \ef{2.3} with 
				\begin{equation}\label{2.1zzx}
					\begin{aligned}
                   & \phi(t,x)>0\quad  \text{in}\; [0,T_*]\times \mathbb{R}^2,\quad \displaystyle\inf_{(t,x)\in [0,T_*]\times \mathbb{R}^2} l>0,\\
						&\phi \in C\left(\left[0, T_*\right] ; L^p\cap D^1\cap D^3\right), \quad \nabla \phi^{\frac{1}{2} \iota} \in C\left(\left[0, T_*\right] ; L^4\right),\\
						&\psi \in C\left(\left[0, T_*\right] ; L^q \cap D^{1} \cap D^2 \right), \; \phi_t \in C\left(\left[0, T_*\right] ; H^2\right), \; \psi_t \in C\left(\left[0, T_*\right];H^1\right), \\
						&  \phi_{t t} \in C\left(\left[0, T_*\right] ; L^2\right) \cap L^2\left(\left[0, T_*\right] ; D^1\right),\quad \psi_{t t}\in L^2\left(\left[0, T_*\right] ; L^2\right), \\
						& u \in C\left(\left[0, T_*\right] ; H^3\right) \cap L^2\left(\left[0, T_*\right] ; H^4\right),\quad l-\bar{l} \in C\left(\left[0, T_*\right] ; H^3\right),\\
						&(u_t, l_t) \in C\left(\left[0, T_*\right] ; H^1\right)\cap L^2\left(\left[0, T_*\right] ; D^2\right), \quad \phi^{3 \iota} \nabla^2 u \in L^{\infty}\left(\left[0, T_*\right] ; H^1\right) ,\\
						&\left(\phi^{2\iota} \nabla u, t^{\frac{1}{2}} \phi^{2 \iota} \nabla^4 u, \phi^{2\iota} \nabla u_t, t^{\frac{1}{2}} \phi^{2 \iota} \nabla^2 u_t\right) \in L^{\infty}\left(\left[0, T_*\right] ; L^2\right),\\
						&\left(\phi^{2 \iota} \nabla^4 u,\phi^{2 \iota} \nabla^2 u_t, t^{\frac{1}{2}} \phi^{2 \iota} \nabla^3 u_t, u_{t t}, t^{\frac{1}{2}} \phi^\iota \nabla u_{t t},l_{t t}, t^{\frac{1}{2}} \phi^{\frac{\iota}{2}} \nabla l_{t t}\right) \in L^2\left(\left[0, T_*\right] ; L^2\right),\\
						&t^{\frac{1}{2}} u_{t t} \in L^{\infty}\left(\left[0, T_*\right] ; L^2\right) \cap L^2\left(\left[0, T_*\right] ; D^1\right), \\
						&\left(l-\bar{l},\phi^{\iota} \nabla l, \phi^{2\iota} \nabla^2 l, \phi^{2\iota} \nabla^3 l,  l_t, t^{\frac{1}{2}} \phi^\iota \nabla^2 l_t, t^{\frac{1}{2}} \phi^{-\frac{\iota}{2}} l_{t t}\right) \in L^{\infty}\left(\left[0, T_*\right] ; L^2\right),\\
						& \phi^{\iota} l_t \in L^{\infty}\left(\left[0, T_*\right] ; D^1\right)\cap L^2\left(\left[0, T_*\right] ; D^2\right).
					\end{aligned}
				\end{equation}
			\end{theorem}
		Then based on Theorem 3.1, we can prove the local well-posedness to the Cauchy problem \eqref{8} with \eqref{2} and \eqref{QH}-\eqref{7} immediately.

	\par\bigskip
    \subsection{Main strategy}
	%
			First, via establishing carefully designed singular-weighted estimates,  \cite{DXZ2}  proved the local well-posedness of regular solutions to the corresponding Cauchy problem in $\mathbb{R}^3$. However, 
 the difference in the Sobolev embedding inequalities between $\mathbb{R}^2$ and $\mathbb{R}^3$ makes it hard to apply the same renormalization and singular-weighted estimates in $\mathbb{R}^3$ to establish the well-posedness theory in $\mathbb{R}^2$.
   For example, one has 
   $$|\phi^{\iota} f|_6\le |\nabla (\phi^{\iota}f)|_2\le C(|\nabla \phi^{\iota}|_\infty  |f|_2+|\phi^{\iota}\nabla f|_2) \quad \text{in}\quad \mathbb{R}^3
   $$
   for some function $f$, where $\phi^{\iota}$ is a quantity of negative power of $\rho$, 
   i.e. $|f|_2, |\nabla \phi^{\iota}|_\infty$ and $|\phi^{\iota}\nabla f|_2$ are suffice to provide the singular-weighted estimate of $|\phi^{\iota} f|_6$, which plays key role in establishing uniform the a prior estimates, for example, in establishing the singular-weighted estimates of $l$, one has
   \begin{equation*}
       \begin{split}
           \frac{a_4}{2}\frac{\text{d}}{\text{d}t}|h^{\frac14}\nabla l_t|_2^2+|w^{-\frac{\nu}{2}} h^{-\frac14}l_{tt}|_2^2=&\underbrace{a_5\int   n g^{\frac32} H(v)_t  l_{tt}}_{\text{trouble term}}+\cdots\\
           \le &C |w^{\frac{\nu}{2}} n|_\infty |g\nabla v|_6 |g^{\frac34}\nabla v_t|_3 |w^{-\frac{\nu}{2}} h^{-\frac14}l_{tt}|_2+\cdots
       \end{split}
   \end{equation*}
   and
      \begin{equation*}
       \begin{split}
           \frac{1}{2}\frac{\text{d}}{\text{d}t}|w^{-\frac{\nu}{2}} h^{-\frac14}l_{tt}|_2^2+a_4|h^{\frac14}\nabla l_{tt}|_2^2=&\underbrace{a_5\int   n g^{\frac32} H(v)_{tt}   l_{tt}}_{\text{trouble term}}+\cdots\\
          \le &C | n|_\infty |g\nabla v|_6 |g^{\frac14} v_{tt}|_3 | h^{\frac14}\nabla l_{tt}|_2+\cdots
       \end{split}
   \end{equation*}
However, one has
  $$|\phi^{\frac{\iota}{2}} f|_4\le C\big(|\nabla \phi^{\iota}|^\frac12_\infty  |f|_2+|\phi^{\iota}\nabla f|_2^\frac12|f|_2^\frac12\big)\quad \text{in}\quad \mathbb{R}^2,$$
which implies that, with $|f|_2$, $\nabla \phi^{\iota}|_\infty$ and $|\phi^{\iota}\nabla f|_2$ at hand, 
 $f$ can bear more singular weight $\phi^{\iota}$ in $\mathbb{R}^3$, but less singular weight $\phi^{\frac{\iota}{2}}$ in $\mathbb{R}^2$, 
and this caused essential difficulties in deducing closed singular-weighted estimates, especially the estimates of $l$ and $u$, and this is the reason why we need to find a different singular-weighted estimates in two dimensional space. See \S \ref{apr} for more details. 

 To compensate for the distance in absorbing the singular weight caused by the dimension difference, we find that under a different linearized system, if we let $\phi$ carry some singular weight, i.e. establish the singular-weighted estimates of $\phi$, then $u$ and $l$ can carry more singular weight of $\rho$ than the $\mathbb{R}^3$ case, which helps to establish the well-posedness in $\mathbb{R}^2$. We remark that the estimates in \cite{DXZ2} do not provide the singular-weighted estimate of $\phi$, since one first needs to establish the equivalence between $h(=\phi^{2\iota})$ and its linearized counterpart $g$, while this equivalence needs in turn the singular-weighted estimates of $\phi$. Here, we introduce the new quantity  $h^{-1}\psi$, which satisfies 
$$(h^{-1}\psi )_t +   v\cdot \nabla ( h^{-1}\psi) +  (\nabla v)^{\top}  h^{-1}\psi   +a\delta \nabla \text{div}v=0,$$
and first deduce the $L^{\infty}$-estimates of $h^{-1}\psi$, then obtain the singular-weighted  estimate of $ \phi$ without using the equivalence of $h$ and $g$, see Lemma \ref{gh} for details.

		At last, 
in order to establish the $L^2$ estimate of $l-\bar{l}$, i.e. the $L^2$ estimate of $S$, which is essential for demonstrating that the regular solution is indeed a classical one, the main difficulty is to deal with the singular term $a_5l^\nu n\phi^{4\iota}H(u)$ in $\eqref{2.3}_3$.  
 Here, we introduce a new quantity $\phi^{2\iota}\nabla u$ and establish its estimates in $L^2([0,T]; L^2)$ by using the singular parabolic equations of $u$, which help to deal with difficulties due to the difference of embedding inequalities in $\mathbb{R}^2$ and $\mathbb{R}^3$. See \S \ref{ells} for more details.
             
			\section{Linearization}
			 This section states the linearized problem of   \eqref{2.3} in $[0,T]\times \mathbb{R}^2$ for some  constants $T>0$ and  $\epsilon>0$:
			\begin{equation}\label{ln}\left\{\begin{aligned}
					&\phi^{\epsilon}_t+v\cdot\nabla\phi^{\epsilon}+(\gamma-1)\phi^{\epsilon} \text{div}v=0,\\[2pt]
					&u^{\epsilon}_t+v\cdot \nabla v+a_1\phi^{\epsilon}\nabla l^{\epsilon}+l^{\epsilon}\nabla\phi^{\epsilon}+a_2(l^{\epsilon})^\nu  h^{\epsilon}  Lu^{\epsilon}\\
					&\qquad \qquad\quad
                     =a_2 g\nabla (l^{\epsilon})^\nu\cdot Q(v)+a_3(l^{\epsilon})^\nu \psi^{\epsilon} \cdot Q(v),\\[2pt]
					&l^{\epsilon}_t+v\cdot\nabla l^{\epsilon}-a_4w^\nu h^{\epsilon}\Delta l^{\epsilon}=a_5w^\nu n^{\epsilon} g^2H(v)+a_6w^{\nu+1}\text{div}\psi^{\epsilon}+\Pi(l^{\epsilon},h^{\epsilon},w,g),\\[2pt]
					&(\phi^{\epsilon},u^{\epsilon},l^{\epsilon})|_{t=0}=(\phi^\epsilon_0,u^\epsilon_0,l^\epsilon_0)
					=(\phi_0+\epsilon,u_0,l_0 )\quad  \text{for} \quad x\in\mathbb{R}^2,\\[2pt]
					&(\phi^{\epsilon},u^{\epsilon},l^{\epsilon} )\rightarrow (\epsilon,0,\bar{l}) \quad \text{as} \hspace{2mm}|x|\rightarrow \infty \quad {\rm for}\quad t\geq 0,
				\end{aligned}\right.\end{equation}
			where $h^{\epsilon}=(\phi^{\epsilon})^{2\iota}$, $\psi^{\epsilon}=\fr{a\delta}{\delta-1}\nabla h^{\epsilon}$,$n^{\epsilon}=(ah^{\epsilon})^b$, $b=\fr{2-\delta-\gamma}{\delta-1}\le 0$, and $$
					\Pi(l^{\epsilon},h^{\epsilon},w,g)=a_7w^{\nu+1}(h^{\epsilon})^{-1}\psi^{\epsilon}\cdot\psi^{\epsilon}+a_8w^\nu\nabla l^{\epsilon}\cdot\psi^{\epsilon}+a_9w^{\nu-1}g\nabla w\cdot\nabla w.$$
			Here, $v =(v^{(1)}, v^{(2)})^\top\in \mathbb{R}^2$, $g$ and $w$ are known  functions satisfying $w> 0$, $(v,g,w)(0,x)=(u_0,h_0=(\phi^\epsilon_0)^{2\iota},l_0)(x)$, and
			\begin{equation}\label{4.1*}
				\begin{split}
					&g\in L^\infty\cap C([0,T]\times \mathbb{R}^2),\ \ \nabla g\in C([0,T];L^q\cap D^{1}),  \\[3pt]
					& g^{\frac12}\nabla^3 g\in C([0,T];L^2),\ \  \nabla g^\frac12\in C([0,T];L^4),\quad g_t\in C([0,T];H^2),\\[3pt]
					&  (\nabla g_{tt},v_{tt},w_{tt})\in L^2([0,T];L^2),\quad v\in C([0,T];H^3)\cap L^2([0,T];H^4),\\[3pt]
					& t^{\fr{1}{2}}v\in L^\infty([0,T];D^4),\quad
					v_t\in C([0,T];H^1)\cap L^2([0,T];D^2),\\[3pt]
					&  t^{\fr{1}{2}}v_t\in L^\infty([0,T];D^2)\cap L^2([0,T];D^3),\quad w-\bar{l}\in C([0,T];H^3),\\[3pt]
					& t^{\fr{1}{2}}(v_{tt},w_{tt})\in L^\infty([0,T];L^2)\cap L^2([0,T];D^1),\quad \inf_{(t,x)\in[0,T]\times\mathbb{R}^2} w>0,\\[3pt]
					& w_t\in C([0,T];H^1)\cap L^2([0,T];D^2),\quad t^\frac12 w_{t}\in  L^\infty([0,T];D^2).
				\end{split}
			\end{equation}
			Moreover, we claim that $(\phi_0,u_0,l_0)$ and  $\bar{l}=e^{\frac{\bar{S}}{c_v}}>0$ in \eqref{ln} is same as the ones in  $\eqref{2.3}$.
			%
			Then according to the conventional theory \cites{KA,oar}, one has the global well-posedness to \ef{ln} in $[0,T]\times \mathbb{R}^2$ for all $\epsilon>0$.
			\begin{lemma}\label{ls}
				Under $\ef{can1}$, if $\epsilon>0$ and $(\phi_0,u_0,l_0)$ satisfy 
				\eqref{a}-\eqref{2.8*}. Then $\ef{ln}$ has a unique classical solution $(\phi^{\epsilon},u^{\epsilon},l^{\epsilon})$ in $[0,T]\times \mathbb{R}^2$ for all $T>0$ with
				\begin{equation}\label{2.13}\begin{aligned}
						&\phi^{\epsilon}-\epsilon\in C([0,T]; L^p\cap D^1\cap D^3), \quad  \phi^{\epsilon}_t\in C([0,T];H^2), \\[3pt]
						& h^{\epsilon}\in L^\infty\cap C([0,T]\times \mathbb{R}^2),\quad  \nabla h^{\epsilon}\in C([0,T];H^2),\\[3pt]
						& h^{\epsilon}_t\in C([0,T];H^2),\quad 	u^{\epsilon}\in C([0,T];H^3)\cap L^2([0,T];H^4),\\[3pt]
						& u^{\epsilon}_t\in C([0,T];H^1)\cap L^2([0,T];D^2),\quad  u^{\epsilon}_{tt}\in L^2([0,T];L^2),\\[3pt]
						& t^{\fr{1}{2}}u^{\epsilon}\in L^\infty([0,T];D^4), \quad t^{\fr{1}{2}}u^{\epsilon}_t\in L^\infty([0,T];D^2)\cap L^2([0,T];D^3),\\[3pt]
						& t^{\fr{1}{2}}u^{\epsilon}_{tt}\in L^\infty([0,T];L^2)\cap L^2([0,T];D^1),\quad  l^{\epsilon}-\bar{l}\in C([0,T]; H^3),\\[3pt]
						& l^{\epsilon}_t\in  C([0,T];H^{1})\cap L^2([0,T];D^2),\quad l^{\epsilon}_{tt}\in L^2([0,T];L^2),\\[3pt]&t^\frac12l_t^{\epsilon}\in L^\infty([0,T];D^2),\quad t^\frac12l_{tt}^{\epsilon}\in L^\infty([0,T];L^2)\cap L^2([0,T];D^1).
				\end{aligned}\end{equation}
			\end{lemma}
			

			\section{Uniform a priori estimates}\label{apr}  
			This section aims at establishing the uniform  a priori estimates of the solution $(\phi^{\epsilon},u^{\epsilon},l^{\epsilon})$, 
			which is independent of $ \epsilon$.
			\par\bigskip
            
			First, for any fixed constant $\epsilon\in(0,1]$, as 
$(\phi^\epsilon_0,u^\epsilon_0,l^\epsilon_0,h^\epsilon_0)=\big(\phi_0+\epsilon,u_0,l_0,(\phi_0+\epsilon)^{2\iota}\big)$, and $(\phi_0,u_0,l_0)$  satisfy 
			\eqref{a}-\eqref{2.8*}, one has 
			\begin{equation}\label{2.14}
				\begin{aligned}
					&2+\epsilon+\bar{l}+\|\phi^\epsilon_0-\epsilon\|_{L^p\cap D^1\cap D^2}+\|\nabla h^\epsilon_0\|_{L^q\cap D^{1}}+|\nabla (\phi_0^\epsilon
					)^{\iota+1}|_2\\&+| (h^\epsilon_0)^{\frac{1}{2}}\nabla^3h^\epsilon_0|_2+|\nabla (h^\epsilon_0)^{\frac{1}{2}}|_4
					+|(h^\epsilon_0)^{-1}|_\infty+\|u^\epsilon_0\|_3 +|(l^\epsilon_0)^{-1}|_\infty\\&+\|l^\epsilon_0-\bar{l}\|_{3}+\sum_{i=1}^7 |g^\epsilon_i|_2 \le c_0,
			\end{aligned}\end{equation}
			for some constant $c_0>1$ independent of $\epsilon$,  where 
			\begin{equation*}
				\begin{split}
					 g^\epsilon_1=(\phi^\epsilon_0)^{ \iota} u^\epsilon_0,\quad 
     g^\epsilon_2=(\phi^{\epsilon}_0)^{ 2\iota}\nabla u^\epsilon_0,\quad g_3^\epsilon=(\phi^\epsilon_0)^{3\iota}Lu^\epsilon_0,\quad 
					g_4^\epsilon=(\phi^\epsilon_0)^{\iota}\nabla\big((\phi^\epsilon_0)^{3\iota}  Lu^\epsilon_0\big),\\
					g_5^\epsilon=(\phi^\epsilon_0)^{\iota} \nabla l_0^{\epsilon}, \quad 
					g_6^\epsilon=(\phi^\epsilon_0)^{ 2\iota }\Delta l^\epsilon_0, \quad g_7^\epsilon=(\phi^\epsilon_0)^{\iota}\nabla\big((\phi^\epsilon_0)^{2\iota}  \Delta l^\epsilon_0\big).
				\end{split}
			\end{equation*}
			
Hereinafter, denote $ M(c)\geq 1 (c\ge0)$ a generic  increasing and continuous function,  $C\geq 1$  a generic  constant, which are  differ from  one line to another and only depend on constants $A, R, c_v, \alpha, \beta, \gamma, \delta, T$. 
		Moreover, 
			for simplicity in statement, we still denote
$(\phi^\epsilon_0,u^\epsilon_0,l^\epsilon_0,h^\epsilon_0,\psi^\epsilon_0)$, 
$(\phi^{\epsilon},u^{\epsilon},l^{\epsilon},h^{\epsilon},\psi^{\epsilon})$, and $g_i^{\epsilon}(i=1,\cdots,7)$ 
as the ones without the superscript $\epsilon$, and   $(\phi,u,l)$ is the unique  solution to $\ef{ln}$.

            \par\bigskip

            Now, for $g$, $w$ and $v$, let constants $c_i (i=1,\cdots,5)$ satisfy
			\begin{equation}\label{801}\begin{aligned}
					\sup_{0\leq t\leq T^*}(\|\nabla g\|^2_{L^\infty\cap L^q\cap D^{1}\cap D^2}+|\nabla g^\frac{1}{2}|_4^2)(t)\leq c_1^2,&\\
					\inf_{[0,T^*]\times \mathbb{R}^2} w(t,x)\geq c^{-1}_1,\quad   \inf_{[0, T^{*}]\times \mathbb{R}^2} g(t,x)\ge c_1^{-1},&\\
					\sup_{0\leq t\leq T^*}  (|w|^2_{\infty}+|v|_\infty^2)(t)+\int^{T^*}_0(|v|^2_{D^2}+|v_t|^2_2)\text{d}t\leq c_1^2,&\\
					\sup_{0\leq t\leq T^*}  (\|v\|^2_{1}+|g\nabla v|_2^2+|g^{\frac12}\nabla w|_2^2)(t)+\int_{0}^{T^*}(|w_t|_2^2+|g\nabla^2w|_2^2)\text{d}t\le c_2^2,&\\
					\sup_{0\leq t\leq T^*}(|w_t|_2^2+|g\nabla^2w|_2^2)(t)+\int_{0}^{T^*}(|g^\frac12\nabla w_t|_2^2+|g\nabla^3 w|_2^2)\text{d}t\le c_2^2,&\\
					\sup_{0\leq t\leq T^*}(|g^{\frac12}\nabla w_t|_2^2+|g\nabla^3w|_2^2)(t)+\int_{0}^{T^*}(| w_{tt}|_2^2+|g^\frac12\nabla^2 w_t|_2^2)\text{d}t\le c_2^2,&\\
					\text{ess}\sup_{0\leq t\leq T^*}t(|g^{-\frac14} w_{tt}|_2^2+|g^\frac12\nabla^2 w_t|_2^2)(t)+\int_{0}^{T^*}t|g^\frac14w_{tt}|_{D^1}^2\text{d}t\le c_2^2,&\\
					\sup_{0\leq t\leq T^*}(|v|^2_{D^2}+|v_t|^2_2+|g^\frac32\nabla^2v|^2_2)(t)+\int^{T^*}_0(|v|_{D^3}^2+|v_t|^2_{D^1})\text{d}t\leq c_3^2,&\\
					\sup_{0\leq t\leq T^*}(|v|^2_{D^3}+|v_t|_{D^1}^2+|g\nabla v_t|^2_2+|g_t|^2_{D^1})(t)+\int^{T^*}_0(|v|^2_{D^4}+|v_t|^2_{D^2}+|v_{tt}|^2_2)\text{d}t\leq c_4^2,&\\
					\sup_{0\leq t\leq T^*}(|g^\frac32\nabla^3v|^2_2+|g_t|^2_\infty)(t)+\int^{T^*}_0(|(g\nabla^2v)_t|^2_2+|g\nabla^4v|^2_2)\text{d}t\leq c_4^2,&\\
					\text{ess}\sup_{0\leq t\leq T^*}t(|v|^2_{D^4}+|g\nabla^2v_t|^2_2)(t)+\int^{T^*}_0 (|g_{tt}|_4^2+ |g_{tt}|^2_{D^1})\text{d}t\leq c^2_5,&\\
					\text{ess}\sup_{0\leq t\leq T^*}t|v_{tt}(t)|^2_2+\int^{T^*}_0t(|v_{tt}|^2_{D^1}+|g^\frac12v_{tt}|_{D^1}^2+|v_t|^2_{D^3})\text{d}t\leq c_5^2,&
				\end{aligned}
			\end{equation}
			for  some time $T^*\in(0,T]$, where  
			\begin{equation}
				1<c_0\leq c_1\leq c_2\leq c_3\leq c_4\leq c_5,
			\end{equation}
 and  $T^*$, $c_i (i=1,\cdots,5)$ will be determined later,  which depend only on $c_0$ and  fixed constants  $A, R, c_v, \alpha,\beta,\gamma,\delta, T$.

            \subsection{Initial information}
			
		First, we state some initial regularity based on \ef{2.14} and compatibility conditions, which will be used in the a priori estimates of $\phi$ and $u$. 
						\begin{lemma}\label{r1}
It follows from \ef{2.14} and compatibility conditions that
\begin{equation*}
    |\phi_0^{3\iota}\nabla^2u_0|_2+|\phi_0^{2\iota}\nabla^2\phi_0|_2+|\phi_0^{3\iota}\nabla^3u_0|_2+|\phi_0^{2 \iota} \nabla^2 l_0|_2+|\phi_0^{3 \iota} \nabla^3 l_0|_2\leq \hat{C},
\end{equation*}
where $\hat{C}>0$ is a generic  constant that depends on $c_0$ but independent of $\epsilon$.
                        \end{lemma}
                        
                        \begin{proof}
				First, according to the definition of $g_3$, $\phi_0>\epsilon$ and $\ef{ln}_5$  that 
				\begin{equation*} \begin{aligned}
				&L(\phi_0^{3\iota}u_0)=  g_3-\frac{\delta-1}{a\delta}G(\phi_0,\psi_0,u_0),\quad 	\phi_0^{3\iota}u_0\longrightarrow 0 \ \ \text{as}\ \ |x|\longrightarrow \infty,\\
					\end{aligned}\end{equation*}
				where 
				\begin{equation}\label{Gdingyi}
                \begin{aligned}
                    \frac23 G(\phi_0,\psi_0,u_0)=&(\alpha+\beta)\big( \phi_0^\iota\psi_0{\rm{div}} u_0+\phi_0^\iota\psi_0\cdot\nabla u_0+u_0\cdot\nabla( \phi_0^\iota\psi_0)\big)\\&+\alpha\phi_0^\iota \psi_0\cdot\nabla u_0+\alpha {\rm{div}}\big(u_0\otimes \phi_0^\iota\psi_0\big).
                \end{aligned}
				\end{equation}
				Then according to Lemma \ref{ell}  and $\ef{2.14}$ that
				\begin{equation}\label{incc}
					\begin{split}
						|\phi_0^{3\iota}u_0|_{D^2}\leq & C(|  g_3|_2+|G(\phi_0,\psi_0,u_0)|_2)\leq C_1<\infty,\\
						|\phi_0^{3\iota}\nabla^2u_0|_2\leq & C\big(|\phi_0^{3\iota}u_0|_{D^2}+|\nabla\psi_0|_4|\phi_0^\iota u_0|_4\\&+|\psi_0|_\infty(|\phi_0^\iota\nabla u_0|_2+|\nabla\phi_0^\iota|_4|u_0|_4)\big)\leq C_1<\infty,
					\end{split}
				\end{equation}
				for some generic constant $C_1>0$ that independent of $\epsilon$. 
 Moreover, one has
				\begin{equation*}
					\begin{split}
						|\phi_0^{2\iota}\nabla^2\phi_0|_2
						\leq C_1.
					\end{split}
				\end{equation*}

				Second, since 
	$\nabla(\phi_0^{3\iota}Lu_0)=\phi_0^{-\iota}g_4\in L^2$,
				thus formally,
				\begin{equation*} \begin{aligned}
&L(\phi_0^{3\iota}u_0)=\Delta^{-1} {\rm{div}}\big(\phi_0^{-\iota}g_4\big)-\frac{\delta-1}{a\delta}G(\phi_0,\psi_0,u_0),\\
						&\phi_0^{3\iota}u_0\longrightarrow 0 \ \ \text{as}\ \ |x|\longrightarrow \infty,\\
					\end{aligned} \end{equation*}   
				this implies
				\begin{equation}\label{incc*}
					\begin{split}
						|\phi_0^{3\iota}u_0|_{D^3}\leq & C(|\phi_0^{-\iota}|_\infty |g_4|_2+|G(\phi_0,\psi_0,u_0)|_{D^1})\leq C_1<\infty,\\
						|\phi_0^{3\iota}\nabla^3u_0|_2\leq & C\big(|\phi_0^{3\iota} u_0|_{D^3}+|\nabla\psi_0|_4|\phi_0^\iota\nabla u_0|_4\\&+|\nabla\phi_0^\iota|_4(|\psi_0|_\infty|\nabla u_0|_4+|\nabla \psi_0|_4|u_0|_\infty)\\&+|\psi_0|^2_\infty|\phi_0^{-3\iota}|_\infty| u_0|_2+|\psi_0|_\infty|\phi_0^\iota\nabla^2 u_0|_2\\
						&+|\nabla^2\psi_0|_2|\phi_0^\iota u_0|_\infty\big)\leq C_1.
					\end{split}
				\end{equation}
				
				Similarly, the definition of $g_6$ and $\phi_0>\epsilon$ imply that
				\begin{equation*}
					 \begin{array}{l}
						\Delta\big(\phi_0^{2 \iota}(l_0-\bar{l})\big)=g_6+2 \nabla\phi_0^{2 \iota} \cdot \nabla l_0+(l_0-\bar{l}) \Delta\phi_0^{2 \iota}, \\
						\phi_0^{2 \iota}(l_0-\bar{l}) \longrightarrow 0 \text { as }|x| \longrightarrow \infty,
					\end{array}
				\end{equation*}
				which, together with the standard  elliptic theory and $\ef{2.14}$, yields that
				\begin{equation*}
					\begin{aligned}
						|\phi_0^{2 \iota}(l_0-\bar{l})|_{D^2} \leq & C\big(|g_6|_2+|\psi_0|_{\infty}|\nabla l_0|_2 +|l_0-\bar{l}|_{\infty}|\nabla^2 h_0|_2\big) \leq C_1<\infty, \\
						|\phi_0^{2 \iota} \nabla^2 l_0|_2 \leq & C\big(|\phi_0^{2 \iota}(l_0-\bar{l})|_{D^2}+|\psi_0|_{\infty}|\nabla l_0|_2 \\
						& +\left|l_0-\bar{l}\right|_{\infty}|\nabla^2 h_0|_2\big) \leq C_1<\infty .
					\end{aligned}
				\end{equation*}
				
				At last, since
				$
				\nabla(\phi_0^{2\iota} \Delta l_0)=\phi_0^{- \iota} g_7 \in L^2
				$, 
				thus formally
				\begin{equation*}\label{incom2}
					 \begin{array}{c}
						\Delta\big(\phi_0^{3 \iota}(l_0-\bar{l})\big)=\Delta^{-1} {\rm{div}}(g_7+\nabla\phi_0^{\iota} \cdot\phi_0^{2\iota} \Delta l_0)
					 +2 \nabla\phi_0^{3 \iota} \cdot \nabla l_0\\
                     +(l_0-\bar{l}) \Delta\phi_0^{3 \iota},\qquad \qquad\qquad\\
						\phi_0^{3 \iota}(l_0-\bar{l}) \longrightarrow  0 \text { as }\quad |x| \longrightarrow \infty,\qquad\qquad\qquad\qquad
					\end{array} 
				\end{equation*}
				which implies
				\begin{equation*}
					\begin{aligned}
						|\phi_0^{3 \iota}(l_0-\bar{l})|_{D^3}& \leq C(|g_7|_2+|\psi_0|_{\infty}|(h_0)^{-1}|_{\infty}^{\frac{1}{2}}|g_6|_2+\aleph) \leq C_1, \\
						|\phi_0^{3 \iota} \nabla^3 l_0|_2 &\leq C(|\phi_0^{3 \iota}(l_0-\bar{l})|_{D^3}+\aleph) \leq C_1,
					\end{aligned}
				\end{equation*}
				where
				$$
				\begin{aligned}
					\aleph= & |\psi_0|_{\infty}|\phi_0^{ \iota} \nabla^2 l_0|_2+|\nabla l_0|_4|\psi_0|_{\infty}|\nabla\phi_0^{\iota}|_4+|\phi_0^{ \iota}\nabla l_0|_4|\nabla^2\phi_0^{2 \iota}|_4 \\
					& +|l_0-\bar{l}|_{\infty}(|(h_0)^{\frac{1}{4}} \nabla^3 h_0|_2+|\nabla\phi_0^{\iota}|_4|\nabla^2\phi_0^{2 \iota}|_4+|\phi_0^{-\iota}|_\infty|\psi_0|_{\infty}|\nabla\phi_0^{ \iota}|_4^2) .
				\end{aligned}
				$$
             
                The proof of Lemma \ref{r1} is complete.
                 \end{proof}

                 \begin{lemma}\label{inin} It holds that 
\begin{equation*}
    \limsup _{\tau \rightarrow 0} \big(|w^{-\frac{\nu}{2}}  l_t(\tau)|_2+|h^{\frac{1}{2}} \nabla l_t(\tau)|_2+|h^\frac12u_t(\tau)|_2 +|h \nabla u_t(\tau)|_2\big) \le M(c_0).
\end{equation*}

\begin{proof}
                 First,   $(3.1)_3$ gives
				\begin{equation*}
					\begin{aligned}
						|w^{-\frac\nu2}l_t(\tau)|_2\le&\big(|w^{-\frac\nu2}\big(-v\cdot\nabla l+a_4w^\nu h\Delta l\\&+a_5w^\nu ng^2H(v)+a_6w^{\nu+1}\text{div}\psi+\Pi(l,h,w,g)\big)|_2\big)(\tau),
					\end{aligned}
				\end{equation*}
				then Lemmas \ref{ls}-\ref{r1} and \ef{2.14}  imply that
				\begin{equation}\label{in1}
					\begin{aligned}
						& \limsup _{\tau \rightarrow 0}|w^{-\frac{\nu}{2}}  l_t(\tau)|_2 \\
						\leq & C\big(|l_0^{-\frac{\nu}{2}}  u_0 \cdot \nabla l_0|_2+|l_0^{\frac{\nu}{2}} h_0 \Delta l_0|_2 +|l_0^{\frac{\nu}{2}}  h_0^b h_0^{2} H(u_0)|_2 \\
						&+|l_0^{1+\frac{\nu}{2}}  \text{div} \psi_0|_2+|l_0^{-\frac{\nu}{2}}  \Pi(l_0, h_0, w_0, g_0)|_2\big) \\
						\leq & C|l_0^{\frac{\nu}{2}}|_{\infty}(|l_0^{-\nu}|_{\infty}|u_0|_{\infty}|\nabla l_0|_2+|g_6|_2\\
						& +|\phi_0^{2 b \iota}|_{\infty}|\phi_0^{2\iota} \nabla u_0|_4^2)+C|l_0^{1+\frac{\nu}{2}}|_{\infty}(|\nabla^2 \phi_0^{2 \iota}|_2+|\nabla \phi_0^{ \iota}|_4^2) \\
						& +C|l_0^{\frac{\nu}{2}}|_{\infty}|\nabla \phi_0^{ \iota}|_4|\phi_0^{\iota}\nabla l_0|_4+C|l_0^{\frac{\nu}{2}-1}|_{\infty}|\phi_0^{\iota} \nabla l_0|_4^2 \leq M(c_0).
					\end{aligned}
				\end{equation}
                
Moreover, $(3.1)_3$ gives
				\begin{equation*}
					\begin{aligned}
						|h^{\frac{1}{2}} \nabla l_t(\tau)|_2 \leq & \mid h^{\frac{1}{2}} \nabla\big(-v \cdot \nabla l+a_4 w^\nu h \Delta l \\
						& +a_5 w^\nu n g^{2} H(v)+a_6 w^{\nu+1} \text{div} \psi+\Pi\big)|_2(\tau),
					\end{aligned}
				\end{equation*}
				thus \ef{4.1*}, \ef{2.14} and Lemmas \ref{ls}-\ref{r1} imply that
					\begin{align*}
						& \limsup _{\tau \rightarrow 0}|h^{\frac{1}{2}} \nabla l_t(\tau)|_2 \\
						\leq & C\big(|h_0^{\frac{1}{2}} \nabla\left(u_0 \cdot \nabla l_0\right)|_2+|h_0^{\frac{1}{2}} \nabla(l_0^\nu h_0 \Delta l_0)|_2 \\
						& +\mid h_0^{\frac{1}{2}} \nabla\big(l_0^\nu h_0^b h_0^2 H(u_0)\big)|_2+|h_0^{\frac{1}{2}} \nabla(l_0^{\nu+1} \text{div} \psi_0)|_2+|h_0^{\frac{1}{2}} \nabla\Pi_0|_2\big) \\
						\leq & C\big(|u_0|_{\infty}|\phi_0^{\iota} \nabla^2 l_0|_2+|\nabla u_0|_{\infty}|\phi_0^{\iota} \nabla l_0|_2+|l_0^\nu|_{\infty}|\phi_0^{3 \iota} \nabla^3 l_0|_2 \\
						& +|\phi_0^{\iota}\nabla l_0^\nu|_4|\phi_0^{2 \iota} \nabla^2 l_0|_4+|l_0^\nu|_{\infty}|\psi_0|_{\infty}|\phi_0^{-\iota}|_{\infty}|\phi_0^{2 \iota} \nabla^2 l_0|_2 \\
						& +|\phi_0^{2 b \iota}|_{\infty}(|l_0^{\nu-1}|_{\infty}|\phi_0^{\iota} \nabla l_0|_\infty|\phi_0^{2 \iota} \nabla u_0|_4^2+|l_0^\nu|_{\infty}|\phi_0^{2 \iota} \nabla u_0|_4|\phi_0^{3 \iota} \nabla^2 u_0|_4) \\
						& +|l_0^\nu|_{\infty}(|\phi_0^{(2 b-1) \iota}|_{\infty}|\nabla \phi_0^{2 \iota}|_\infty|\phi_0^{2 \iota} \nabla u_0|_4^2+|\nabla \psi_0|_4|\phi_0^{\iota} \nabla l_0|_4) \\
						& +|l_0^{\nu+1}|_{\infty}|\phi_0^{\iota} \nabla^2 \psi_0|_2+|l_0^\nu|_{\infty}|\phi_0^{-2 \iota}|_{\infty}|\psi_0|_{\infty}^2|\phi_0^{\iota} \nabla l_0|_2 \\
						& +|l_0^{\nu+1}|_{\infty}(|\phi_0^{- \iota}|_{\infty}|\nabla \phi_0^{ \iota}|_4^2|\nabla \phi_0^{2 \iota}|_\infty+|\nabla \phi_0^{ \iota}|_4|\nabla \psi_0|_4) \\
						& +|l_0^{\nu-1}|_{\infty}|\phi_0^{\iota} \nabla l_0|_2|\nabla l_0|_{\infty}|\psi_0|_{\infty}+|l_0^{\nu-2}|_{\infty}|\phi_0^{\iota}\nabla l_0|_{\infty}|\phi_0^{ \iota} \nabla l_0|_4^2 \\
						& +|l_0^\nu|_{\infty}(|\psi_0|_{\infty}|\phi_0^{\iota} \nabla^2 l_0|_2+|\nabla \psi_0|_4|\phi_0^{\iota} \nabla l_0|_4) \\
						& +|l_0^{\nu-1}|_{\infty}(|\psi_0|_{\infty}|\phi_0^{\iota} \nabla l_0|_2|\nabla l_0|_{\infty}+|\phi_0^{ \iota}\nabla l_0|_4|\phi_0^{2 \iota} \nabla^2 l_0|_4)\big) \leq M(c_0). 
					\end{align*}
                    
            Second,  $(3.1)_2$ gives
				\begin{equation*}
					\begin{aligned}
						|h^\frac12u_t(\tau)|_2  \leq & |h^\frac12\mathcal{K}(\tau)|_2 \\\leq & C(|v|_{\infty}|h^\frac12\nabla v|_2+|\phi|_{\infty}|h^\frac12\nabla l|_2+|h^\frac12\nabla \phi|_2|l|_{\infty} 
						 +|l^\nu|_{\infty}|h^\frac32 L u|_2\\&+|l^{\nu-1}|_{\infty}|g \nabla v|_{\infty}|h^\frac12\nabla l|_2+|\psi|_{\infty}|l^\nu|_{\infty}|h^\frac12\nabla v|_2)(\tau),
					\end{aligned}
				\end{equation*}
                where $\mathcal{K}=  v \cdot \nabla v+a_1 \phi \nabla l+l \nabla \phi+a_2 h l^\nu L u-a_2 \nabla l^\nu \cdot g Q(v) -a_3 l^\nu \psi \cdot Q(v)$, thus \ef{4.1*}, \ef{2.14}, \ef{incc}, \ef{incc*} and Lemma \ref{ls} imply that
				\begin{equation*}
				    \begin{aligned}
					\limsup _{\tau \rightarrow 0}|h^\frac12u_t(\tau)|_2 \leq & C(|u_0|_{\infty}\left|\phi_0^{ \iota}\nabla u_0\right|_2+\left|\phi_0\right|_{\infty}\left|\phi_0^{ \iota}\nabla l_0\right|_2+\left|\phi_0^{ \iota}\nabla \phi_0\right|_2\left|l_0\right|_{\infty}+\left|l_0^\nu\right|_{\infty}\left|g_3\right|_2 \\
					& +\left|\psi_0\right|_{\infty}|l_0^\nu|_{\infty}\left|\phi_0^{ \iota}\nabla u_0\right|_2+|l_0^{\nu-1}|_{\infty}|\phi_0^{2 \iota} \nabla u_0|_{\infty}\left|\phi_0^{ \iota}\nabla l_0\right|_2) \leq M(c_0).
				\end{aligned}
				\end{equation*}

                 At last, according to $h_0 l_0^\nu(h \nabla L u_0+ L u_0 \otimes \nabla h_0) 
					=  l_0^\nu g_4$, 
				$\ef{ln}_2$ \ef{4.1*}, \ef{2.14}, and Lemmas \ref{ls}-\ref{r1} that
				\begin{equation}\label{in4}
					\begin{aligned}
						& \limsup _{\tau \rightarrow 0}|h \nabla u_t(\tau)|_2 \leq \limsup _{\tau \rightarrow 0}|h \nabla \mathcal{K}(\tau)|_2 \\
						\leq & C\big(| u_0|_4|\phi_0^{2\iota}\nabla^2 u_0|_4+|\nabla u_0|_{\infty}|\phi_0^{2\iota} \nabla u_0|_2+|l_0|_{\infty}|\phi_0^{2\iota} \nabla^2 \phi_0|_2 \\
						& +| \phi_0^{2\iota+1}\nabla^2 l_0|_2+|\phi_0^\iota \nabla l_0|_4|\phi_0^\iota \nabla \phi_0|_4+|l_0^\nu|_{\infty}(|\nabla \psi_0|_4|\phi_0^{2\iota} \nabla u_0|_4 \\
						& +|\psi_0|_{\infty}|\phi_0^{2\iota} \nabla^2 u_0|_2)+|l_0^{\nu-1}|_{\infty}|\psi_0|_{\infty}|\phi_0^{2\iota} \nabla u_0|_4|\nabla l_0|_4 \\
						& +|l_0^\nu|_{\infty}|g_4|_2+|l_0^{\nu-1}|_{\infty}|h_0^{\frac{3}{2}} L u_0|_4|h_0^{\frac{1}{2}}\nabla l_0|_4 +|h_0\nabla^2 l_0^\nu|_2|h_0 \nabla u_0|_{\infty} \\
						& +|h_0^{\frac{1}{2}} \nabla l_0^\nu|_4(|h_0^{\frac{3}{2}} \nabla^2 u_0|_4+|\psi_0|_{\infty}|h_0^{\frac{1}{2}}\nabla u_0|_4)\big) \leq M(c_0). 
					\end{aligned}
				\end{equation}
				
            The proof of Lemma \ref{inin} is complete.
                \end{proof}
                 \end{lemma}

			\subsection{Uniform a priori estimates}

			\subsubsection{The estimates on the density-related variables}

			Let
			$$ 
			\varphi=h^{-1},\quad \zeta=\nabla h^\frac{1}{2},\quad 
			n=(ah)^b \quad \text{and}\quad \psi=\frac{a\delta}{\delta-1}\nabla h.$$ 
			\begin{lemma}\label{gh}
				Set $T_1=\min\{T^*,(1+Cc_0^2c_4)^{-1}\}$. For $0\leq t\leq T_1$, one has  
				\begin{equation}\label{g/h}
					\begin{aligned}
						h>\fr{1}{2c_0}, \quad \fr{2}{3}\epsilon^{-2\iota}<\varphi <  2c_0,\quad  \tilde{C}^{-1}\leq gh^{-1}\leq  \tilde{C},&
					\end{aligned}
				\end{equation}
				where $\tilde{C}$ is a suitable constant independent of $ \epsilon$ and $ c_i (i=0,\cdots, 5)$.
			\end{lemma}
			\begin{proof}
				First,  {\emph{estimates on $\varphi=h^{-1}$.}} 
				Note that 
				\[
				\varphi_t+v\cdot \nabla \varphi +(1-\delta)  \varphi \text{div} v=0,
				\]
				then along the particle path $\mathbb{X}(t;x)$ 
				\begin{equation}
						\label{2.34b}
						\displaystyle \frac{\text{d}}{\text{d}s}\mathbb{X}(t;x)=v\big(s,\mathbb{X}(t; x)\big),\quad  0\leq t\leq T;\quad
						\mathbb{X}(0;x)=x,  \quad   x\in \mathbb{R}^2,
				\end{equation}
				one has
				\begin{equation*}
					\begin{split}
						\displaystyle
						\varphi\big(t,\mathbb{X}(t;x)\big)=\varphi_0(x)\exp\Big((\delta-1)\int_0^t  \text{div}v \big(s,\mathbb{X}(s;x)\big)\text{d}s\Big).
					\end{split}
				\end{equation*}
				This, along with  \eqref{801}, implies that 
				\begin{equation}\label{2.34d}
					\begin{split}
						\displaystyle
						\frac{2}{3}\epsilon^{-2\iota}<\varphi(t,x)<2|\varphi_0|_\infty\leq 2c_0\quad \text{for} \quad  (t,x)\in [0,T_1]\times \mathbb{R}^2,
					\end{split}
				\end{equation}
				and 
				\[
				h>\frac{1}{2c_0}\quad \text{for} \quad  0\le t\le T_1.
				\]

				Before moving to prove the equivalence of $g$ and $h$ in short time, we first give the estimate of $h^{-1}\nabla h $, i.e. $h^{-1}\psi$.
				Notice that
				\[
				(h^{-1}\psi )_t +   v\cdot \nabla ( h^{-1}\psi) +  (\nabla v)^{\top}  h^{-1}\psi   +a\delta \nabla \text{div}v=0,
				\]
				along with the particle path \eqref{2.34b}, one has
				\begin{equation}
					\begin{split}
						(h^{-1}\psi) \big(t, &\mathbb{X}(t;x)\big)=\exp \left(-\int_0^t (\nabla v)^{\top} \big(s, \mathbb{X}(s; x)\big) \text{d}s \right) \\
						&  \cdot\left(\frac{a\delta}{\delta-1}h_0^{-1}\nabla h_0-a\delta \int_0^t \nabla \text{div} v \exp \int_0^{s} (\nabla v)^{\top} \text{d}\tau \big(s, \mathbb{X}(s;x)\big) \text{d}s  \right),
					\end{split}
				\end{equation}
				this together with \eqref{801} implies that
				\begin{equation}\label{h1p}
					|h^{-1}\psi|_\infty\le C c_0^2\quad \text{for} \quad 0\le t\le T_1.
				\end{equation}
				
				Second, 
				denoting $gh^{-1}=y(t,x)$, then 
				\begin{equation*}
					y_t+y h_t h^{-1} =g_t h^{-1},\quad y(0,x)=1,
				\end{equation*}
			thus
				\begin{equation}\label{eqhg}
					y(t, x)=\exp\left(-\int^t_0 h_s h^{-1} \text{d}s\right)\left(1+\int^t_0 g_s h^{-1} \exp\big(\int^s_0 h_\tau h^{-1}  \text{d}\tau\big)\text{d}s\right),
				\end{equation}
				where  
				\[
				h_t h^{-1} =- v\cdot \nabla h h^{-1}- (\delta-1)  \text{div}v,
				\]
				and by  \eqref{801}, \eqref{h1p}, it holds that
				\begin{equation}\label{mhh1}
					\begin{split}
						|h_t h^{-1}|_\infty 
						\le & C\big( |v|_\infty |h^{-1} \nabla h|_\infty  + |\nabla v|_{\infty}\big)
						\le Cc_0^2c_4,
					\end{split}
				\end{equation}
				this, together with \eqref{801}, \eqref{2.34d} and \eqref{eqhg},
				implies that
				\begin{equation}\label{djx}
					  \tilde{C}^{-1}\leq hg^{-1}(t)\leq \tilde{C}
				\end{equation}
				for $ 0\le t\le T_1=\min\{T^*,(1+Cc_0^2c_4)^{-1}\}$.
						Moreover,
\begin{equation*}\label{g-1}
    |g^{-1}|_\infty=|hg^{-1}\varphi|_\infty\le \tilde{C}|\varphi|_\infty\le Cc_0.
\end{equation*}

Based on \eqref{djx}, we will not distinguish $h$ and $g$ in the rest of the a priori estimates.
				The proof of Lemma \ref{gh} is complete.
                
			\end{proof}

			\subsubsection{Estimates on $\phi$.} 
			
			\begin{lemma}\label{phiphii}
				Set $T_2=\min\{T_1,(1+Cc_4)^{-2}\}$. Then for $ t\in [0, T_2]$, one has
				\begin{equation*} \begin{aligned}
						&|\phi|_\infty+|h^\frac12\nabla \phi|_2+\|\phi-\epsilon\|_{L^p\cap D^1 \cap D^3}\leq Cc_0,\ \ |h\nabla^2 \phi|_2\leq Cc_0^2,\ \ |h^\frac12\phi_t|_2\leq Cc_0^\frac32c_2,\\
						&| h\nabla\phi_t|_{2}\leq Cc_0^2c_3,\quad| \phi_t|_{D^2}\leq Cc_0c_4,\quad|\phi_{tt}|_2\leq Cc_0c_4^2,\quad\int^t_0\|\phi_{ss}\|^2_1 {\rm{d}}s\leq Cc_0^2c_4^2.
				\end{aligned}\end{equation*}
			\end{lemma}
			\begin{proof} First, it follows from $\eqref{ln}_1$
            and $\eqref{801}$, one has
				%
				  
				\begin{equation}\label{2093}\begin{aligned}
						\|\phi-\epsilon\|_{L^p\cap D^1\cap D^3}\leq &C\Big(\|\phi_0-\epsilon\|_{L^p\cap D^1 \cap D^3}+\int^t_0\|\nabla v\|_3\text{d}s\Big)\\
						&\cdot \exp\Big(\int^t_0C\|v\|_3 \text{d}s\Big)
						\leq Cc_0,
				\end{aligned}\end{equation}
				for $0\leq t\leq T_2$. 
                Moreover, 
				\begin{equation}\label{mbphii}
					|\phi|_\infty= |\phi-\epsilon +\epsilon|_\infty\leq |\phi-\epsilon|_\infty +\epsilon\le C\|\phi-\epsilon\|_{L^p\cap D^1\cap D^3}+\epsilon\le Cc_0.
				\end{equation}
                
                Similarly, through the standard energy estimates, and
                \begin{equation*}
					\begin{split}
						h^{\frac12}_t+v\cdot \nabla h^\frac12+ \frac{\delta-1}{2}h^\frac12 \text{div}v=0,\qquad
                        h_t+v\cdot \nabla h+ (\delta-1)h \text{div}v=0,
					\end{split}
				\end{equation*}
   it is easy to deduce from $\eqref{ln}_1$, $\eqref{801}$ and  Lemma \ref{r1}-\ref{gh} that
                \begin{equation}\label{2094}
                    \begin{aligned}
                       |h^\frac12\nabla \phi|_2\le& C\Big(|\phi
						^{\iota}_0\nabla \phi_0|_2+\int^t_0|h^\frac12\nabla^2 v|_2\text{d}s\Big)\exp{\left(\int_0^tC\|v\|_3\text{d}s\right)}\le Cc_0,\\
						|h\nabla^2 \phi|_2\le& C\Big(|\phi
						^{2\iota}_0\nabla^2 \phi_0|_2+\int^t_0(|h\nabla^2 v|_2+|h\nabla^3 v|_2)\text{d}s\Big)\\
						&\cdot\exp{\left(\int_0^tC\|v\|_3\text{d}s\right)}\le Cc_0^2,
                    \end{aligned}
                \end{equation}
				for $0\leq t\leq T_2$.
				
				Then, according to  $\ef{ln}_1$, \ef{801} and \eqref{2093}-\eqref{2094},  for $0\leq t\leq T_2$,  one has
				\begin{equation}\label{phitt} \begin{aligned}
						&|\phi_t|_2\leq C(|v|_4|\nabla\phi|_4+|\phi|_\infty|\nabla v|_2)\leq Cc_0c_2,\\
						&|\phi_t|_{D^1}\leq C(|v|_\infty|\nabla^2\phi|_2+|\nabla\phi|_4|\nabla v|_4+|\phi|_\infty|\nabla^2 v|_2)\leq Cc_0c_3,\\
                        &|h^\frac12\phi_t|_2\leq C(|v|_\infty|h^\frac12\nabla\phi|_2+|\phi|_\infty|h^\frac12\nabla v|_2)\leq Cc_0^\frac32c_2,\\
						&|h\nabla\phi_t|_{2}\leq C(|v|_\infty|h\nabla^2\phi|_2+|\nabla\phi|_4|h\nabla v|_4+|\phi|_\infty|h\nabla^2 v|_2)\leq Cc_0^2c_3,\\
						&|\phi_t|_{D^2}\leq C\|v\|_3(\|\nabla\phi\|_2+|\phi|_\infty) \leq Cc_0c_4.
				\end{aligned} \end{equation}

				Second,  applying $\partial_t$ to $\ef{ln}_1$,
				using \ef{801}, \eqref{2093}-\eqref{phitt}, for $0\leq t\leq T_2$, one has
				\begin{equation*}\begin{aligned}
						&|\phi_{tt}|_2\leq  C(|v_t|_4|\nabla\phi|_4+|v|_\infty|\nabla\phi_t|_2+|\phi_t|_4|\nabla v|_4+|\phi|_\infty|\nabla v_t|_2)\leq Cc_0c_4^2,\\
						\int^t_0& \|\phi_{ss}\|^2_1\text{d}s\leq C \int_0^t\big( \| (v\cdot\nabla\phi )_s\|_1^2+\| (\phi\dv v)_s\|_1^2\big)\text{d}s \leq  Cc_0^2c_4^2.
				\end{aligned}\end{equation*}

				The proof of Lemma \ref{phiphii} is complete.
                
			\end{proof}

			\subsubsection{Estimates on $\psi$.}
			\begin{lemma}\label{psi} Set $T_3=\min\{T_2,(1+Cc_4)^{-5}\}$. Then  for $t\in [0, T_3]$ and  $2<q<\infty$, one has
				\begin{equation*}
					\begin{split}
						|\psi|_\infty^2+\|\psi\|^2_{L^q\cap D^{1}\cap D^2}\leq Cc_0^2,\quad |\psi_t|_2\leq Cc_0c_3,\quad |\psi_t|^2_{D^1}+  \int^t_0 |\psi_{ss}|^2_2{\rm{d}}s\le Cc_4^4,\\
						|h_t |_\infty\le  M(c_0)c_3^{\frac34}c_4^{\frac14},\quad \int_0^t |h_s|_\infty^2 {\rm{d}}s\le M(c_0),\quad  \int_0^t |  h_{ss}|_4^2 {\rm{d}}s \le  M(c_0)c_4^2.
					\end{split}
				\end{equation*}
				
			\end{lemma}
			\noindent
			\begin{proof} First, it follows from    $\ef{ln}_1$ that 
				\begin{equation}\label{psieq}
					\psi_t+\sum_{k=1}^2 A_k(v) \partial_k\psi+ B(v) \psi+a\delta h  \nabla  \text{div} v  =0,
				\end{equation}
				where $A_k(v)=(a^k_{ij})_{2\times 2}$ for  $i$, $j$, $k=1, 2$,
				are symmetric  with
				$a^k_{ij}=v^{(k)}\quad \text{for}\ i=j;\ \text{otherwise}\  a^k_{ij}=0, $ and $B(v)=(\nabla v)^\top+(\delta-1)\text{div}v\mathbb{I}_2$.   
Then, the classical energy estimate  on \eqref{psieq} and Lemma \ref{gh} imply
				\begin{equation}\label{psiq}
					\begin{split}
						\frac{\text{d}}{\text{d}t}|\psi|^q_q \leq & C(|\nabla v|_\infty|\psi|^q_q+|hg^{-1}|_\infty |g\nabla^2v|_{q}|\psi|_q^{q-1})\\
						\leq & C(|\nabla v|_\infty|\psi|^q_q +\|g\nabla^2v\|_{1}|\psi|_q^{q-1}),
					\end{split}
				\end{equation} 
				this, together  with \eqref{2.14},  \eqref{801}, \eqref{psiq} and  Gronwall's inequality, gives 
				\begin{equation}\label{psideq} 
					|\psi|_q \leq Cc_0 \quad \text{for} \quad  0\leq t\leq T_3.
				\end{equation}

				Similarly, one can deduce that 
				\begin{equation*}
					\frac{\text{d}}{\text{d}t}\|\psi\|_{D^{1}\cap D^2}\leq C\big(c_4\|\psi\|_{D^{1}\cap D^2}+|g\nabla^4 v|_{2}+ c_4\big),
				\end{equation*}
				this, together with  \eqref{2.14} and Gronwall's inequality give that for $0\leq t \leq T_3$,
				\begin{equation}\label{2.26a}\begin{split}
						\|\psi\|_{D^{1}\cap D^2}\leq&  \Big(c_0+Cc_4t+C  \int_0^t |g\nabla^4 v|_{2} \text{d}s\Big) \exp(Cc_4t)\leq Cc_0.
					\end{split}
				\end{equation}
				Moreover,  \eqref{psideq}-\eqref{2.26a} and Lemma \ref{lem2as} imply
				\begin{equation}\label{3.24}
					|\psi|_\infty\le Cc_0,\quad \text{for} \quad 0\leq t \leq T_3.
				\end{equation}

				Third, according to \eqref{801}, \eqref{mhh1}, $\ef{psieq}$, \eqref{psideq}-\eqref{3.24},  for  $0\leq t \leq T_3$, one has
				\begin{equation*}\label{psit2} 
					\begin{aligned}
						|\psi_t|_2\le & C\big(| v|_\infty|\nabla \psi|_{2}+|\nabla v|_2|\psi|_{\infty} +|hg^{-1}|_\infty|g\nabla^2 v|_2\big)\leq Cc_0c_3,\\ 
						 | \psi_t|_{D^1} \le & C \big(\|v\|_{2} \|\psi\|_{L^q\cap D^{1}\cap D^2} +|hg^{-1}|_\infty |g\nabla^3 v|_{2}\big) \leq Cc^2_4,\\
						\int_0^t |\psi_{ss}|^2_{2} \text{d}s
						\leq& C\int_0^t \big(|v_s|^2_4|\nabla \psi|^2_4+|\nabla v|^2_\infty|\psi_s|^2_{2}+|v|^2_\infty|\nabla \psi_s|^2_2+|\psi|^2_\infty|\nabla v_s|^2_{2}\\
						&+|hg^{-1}|_\infty^2 ( |h^{-1} h_t|_\infty^2 |g\nabla^2 v|_2^2+|g \nabla^2 v_s|^2_{2} )\big) \text{d}s
						\leq Cc^4_4.
					\end{aligned}
				\end{equation*}

				At last, it follows from    \ef{801}, the estimates on $\psi$, Lemma \ref{gh} and
				\begin{equation*}
					\begin{split}
						h_t=-   v\cdot \nabla h- (\delta-1)h \text{div}v,\quad
						h_{tt}=& - (v\cdot \nabla h)_t-    (\delta-1) (h  \text{div}v)_t,
					\end{split}
				\end{equation*}
			that for $0\le t\le T_3$
				\begin{equation*} 
					\begin{split}
						|h_t |_\infty\le  & C(  |v|_\infty |\psi|_\infty  +  |h  \nabla v|_\infty) \le M(c_0)c_3^{\frac34}c_4^{\frac14},\\
						\int_0^t |  h_{ss}|_4^2 \text{d}s \le & C\int_0^t \big(  |\psi|_\infty^2 |v_s|_4^2+|\psi_s|_4^2|v|_{\infty}^2+|h_s|^2_\infty |\nabla v|_2^2
						+  | h\nabla v_s|_4^2 \big)\text{d}s \\ \le &  M(c_0)c_4^2.
					\end{split}
				\end{equation*}

				The proof of Lemma \ref{psi} is complete.
                
			\end{proof}

			\begin{lemma}\label{h3438n}
				For $t\in[0,T_3]$ and $q>2$, one has
				\begin{equation*}
					\begin{aligned}
      |\zeta|_4+|n|_\infty+|\nabla n|_\infty+\|n\|_{D^{1,q}\cap D^{1,4}\cap D^2\cap D^3}\le M(c_0),\quad
      |n_t|_2\le M(c_0)c_2,\\
						|n_t|_\infty+|\nabla n_t|_2+|\nabla n_t|_4\le  M(c_0)c_4,\quad|n_{tt}|_2\le M(c_0)c_4^3.
					\end{aligned}
				\end{equation*}
			\end{lemma}
			\begin{proof}

				First, note that \ef{psieq}   implies that
				\begin{equation*}
					\zeta_t+\sum_{k=1}^{2}A_k(v)\partial_k\zeta+B^*(v)\zeta+\frac12(\delta-1)h^\frac{1}{2}\nabla\text{div}v=0,
				\end{equation*}	
				where $B^*(v)=(\nabla v)^\top+\frac{(\delta-1)}{2}\text{div}v\mathbb{I}_2$. Thus one can easily obtain 
				\begin{equation*}
					|\zeta(t)|_4\le Cc_0\quad\text{for}\quad 0\le t\le T_3.
				\end{equation*}
				

				Second, $n=(ah)^b$ implies that
				\begin{equation*}
					n_t+v\cdot\nabla n+(2-\delta-\gamma)a^bh^{b}\text{div}v=0,
				\end{equation*}
this, together with Lemmas \ref{gh}-\ref{psi} and \ef{801}, implies  that for $0\le t\le T_3$,
					\begin{align*}
						&|n|_\infty\le a^b|\varphi|_\infty^{-b}\le M(c_0),\quad|\nabla n|_q=a^b|bh^{b-1}\nabla h|_q\le M(c_0),\quad\\
						&|\nabla n|_\infty=a^b|bh^{b-1}\nabla h|_\infty\le M(c_0),\quad |\nabla n|_4=2a^b|bh^{b-\frac{1}{2}}\nabla h^\frac{1}{2}|_4\le M(c_0),\\&|\nabla^2n|_2\le C(|h^{b-1}\nabla^2h|_2+|h^{b-1}\nabla h^\frac{1}{2}\cdot\nabla h^\frac{1}{2}|_2)\le M(c_0),\\&
						|\nabla^3n|_2\le C(|h^{b-1}\nabla^3h|_2+|h^{b-2}\nabla^2h\cdot\nabla h|_2+|h^{b-2}\nabla h\cdot\nabla h^{\frac{1}{2}}\cdot\nabla h^{\frac{1}{2}}|_2)\le M(c_0),\\&
						|n_t|_\infty\le C(|v|_\infty|\nabla n|_\infty+|h^{b}|_\infty|\text{div}v|_\infty)\le M(c_0)c_4,\\&
						|n_t|_2\le C(|v|_4|\nabla n|_4+|h^b|_\infty|\text{div}v|_2)\le M(c_0)c_2,\\&
						|\nabla n_t|_2\le C(|\nabla(v\cdot\nabla n)|_2+|\nabla(h^{b}\text{div}v)|_2)\le M(c_0)c_4,\\&
						|\nabla n_t|_4\le C(|\nabla(v\cdot\nabla n)|_4+|\nabla(h^{b}\text{div}v)|_4)\le M(c_0)c_4,\\&
						|n_{tt}|_2\le C(|(v\cdot\nabla n)_t|_2+|(h^{b}\text{div}v)_t|_2)\le M(c_0)c_4^3.
					\end{align*}

				The proof of Lemma \ref{h3438n} is complete.

			\end{proof}
			
			\subsubsection{Estimates on $l$.}  \label{ells}
			
			\begin{lemma}\label{l1}
				For $T_4=\min\{T_3, (1+Cc_4)^{-10-4\nu}\}$ and $t\in[0, T_4]$, one has
				\begin{equation}\label{802}
					\begin{aligned}
						|h^\frac{1}{2}\nabla l(t)|_2^2+\int_{0}^{t}|w^{-\frac{\nu}{2}}l_s|_2^2{\rm{d}}s\le M(c_0),\quad |(l-\bar{l})|_2^2\le M(c_0)c_1^\nu,\\
								\int_{0}^{t}|l_s|_2^2{\rm{d}}s\le M(c_0)c_1^{\nu},\quad
       \int_{0}^{t}(|h\nabla^2l|_2^2+|\nabla^2l|^2_2){\rm{d}}s\le M(c_0)c_1^{3\nu}.
					\end{aligned}
				\end{equation}
			\end{lemma}
			\begin{proof}
				First,  $(3.1)_3$ implies
				\begin{equation}\label{803}
					\begin{aligned}
						w^{-\nu}(l_t+v\cdot\nabla l)-a_4h\Delta l=a_5ng^2H(v)+a_6w\text{div}\psi+w^{-\nu}\Pi(l,h,w,g),
					\end{aligned}
				\end{equation}
				and 
				\begin{equation}\label{620}
					\begin{aligned}
						w^{-\nu}((l-\bar{l})_t+v\cdot\nabla (l-\bar{l}))-a_4h\Delta (l-\bar{l})=a_5ng^2H(v)+a_6w\text{div}\psi\\+w^{-\nu}\Pi(l-\bar{l},h,w,g),
					\end{aligned}
				\end{equation}
                where
                \begin{equation*}
					H(v)=2\alpha\sum_{i=1}^{2}(\partial_iv^{(i)})^2+\beta(\text{div}v)^2+\alpha\sum_{i\neq j}^{2}(\partial_iv^{(j)})^2+2\alpha v_2^{(1)}v_1^{(2)}.
			\end{equation*}
				Multiplying \ef{620} by $l-\bar{l}$ and integrating over $\mathbb{R}^2$, it follows from \ef{801} and Lemmas \ref{gh}-\ref{h3438n} that
				
					\begin{align*}
						&\frac{1}{2}\frac{\text{d}}{\text{d}t}|w^{-\frac{\nu}{2}}(l-\bar{l})|_2^2+a_4|h^\frac{1}{2}\nabla (l-\bar{l})|_2^2\\=&-a_4\int\nabla h\cdot\nabla (l-\bar{l})(l-\bar{l})-\int w^{-\nu}v\cdot\nabla (l-\bar{l})(l-\bar{l})+a_5\int ng^2H(v)(l-\bar{l})\\&+a_6\int w\text{div}\psi (l-\bar{l})+\int w^{-\nu}\Pi(l-\bar{l},h,w,g)(l-\bar{l})+\frac{1}{2}\int (w^{-\nu})_t| (l-\bar{l})|^2\\
						\le&C(|w^\frac{\nu}{2}|_\infty|\psi|_\infty+|w^{-\frac{\nu}{2}}|_\infty|v|_\infty)|\varphi|^\frac{1}{2}_\infty|w^{-\frac{\nu}{2}}(l-\bar{l})|_2|h^\frac{1}{2}\nabla (l-\bar{l})|_2\\&+C\big(|w^\frac{\nu}{2}|_\infty|n|_\infty|g\nabla v|_4^2+|w^{1+\frac{\nu}{2}}|_\infty(|\nabla\psi|_2+|\nabla h^\frac{1}{2}|^2_4)\\&+
						|w^{-1+\frac{\nu}{2}}|_\infty|g^\frac12\nabla w|_4^2\big)|w^{-\frac{\nu}{2}}(l-\bar{l})|_2\\
						&+C|w^{-1}|_\infty|w_t|_\infty|w^{-\frac{\nu}{2}}(l-\bar{l})|_2^2\\
						\le&M(c_0) c_3^{\nu+10}|g^\frac12\nabla^2w_t|_2^\frac14|w^{-\frac{\nu}{2}}(l-\bar{l})|_2^2+M(c_0)+\frac{a_4}{2}|h^\frac{1}{2}\nabla (l-\bar{l})|_2^2,
					\end{align*}
				where we have used the following estimates
				\begin{equation}\label{701}
					\begin{aligned}
						|g\nabla v|_4&\le C|g\nabla v|_2^\frac12|\nabla(g\nabla v)|^\frac12_2\le Cc_2^\frac12(|g\nabla^2 v|_2+|\nabla g|_\infty|\nabla v|_2)^\frac12\le Cc_3^\frac32,\\
					|g^\frac12\nabla w|_4&
					\le C|g^\frac12\nabla w|_2^\frac12|\nabla(g^\frac12\nabla w)|_2^\frac12\\&\le Cc_2^\frac12(|g^\frac12\nabla^2 w|_2+|\nabla g|_\infty|g^{-1}|_\infty|g^\frac12\nabla w|_2)^\frac12\le Cc_2^2,\\
|w_t|_\infty&\le|w_t|^\frac14_2|\nabla w_t|^\frac12_2|\nabla^2w_t|_2^\frac14\\&\le|w_t|^\frac14_2|g^\frac14\nabla w_t|^\frac12_2|g^\frac12\nabla^2w_t|_2^\frac14|g^{-1}|^\frac14_\infty\le Cc_2|g^\frac12\nabla^2w_t|_2^\frac14.
					\end{aligned}
				\end{equation}
                
				Thus,  \ef{2.14}, \ef{801} and Gronwall's inequality give 
				\begin{equation*}\label{621}
					|w^{-\frac{\nu}{2}}(l-\bar{l})|_2^2+\int_{0}^{t}|h^\frac{1}{2}\nabla (l-\bar{l})|_2^2\text{d}s\le M(c_0)\quad \text{for}\quad 0\le t\le T_4.
				\end{equation*}
				 Moreover, 
			  \ef{801} and Lemma \ref{gh} imply
				\begin{equation*}\label{622}
					|l-\bar{l}|_2^2\le M(c_0)c_1^\nu,\quad\int_{0}^{t}|h^\frac{1}{2}\nabla l|_2^2\text{d}s\le M(c_0).
				\end{equation*}
				
				Multiplying \ef{803} by $l_t$ and integrating over $\mathbb{R}^2$, 
				it follows from \ef{801}, \ef{701}, Lemmas \ref{gh}-\ref{h3438n}  that 
				\begin{equation*}
					\begin{aligned}
						&\frac{a_4}{2}\frac{\text{d}}{\text{d}t}|h^\frac{1}{2}\nabla l|_2^2+|w^{-\frac{\nu}{2}}l_t|_2^2\\=&-a_4\int\nabla h\cdot\nabla ll_t-\int w^{-\nu}v\cdot\nabla ll_t+a_5\int ng^2H(v)l_t\\&+a_6\int w\text{div}\psi l_t+\int w^{-\nu}\Pi(l,h,w,g)l_t+\frac{1}{2}\int a_4h_t|\nabla l|^2\\
						\le&C(|w^\frac{\nu}{2}|_\infty|\psi|_\infty+|w^{-\frac{\nu}{2}}|_\infty|v|_\infty)|\varphi|^\frac{1}{2}_\infty|h^\frac{1}{2}\nabla l|_2|w^{-\frac{\nu}{2}}l_t|_2\\&+
						C\big(|w^\frac{\nu}{2}|_\infty|n|_\infty|g\nabla v|_4^2+|w^{1+\frac{\nu}{2}}|_\infty(|\nabla\psi|_2+|\nabla h^\frac{1}{2}|^2_4)\\&+
						|w^{-1+\frac{\nu}{2}}|_\infty|g^\frac{1}{2}\nabla w|_4^2\big)|w^{-\frac{\nu}{2}}l_t|_2+C|\varphi|_\infty|h_t|_\infty|h^\frac{1}{2}\nabla l|_2^2\\
		\le&M(c_0)c_4^{\nu+10}|h^\frac12\nabla l|_2^2+M(c_0)+\frac12|w^{-\frac{\nu}{2}}l_t|_2^2,
					\end{aligned}
				\end{equation*}
which along with \ef{2.14}, \ef{801} and Gronwall's inequality imply 
				\begin{equation*}\label{623}
					|h^\frac12\nabla l|_2^2+\int_{0}^{t}|w^{-\frac{\nu}{2}}l_s|_2^2\text{d}s\le M(c_0) \quad \text{for}\quad 0\le t\le T_4.
				\end{equation*}
  Moreover, \ef{801} and Lemma \ref{gh} gives
				\begin{equation}\label{804}
					|\nabla l|_2^2\le M(c_0),\quad\int_{0}^{t}|l_s|_2^2\text{d}s\le M(c_0)c_1^\nu.
				\end{equation}

				On the other hand, \ef{803} implies that 
				\begin{equation}\label{805}
					\begin{aligned}
						-a_4\Delta\big(h(l-\bar{l})\big)&=-a_4h\Delta l-a_4F(\nabla h,l-\bar{l})=w^{-\nu}\mathcal{A}-a_4F(\nabla h,l-\bar{l}),
					\end{aligned}
				\end{equation}
				where
				\begin{equation}
					\begin{aligned}
						& \mathcal{A}=-\left(l_t+v \cdot \nabla l\right)+a_5 w^\nu n g^{2} H(v)+a_6 w^{\nu+1} \text{div} \psi+\Pi(l-\bar{l}, h, w, g), \\
						& F=F\big(\nabla h, l-\bar{l}\big)=(l-\bar{l}) \Delta h+2 \nabla h \cdot \nabla l .
					\end{aligned}
				\end{equation}
				%
				%
				Then  \ef{2.14},  \ef{801}, \ef{804} and Lemma 3.2-\ref{h3438n} imply 
				\begin{equation}\label{806}
					\begin{aligned}
						|\mathcal{A}|_2\le& C(|l_t|_2+|v|_\infty|\nabla l|_2+|w^\nu|_\infty|n|_\infty|g\nabla v|_4^2+|w^{\nu+1}|_\infty|\nabla\psi|_2\\
					&+|w^{\nu+1}|_\infty|\nabla h^\frac12|_4^2+|w^\nu|_\infty|\psi|_\infty|\nabla l|_2+|w^{\nu-1}|_\infty|g^\frac12\nabla w|_4^2)\\
						\le& M(c_0)(|l_t|_2+c_1^{1+\nu}),\\
						|F|_2\le& C(|\nabla^2 h|_4|l-\bar{l}|_4+|\psi|_\infty|\nabla l|_2)
						\le M(c_0)c_1^\frac\nu4,
					\end{aligned}
				\end{equation}
				where  we have used  the fact that for $0\le t\le T_4$
				\begin{equation}\label{g32}
					\begin{aligned}
						\|v\|_2\le \|u_0\|_2&+t^\frac12\left(\int _0^t\|v_s\|_2^2\text{d}s\right)^\frac12\le M(c_0)(1+c_4t^\frac12)\le M(c_0),\\
						|g\nabla v|_4\le& |g(0,x)\nabla u_0|_4+t^\frac12\left(\int_0^t|(g\nabla v)_s|_4^2\text{d}s\right)^\frac12\\
						\le& C|h_0\nabla u_0|_2^\frac12|\nabla(h_0 u_0\big)|_2^\frac12+Ct^\frac12\Big(\int_0^t(|g_s|_\infty^2|\nabla v|_4^2+|g\nabla v_s|_4^2)\text{d}s\Big)^
						\frac12\\
						\le& M(c_0)+t^\frac12\Big(\int_0^t|g\nabla v_s|_2|\nabla(g\nabla v_s)|_2\text{d}s\Big)^\frac12\le M(c_0),\\
						|g^\frac12\nabla w|_4\le& |g^\frac12(0,x)\nabla l_0|_4+t^\frac12\Big(\int_0^t|(g^\frac12\nabla w)_s|_4^2\text{d}s\Big)^\frac12\\
						\le& C|h_0^\frac12\nabla l_0 |_2^\frac12|\nabla(h_0^\frac12\nabla l_0 )|_2^\frac12\\&+Ct^\frac12\Big(\int_0^t\big(|g^{-\frac12}g_s|^2_\infty|\nabla w|_4^2+|g^\frac12\nabla w_s|_4^2\big)\text{d}s\Big)^\frac12\\
						\le& M(c_0)+Ct^\frac12\Big(\int_0^t|g^\frac12\nabla w_s|_2|\nabla(g^\frac12\nabla w_s)|_2\big)\text{d}s\Big)^\frac12\le M(c_0),\\
						|l-\bar{l}|_4\le&|l-\bar{l}|_2^\frac12|\nabla l|_2^\frac12\le M(c_0)c_1^\frac\nu4.
					\end{aligned}
				\end{equation}
                
				Then \ef{806}, Lemma \ref{ell} and Lemma \ref{psi}  imply
				\begin{equation}\label{807}
					\begin{aligned}
						|h(l-\bar{l})|_{D^2} \leq & C(|w^{-\nu} \mathcal{A}|_2+|F(\nabla h, l-\bar{l})|_2) \\
						\leq & M(c_0)(c_1^{2\nu+1} +c_1^\nu| l_t|_2), \\
						|h \nabla^2 l|_2 \leq & C(|h(l-\bar{l})|_{D^2}+|\nabla^2 h(l-\bar{l})|_2 
						+|\nabla l \cdot \nabla h|_2) \\
						\leq & C(|h(l-\bar{l})|_{D^2}+|\nabla^2h|_4|l-\bar{l}|_4 
						+|\psi|_{\infty}|\nabla l|_2) \\\leq& M(c_0)(c_1^{2\nu+1} +c_1^\nu| l_t|_2) .
					\end{aligned}
				\end{equation}
				Consequently, \eqref{807} together with \ef{804}, yield that $\ef{802}_2$ holds for $0\le t\le T_4$. 

			\end{proof}

			\begin{lemma}\label{l2}
				For $T_5=\min\{T_4,(1+Cc_4)^{-16-8\nu}\}$ and $t\in[0,T_5]$, one has
				\begin{equation}
					\begin{aligned}
						|w^{-\frac{\nu }{2}}l_t|_2^2+\int_0^t(|h^\frac12\nabla l_s|_2^2+|\nabla l_s|^2_2){\rm{d}}s&\le M(c_0),\\
						|l_t|_2+|h\nabla^2l|_2+|l|_{D^2}&\le M(c_0)c_1^{2\nu+1},\\
						\int_0^t(|h\nabla^3l|_2^2+|h\nabla^2l|_{D^1}^2+|l|^2_{D^3}){\rm{d}}s&\le M(c_0)c_1^{2\nu+2}.	
					\end{aligned}
				\end{equation}
			\end{lemma}
			\begin{proof}
				Applying $\partial_t$ to $(3.1)_3$, one has
				\begin{equation}\label{703}
					\begin{aligned}
							l_{tt}-a_4w^\nu h\Delta l_t=&-(v\cdot\nabla l)_t+a_4(w^\nu)_t h\Delta l+a_4w^\nu h_t\Delta l\\&+
						a_5(w^\nu ng^2H(v))_t+a_6(w^{\nu+1}\text{div}\psi)_t+\Pi(l,h,w,g)_t.
					\end{aligned}
				\end{equation}
                
				Multiplying \ef{703} by $w^{-\nu}l_t$ and integrating over $\mathbb{R}^2$, one has
				\begin{equation}\label{808}
					\begin{aligned}
						& \frac{1}{2} \frac{\text{d}}{\text{d} t}|w^{-\frac{\nu}{2}}  l_t|_2^2+a_4|h^{\frac{1}{2}} \nabla l_t|_2^2 \\
						= & -\int( v \cdot \nabla l)_t w^{-\nu} l_t 
						+a_4 \int (w_t^\nu h \Delta l+w^\nu h_t \Delta l) w^{-\nu} l_t \\
						& +\int a_5\big(w^\nu n g^2 H(v)\big)_t w^{-\nu} l_t +\int a_6(w^{\nu+1}  \text{div} \psi)_t w^{-\nu} l_t  \\
						& +\int \Pi(l, h, w, g)_t w^{-\nu} l_t-a_4 \int\big( \nabla h\cdot \nabla l_t l_t+\frac{1}{2} (w^{-\nu} )_t|l_t|^2\big)=\sum_{i=1}^6 I_i,
					\end{aligned}
				\end{equation}
    where  
					\begin{align*}
						I_1=&-\int(v\cdot\nabla l)_tw^{-\nu}l_t\le C|w^{-\frac\nu2}|_\infty(|v_t|_4|\nabla l|_4+|\varphi|^\frac12_\infty|v|_\infty|h^\frac12\nabla l_t|_2)|w^{-\frac\nu2}l_t|_2,\\
						I_2=&a_4\int(w_t^\nu h\Delta l+w^\nu h_t\Delta l)w^{-\nu}l_t\\
						=&-a_4\int\nabla(w_t^\nu h w^{-\nu}l_t)\cdot\nabla l+a_4\int h_t\Delta ll_t\\
						\le&C(|w^{-2+\frac\nu2}|_\infty|w_t|_\infty|g^{-1}h|_\infty^\frac12|g^\frac12\nabla w|_4|h^\frac12\nabla l|_4
						+|w^{-1+\frac\nu2}|_\infty|w_t|_\infty|\psi|_\infty|\nabla l|_2\\&
						+|w^{-1+\frac\nu2}|_\infty|g^{-1}h|_\infty^\frac12|g^\frac12\nabla w_t|_4|h^\frac12\nabla l|_4+|w^\frac\nu2|_\infty|\varphi|^\frac12_\infty|h^{-\frac12}h_t|_\infty|h\nabla^2l|_2)|w^{-\frac\nu2}l_t|_2\\
						&+C|w^{-1}|_\infty|w_t|_\infty|h^{\frac12}\nabla l|_2|h^\frac12\nabla l_t|_2,\\
						I_3=&a_5\int\big(w^\nu ng^2H(v)\big)_tw^{-\nu}l_t\\
						\le&C\big((|w^{-1+\frac\nu2}|_\infty|n|_\infty|w_t|_\infty+|w^\frac\nu2|_\infty|n_t|_\infty)|g\nabla v|_4^2\\
						&+|w^\frac\nu2|_\infty|n|_\infty(|g^{-\frac12}g_t|_\infty|g^\frac34\nabla v|_4^2+|g\nabla v_t|_4|g\nabla v|_4)\big)|w^{-\frac\nu2}l_t|_2,\\
						I_4=&a_6\int(w^{\nu+1}\text{div}\psi)_tw^{-\nu}l_t
				\le C\big(|w^\frac\nu2|_\infty|\nabla\psi|_4|w_t|_4+|w^{1+\frac\nu2}|_\infty|\nabla\psi_t|_2\big)|w^{-\frac\nu2}l_t|_2,
						\\
						I_5=&\int\Pi(l,h,w,g)_tw^{-\nu}l_t\\\le&C\Big(|w^{1+\frac\nu2}|_\infty(|\varphi|^{\frac{1}{2}}_\infty|h^{-\frac12}h_t|_\infty|\nabla h^\frac12|_4^2+|\varphi|^\frac12_\infty|\nabla h^\frac12|_4|\psi_t|_4)\\&+|w^\frac\nu2|_\infty|w_t|_\infty|\nabla h^\frac12|_4^2+
		|w^{-2+\frac\nu2}|_\infty|g^\frac12\nabla w|_4^2|w_t|_\infty\\&+|w^{-1+\frac\nu2}|_\infty|\nabla l|_2|w_t|_\infty|\psi|_\infty+|w^\frac\nu2|_\infty(|\varphi|_\infty^\frac12|\psi|_\infty|h^\frac12\nabla l_t|_2+|\psi_t|_4|\nabla l|_4)\\&+|w^{-1+\frac\nu2}|_\infty(|g^{-\frac12}g_t|_\infty|g^\frac12\nabla w|_4|\nabla w|_4+|g^\frac12\nabla w|_4|g^\frac12\nabla w_t|_4)\Big)|w^{-\frac\nu2}l_t|_2,\\
						I_6=&-a_4\int\big(\nabla h\cdot\nabla l_tl_t+\frac12(w^{-\nu})_t|l_t|^2\big)
						\\\le&	C|w^\frac\nu2|_\infty|\varphi|^{\frac12}_\infty|\psi|_\infty|h^\frac12\nabla l_t|_2|w^{-\frac\nu2}l_t|_2+C|w^{-1}|_\infty|w_t|_\infty|w^{-\frac\nu2}l_t|_2^2.
					\end{align*}

				Integrating \ef{808} over $(\tau,t)\big(\tau\in(0,t)\big)$, by \ef{801}, Lemmas \ref{gh}-\ref{l1}, \ef{701}, \ef{807}  and 
$|w_t|_4\le|w_t|_2^2|\nabla w_t|_2^2\le M(c_0)c_2$,
	one has
				\begin{equation}\label{811}
					\begin{aligned}
						&|w^{-\frac\nu2}l_t(t)|_2^2+\int_\tau^t|h^\frac12\nabla l_s|_2^2\text{d}s\\\le& M(c_0)c_4^{6\nu+14}\int_\tau^t|w^{-\frac\nu2}l_s|_2^2\text{d}s+M(c_0)+|w^{-\frac\nu2}l_t(\tau)|_2^2,
					\end{aligned}
				\end{equation} 
				for $0\le t\le T_4$.
%
               %
				Letting $\tau\rightarrow0$ in \ef{811}, then Lemma \ref{inin} and Gronwall's inequality imply  
				\begin{equation}\label{702}
					|w^{-\frac{\nu}{2}}l_t(t)|_2^2+\int_0^t(|h^\frac12\nabla l_s|_2^2+|\nabla l_s|_2^2)\text{d}s\le M(c_0)\quad \text{for}\ \ 0\le t\le T_5.
				\end{equation}
				Then, \eqref{702}-\ef{807} give  
				\begin{equation}\label{812}
    \begin{aligned}
	|hl(t)|_{D^2}&+|h\nabla^2l(t)|_2+|l(t)|_{D^2}\le M(c_0)c_1^{2\nu+1},\\
	|l-\bar{l}|_2\le& |l_0-\bar{l}|_2+\left(\int_0^t|l_s|_2^2\text{d}s\right)^\frac12t^\frac12\le 
						M(c_0),\\
						|l-\bar{l}|_\infty \le& |l-\bar{l}|_2^\frac14|\nabla l|_2^\frac12|\nabla^2 l|_2^\frac14\le M(c_0)c_1^{\frac\nu2+\frac14}.
					\end{aligned}
				\end{equation}
                
%
		Furthermore, Lemmas \ref{gh}-\ref{l1}, \ef{2.14}, \ef{801} ,\ef{g32} and the results above imply
				\begin{equation}\label{813}
					\begin{aligned}
						|\mathcal{A}|_{D^1} \leq& C\big(|\nabla l_t|_2+|v|_{\infty}|\nabla^2 l|_2+|\nabla v|_4|\nabla l|_4+|w^{\nu-1}|_{\infty}|n|_{\infty}|\nabla w|_6|g\nabla v|_6^2 \\
						& +|w^\nu|_{\infty}(|\nabla n|_{\infty}|g\nabla v|_4^2 +|n|_{\infty}|\nabla g|_{\infty}|g\nabla v|_4 |\nabla v|_4 \\
						& +|n|_{\infty}|g^2\nabla v\cdot\nabla^2 v|_2+|\nabla w|_4|\nabla \psi|_4) +|w^{\nu+1}|_{\infty}|\nabla^2 \psi|_2   \\
						&+|w^\nu|_{\infty}|\varphi|_{\infty}|\psi|_{\infty}^2|\nabla w|_2+|w^{\nu+1}|_{\infty}(|\varphi|_{\infty}|\psi|_{\infty}|\nabla h^{\frac{1}{2}}|_4^2 \\
						& +|\varphi|_{\infty}|\psi|_{\infty}|\nabla\psi|_2)+|w^{\nu-1}|_{\infty}|\psi|_{\infty}|\nabla l|_4|\nabla w|_4 \\
						& +|w^\nu|_{\infty}(|\psi|_{\infty}|\nabla^2 l|_2 +|\nabla \psi|_4|\nabla l|_4)+|w^{\nu-2}|_{\infty}|g^\frac12 \nabla w|^2_6|\nabla w|_6 \\
						& +|w^{\nu-1}|_{\infty}(|\nabla g|_{\infty}|\nabla w|_4^2+|g\nabla w\cdot\nabla^2 w|_2)\big) \\
						\leq&  M\left(c_0\right)\left(\left|\nabla l_t\right|_2+c_1^{3 \nu+2}\right), \\
						|F|_{D^1} \leq& C\big(|\psi|_{\infty}|\nabla^2 l|_2+|\nabla\psi|_4|\nabla l|_4
						+|\nabla^3 h|_2 | l-\bar{l}|_{\infty}\big) \leq M(c_0) c_1^{2\nu+1},\\
					\end{aligned}
				\end{equation}
				where we also used the   estimates
					\begin{align*}
     |g\nabla v|_\infty\le&|g\nabla v|_2^\frac14|\nabla 
     (g\nabla v)|_2^\frac12|\nabla^2(g\nabla v)|_2^\frac14\le M(c_0)c_4,\\
     |g^\frac12\nabla w|_\infty\le&|g^\frac12\nabla w|_2^\frac14|\nabla 
     (g^\frac12\nabla w)|_2^\frac12|\nabla^2(g^\frac12\nabla w)|_2^\frac14\le M(c_0)c_2,\\
     \|w\|_{D^1\cap D^2}\le&\|l_0\|_{D^1\cap D^2}+t^\frac12\left(\int_0^t\|w_s\|^2_{D^1\cap D^2}\text{d}s\right)^\frac12\le M(c_0),\\
		|g\nabla v|_6\le& |g(0,x)\nabla u_0|_6+t^\frac12\left(\int_0^t|(g\nabla v)_s|_6^2\text{d}s\right)^\frac12\\
						\le& C|h_0\nabla u_0|_2^\frac13|\nabla(h_0 u_0\big)|_2^\frac23+Ct^\frac12\Big(\int_0^t(|g_s|_\infty^2|\nabla v|_6^2+|g\nabla v_s|_6^2)\text{d}s\Big)^
						\frac12\\
						\le& M(c_0)+t^\frac12\Big(\int_0^t|g\nabla v_s|_2^\frac23|\nabla(g\nabla v_s)|_2^\frac43\text{d}s\Big)^\frac12\le M(c_0),\\
						|g^\frac12\nabla w|_6\le& |g^\frac12(0,x)\nabla l_0|_6+t^\frac12\big(\int_0^t|(g^\frac12\nabla w)_s|_6^2\text{d}s\big)^\frac12\\
						\le& C|h_0^\frac12\nabla l_0 |_2^\frac13|\nabla(h_0^\frac12\nabla l_0 )|_2^\frac23\\&+Ct^\frac12\Big(\int_0^t\big(|g^{-\frac12}g_s|^2_\infty|\nabla w|_6^2+|g^\frac12\nabla w_s|_6^2\big)\text{d}s\Big)^\frac12\\
						\le& M(c_0)+Ct^\frac12\Big(\int_0^t|g^\frac12\nabla w_s|^\frac23_2|\nabla(g^\frac12\nabla w_s)|^\frac43_2\big)\text{d}s\Big)^\frac12\le M(c_0),\\
						|g^2 \nabla v\cdot\nabla^2 v|_2\le&|g^2(0,x)\nabla u_0\cdot\nabla^2 u_0|_2+t^\frac12\left(\int_0^t|(g^2\nabla v\cdot \nabla^2 v)_s|_2^2\text{d}s\right)^\frac12\\
						\le& C|h_0\nabla u_0|_4|h_0\nabla ^2u_0|_4+Ct^\frac12\Big(\int_0^t(|g_s|_\infty^2|g\nabla v|_4^2|\nabla^2 v|_4^2\\&+|g\nabla v_s|_4^2|g\nabla^2 v|_4^2+|g\nabla^2 v_s|_2^2|g\nabla v|_\infty^2)\text{d}s\Big)^
						\frac12\le M(c_0),\\
						|g \nabla w\cdot\nabla^2 w|_2\le&|g(0,x)\nabla l_0\cdot\nabla^2 l_0|_2+t^\frac12\left(\int_0^t|(g\nabla w\cdot \nabla^2 w)_s|_2^2\text{d}s\right)^\frac12\\
						\le& |h_0^\frac12\nabla l_0|_4|h_0^\frac12\nabla^2 l_0|_4+Ct^\frac12\Big(\int_0^t(|g_s|_\infty^2|\nabla w|_4^2|\nabla^2 w|_4^2\\&+|g^\frac12\nabla w_s|_4^2|g^\frac12\nabla^2 w|_4^2+|g^\frac12\nabla^2 w_s|_2^2|g^\frac12\nabla w|_\infty^2)\text{d}s\Big)^
						\frac12
						\le  M(c_0),
					\end{align*}
				for $0\le t\le T_5$. 
					 Then \ef{805}, \ef{812}-\ef{813} and Lemmas \ref{ell}, \ref{gh}-\ref{l1} give that
				\begin{equation}\label{814}
					\begin{aligned}
						|h(l-\bar{l})(t)|_{D^3}  \leq& C(|w^{-\nu} \mathcal{A}|_{D ^1}+|F(\nabla h, l-\bar{l})|_{D^1}) \\
						\leq&  C(|w^{-\nu}|_{\infty}|\mathcal{A}|_{D^1}+|\nabla w^{-\nu}|_4|\mathcal{A}|_4+|F|_{D^1}) \\
						\leq&  M(c_0)(c_1^{\nu+1}|\nabla l_t|_2+c_1^{4 \nu+3}), \\
						|h \nabla^3 l(t)|_2 \leq&  C\big(|h(l-\bar{l})|_{D^3}+|\nabla^3 h|_2|l-\bar{l}|_\infty+|\nabla \psi|_4|\nabla l|_4+|\psi|_\infty|\nabla^2l|_2\big) \\
						\leq&  M(c_0)(c_1^{\nu+1}|\nabla l_t|_2+c_1^{4 \nu+3}).
					\end{aligned}
				\end{equation}
				
				Finally, one gets from \ef{702}, \ef{814} and Lemma \ref{gh} that
				\begin{equation*}
					\begin{aligned}
						\int_0^t(|h\nabla^3 l|_2^2+|h\nabla^2 l|_{D^1}^2+|l|_{D^3}^2)\text{d}s\le M(c_0)c_1^{2\nu+2},
					\end{aligned}
				\end{equation*}
				for $0\le t\le T_5$. 
				The proof of Lemma \ref{l2} is completed.

			\end{proof}
			
			\begin{lemma}\label{l3}
				For $T_6=\min \left\{T_5,\left(1+C c_4\right)^{-40-10 \nu}\right\}$ and $t \in\left[0, T_6\right]$, one has
				\begin{equation}\label{704}
					\begin{aligned}
						|h^{\frac{1}{2}} \nabla l_t|_2^2+|\nabla l_t|_2^2+\int_0^t|w^{-\frac{\nu}{2}}  l_{s s}|_2^2 {\rm{d}}s & \leq M\left(c_0\right), \\
						|h \nabla^3 l|_2+|h \nabla^2 l|_{D^1}+|l|_{D^3} & \leq M\left(c_0\right) c_1^{4 \nu+3}, \\
						\int_0^t(|h^\frac12 \nabla^2 l_s|_2^2+|\nabla^2 l_s|_2^2) {\rm{d}}s & \leq M(c_0) c_1^{3 \nu} .
					\end{aligned}
				\end{equation}
				
			\end{lemma}
			
			\begin{proof}
				First, 
				multiplying \ef{703} by $w^{-\nu} l_{t t}$ and integrating over $\mathbb{R}^2$, one has
				\begin{equation}\label{815}
					\frac{a_4}{2} \frac{\text{d}}{\text{d} t}|h^{\frac{1}{2}} \nabla l_t|_2^2+|w^{-\frac{\nu}{2}}  l_{t t}|_2^2=\sum_{i=7}^{12} I_i, 
				\end{equation}
				where 
				\begin{equation}\label{816}
					\begin{aligned}
						I_7=&-\int( v \cdot \nabla l)_t w^{-\nu} l_{t t} \\
						\leq &C |w^{-\frac{\nu}{2}}|_{\infty}(\left|v_t\right|_4|\nabla l|_4+|\varphi|_{\infty}^{\frac{1}{2}}|v|_{\infty}|h^{\frac{1}{2}} \nabla l_t|_2)|w^{-\frac{\nu}{2}}  l_{t t}|_2, \\
						I_8=&a_4 \int(w_t^\nu h \Delta l+w^\nu h_t \Delta l) w^{-\nu} l_{t t} \\
						\leq &C(|w^{-1+\frac\nu2}|_\infty|w_t|_\infty|h\nabla^2 l|_2
						 +|w^{\frac{\nu}{2}}|_{\infty}|\varphi|_{\infty}^{\frac{1}{2}}|h^{-\frac12}h_t|_{\infty}|h \nabla^2 l|_2)|w^{-\frac{\nu}{2}}  l_{t t}|_2,\\
						I_9=&a_5 \int\big(w^\nu n g^{2} H(v)\big)_t w^{-\nu} l_{t t} \\
						\leq& C\big((|w^{-1+\frac{\nu}{2}}|_{\infty}|n|_{\infty}|w_t|_\infty+|w^{\frac{\nu}{2}}|_{\infty}|n_t|_\infty)|g \nabla v|_4^2 \\
						& +|w^{\frac{\nu}{2}}|_{\infty}|n|_{\infty}|g^{-\frac12}g_t|_\infty|g^\frac34\nabla v|_4^2+|w^{\frac{\nu}{2}}|_{\infty}|n|_\infty|g\nabla v|_4|g \nabla v_t|_4\big)|w^{-\frac{\nu}{2}}  l_{t t}|_2, \\
						I_{10}=&a_6 \int(w^{\nu+1}  \text{div} \psi)_t w^{-\nu} l_{t t} \\
						\leq& C(|w^{\frac{\nu}{2}}|_{\infty}|\nabla \psi|_4|w_t|_4+|w^{1+\frac{\nu}{2}}|_{\infty}|\nabla \psi_t|_2)|w^{-\frac{\nu}{2}}  l_{t t}|_2, \\
						I_{11}=&\int \Pi(l, h, w, g)_t w^{-\nu} l_{t t} \\
						\leq &C\big(|w^\frac\nu2|_\infty|w_t|_\infty|\nabla h^\frac12|_4^2\\&+|w^{1+\frac\nu2}|_\infty(|\varphi|^{\frac{1}{2}}_\infty|h^{-\frac12}h_t|_\infty|\nabla h^\frac12|_4^2+|\varphi|^\frac12_\infty|\nabla h^\frac12|_4|\psi_t|_4)\big)|w^{-\frac\nu2}l_t|_2\\&+C\big(|w^{-1+\frac\nu2}|_\infty|\nabla l|_2|w_t|_\infty|\psi|_\infty+|w^\frac\nu2|_\infty(|\varphi|_\infty^\frac12|\psi|_\infty|h^\frac12\nabla l_t|_2\\&+|\psi_t|_4|\nabla l|_4)+
						|w^{-2+\frac\nu2}|_\infty|g^\frac12\nabla w|_4^2|w_t|_\infty\\&+|w^{-1+\frac\nu2}|_\infty(|g^{-\frac12}g_t|_\infty|g^\frac12\nabla w|_4|\nabla w|_4+|g^\frac12\nabla w|_4|g^\frac12\nabla w_t|_4)\big)|w^{-\frac{\nu}{2}}  l_{t t}|_2, \\
						I_{12}=&  -a_4 \int \nabla h \cdot \nabla l_t l_{t t}+\frac{a_4}{2} \int h_t|\nabla l_t|^2 \\
						\leq&  C|w^{\frac{\nu}{2}}|_{\infty}|\varphi|_{\infty}^{\frac{1}{2}}|\psi|_{\infty}|h^{\frac{1}{2}} \nabla l_t|_2|w^{-\frac{\nu}{2}}  l_{t t}|_2+C|\varphi|^\frac12_{\infty}|h^{-\frac12}h_t|_{\infty}|h^{\frac{1}{2}} \nabla l_t|_2^2 .
					\end{aligned}
				\end{equation}
				Integrating \ef{815} over $(\tau, t)$, for $0 \leq t \leq T_5$, \eqref{815}-\ef{816} imply
				\begin{equation}\label{823}
					\begin{aligned}
						 |h^{\frac{1}{2}} \nabla l_t(t)|_2^2+\int_\tau^t|w^{-\frac{\nu}{2}}  l_{s s}|_2^2 \text{d} s 
						\leq & M(c_0) + C|h^{\frac{1}{2}} \nabla l_t(\tau)|_2^2 \\&+M(c_0) c_4^{10\nu+30} \int_0^t|h^{\frac{1}{2}} \nabla l_s|_2^2 \text{d} s .
					\end{aligned}
				\end{equation}
				Letting $\tau \rightarrow 0$ in \eqref{823}, then   Lemma \ref{inin} and Gronwall's inequality imply 
				\begin{equation}\label{818}
					\begin{aligned}
						& |h^{\frac{1}{2}} \nabla l_t(t)|_2^2+|\nabla l_t(t)|_2^2+\int_0^t|w^{-\frac{\nu}{2}}  l_{s s}|_2^2 \text{d} s \leq M(c_0),
					\end{aligned}
				\end{equation}
				 for $0 \leq t \leq T_6$.
				and \eqref{818}-\ef{814} give
				\begin{equation*}
					|h(l-\bar{l})|_{D^3}+|h \nabla^3 l|_2+|h \nabla^2 l|_{D^1}+|\nabla^3 l|_2 \leq M(c_0) c_1^{4 \nu+3} .
				\end{equation*}

				Furthermore,  \ef{703} gives
				\begin{equation}\label{819}
					\begin{aligned}
						-a_4 \Delta(h^\frac12 l_t) & =-a_4h^\frac12 \Delta l_t-a_4 F(\nabla h^\frac12, l_t)   =w^{-\nu} \mathcal{B}-a_4 F(\nabla h^\frac12, l_t),
					\end{aligned}
				\end{equation}
				with
				\begin{equation}
					\begin{aligned}
						\mathcal{B}= & - h^{-\frac12}l_{t t}-h^{-\frac12}( v \cdot \nabla l)_t+a_4 (w^\nu)_t h^\frac12\Delta l++a_4 w^\nu h^{-\frac12}h_t\Delta l \\
						& +a_5h^{-\frac12}\big(w^\nu n g^{2} H(v)\big)_t+a_6h^{-\frac12}(w^{\nu+1}  \text{div} \psi)_t+h^{-\frac12}\Pi(l, h, w, g)_t,\\
                        \hat{F}=&F(\nabla h^\frac12, l_t).
					\end{aligned}
				\end{equation}
				%
			%
				Then \ef{801} and Lemmas \ref{gh}-\ref{l2} imply 
				$$
				\begin{aligned}
					|\mathcal{B}|_2 \leq & C\Big(|\varphi|_{\infty}^{\frac{1}{2}}| l_{t t}|_2 +\|\nabla l \|_1|\varphi|_{\infty}^{\frac{1}{2}}|v_t|_4\\
					&+| \varphi|_{\infty} |v|_{\infty}|h^{\frac{1}{2}} \nabla l_t|_2  +|(w^\nu)_t|_4|h^{\frac{1}{2}} \Delta l|_4+|w^\nu|_{\infty}|(h^{\frac{1}{2}})_t \Delta l|_2 \\
					& +|\varphi|_\infty^\frac12\big(|w^\nu|_{\infty}|n_t|_{\infty}|g \nabla v|_4^2+|w^{\nu-1}|_{\infty}|n|_{\infty}|w_t|_4|g \nabla v|_{\infty}|g \nabla v|_4 \\
					& +|w^\nu|_{\infty}|n|_{\infty}(|g \nabla v|_2|g_t|_{\infty}|\nabla v|_{\infty}+|g \nabla v|_{\infty}|g \nabla v_t|_2)\big) \\
					& +|w^\nu|_{\infty}|\varphi|_{\infty}^{\frac{1}{2}}|w_t|_4|\nabla \psi|_4+|w^{\nu+1}|_{\infty}|\varphi|_{\infty}^{\frac{1}{2}}|\nabla \psi_t|_2 \\
					& +|w^\nu|_{\infty}|\varphi|_{\infty}|w_t|_4|\nabla h^{\frac{1}{2}}|_4|\nabla h|_\infty+|w^{1+\nu}|_{\infty}(|\varphi|_{\infty}^{\frac{3}{2}}|h_t|_{\infty}|\nabla h^{\frac{1}{2}}|_4^2 \\
					& +|\varphi|_{\infty}^{\frac{3}{2}}|\psi|_{\infty}|\psi_t|_2)+|w^{-1+\nu}|_{\infty}|\nabla l|_4|w_t|_4|\varphi|_{\infty}^{\frac{1}{2}}|\psi|_{\infty} \\
					& +|w^\nu|_{\infty}(|\varphi|_{\infty}^{\frac{1}{2}}|\psi|_{\infty}|\nabla l_t|_2+|\varphi|_{\infty}^{\frac{1}{2}}|\psi_t|_2|\nabla l|_{\infty}) \\
					& +|\varphi|_\infty^\frac12(|g^\frac12 \nabla w|_{4}^2|w^{-2+\nu}|_{\infty}|w_t|_\infty+|w^{-1+\nu}|_{\infty}|g_t|_{\infty}|\nabla w|_4^2) \\
					& 
     +|gh^{-1}|_\infty^\frac12|w^{-1+\nu}|_{\infty}|g^\frac12\nabla w_t|_2| \nabla w|_{\infty}\Big), \\
					|\hat{F}|_2 \leq & C\big(|\varphi|_{\infty}^{\frac{3}{2}}|\psi|_{\infty}^2| l_t|_2+|\varphi|_{\infty}^{\frac{1}{2}}(|l_t|_4|\nabla \psi|_4+|\psi|_{\infty}|\nabla l_t|_2)\big),
				\end{aligned}
				$$
                which, together with \ef{703}, \ef{818}-\ef{819} and Lemmas \ref{ell}, \ref{gh}-\ref{l2}, gives
				\begin{equation}\label{820}
					\begin{aligned}
						|h^\frac12 l_t|_{D^2} & \leq M(c_0)(c_1^\nu| l_{t t}|_2+c_4^{5 \nu+10}), \\
						|h^\frac12 \nabla^2 l_t|_2 \leq & M(c_0)(|h l_t|_{D^2}+|\varphi|_{\infty}^{\frac{3}{2}}|\psi|_{\infty}^2| l_t|_2 \\
						& +|\varphi|_{\infty}^{\frac{1}{2}}|l_t|_4|\nabla \psi|_4+|\varphi|_{\infty}^{\frac{1}{2}}|\psi|_{\infty}|\nabla l_{t}|_2) \\
						\leq & M(c_0)(c_1^\nu|l_{tt}|_2+c_4^{5 \nu+10}) .
					\end{aligned}
				\end{equation}   
				Moreover, \ef{818} and \ef{820} implies $\ef{704}_3$,
			and
				the proof of Lemma \ref{l3} is complete.
                
			\end{proof}


			\begin{lemma}\label{lt}
				For $T_7=\min \{T_6,\big(1+M(c_0) c_5\big)^{-40-10 \nu}\}$ and $t \in\left[0, T_7\right]$,  one has
				\begin{equation}
					\begin{aligned}
						t^{\frac{1}{2}}|l_t(t)|_{D^2}+t^{\frac{1}{2}}|h^\frac12 \nabla^2 l_t(t)|_2+t^{\frac{1}{2}}|h^{-\frac{1}{4}} l_{t t}(t)|_2 \leq M(c_0) c_1^{\frac{\nu}{2}}, \\
						\int_0^t s(|l_{s s}|_{D^1}^2+|h^{\frac{1}{4}} l_{s s}|_{D^1}^2) {\rm{d}} s \leq M(c_0), \quad
						\frac{1}{2} c_0^{-1} \leq l \leq \frac{3}{2} c_0. 
					\end{aligned}
				\end{equation}
				
			\end{lemma} 
			\begin{proof}
				First, applying $\partial_t$ to \ef{703}, one has
				\begin{equation}\label{821}
					\begin{aligned}
						& l_{t t t}-a_4 w^\nu h \Delta l_{t t}+( v \cdot \nabla l)_{t t} \\
						= & 2 a_4(w^\nu h)_t \Delta l_t+2 a_4(w^\nu)_th_t \Delta l 
						+a_4(w^\nu)_{t t}h \Delta l+a_4 w^\nu h_{tt} \Delta l \\
						& +a_5\big(w^\nu n g^{2} H(v)\big)_{tt}+a_6(w^{\nu+1}  \text{div} \psi)_{tt}+\Pi(l, h, w, g)_{t t} .
					\end{aligned}
				\end{equation}   
				Multiplying \ef{821} by $w^{-\nu}h^{-\frac12} l_{t t}$ and integrating over $\mathbb{R}^2$, one has 
				\begin{equation}\allowdisplaybreaks[4]\label{822}
					\begin{aligned}
						& \frac{1}{2} \frac{\text{d}}{\text{d} t}|w^{-\frac{\nu}{2}} h^{-\frac{1}{4}} l_{t t}|_2^2+a_4|h^{\frac{1}{4}} \nabla l_{t t}|_2^2 \\
						= & -\int( v \cdot \nabla l)_{t t}w^{-\nu}h^{-\frac{1}{2}} l_{t t} +\int\big(2 a_4(w^\nu h)_t \Delta l_t+2 a_4(w^\nu)_th_t \Delta l\big) w^{-\nu} h^{-\frac{1}{2}}l_{t t} \\
						& +\int a_4(w^\nu)_{t t}h^{\frac{1}{2}} \Delta l w^{-\nu} l_{t t}+\Big(a_4 w^\nu h_{t t} \Delta l+a_5\big(w^\nu n g^{2} H(v)\big)_{tt}\Big) w^{-\nu}h^{-\frac{1}{2}} l_{t t} \\
						& +\int\big(a_6(w^{\nu+1}  \text{div} \psi)_{t t}+a_7(w^{\nu+1} h^{-1} \psi \cdot \psi)_{t t}+ a_8(w^\nu  \nabla l \cdot \psi)_{t t} \\&+a_9 (w^{\nu-1} g \nabla w \cdot \nabla w)_{t t}\big) w^{-\nu}h^{-\frac{1}{2}} l_{t t}\\
						&-a_4 \int \nabla h^{\frac{1}{2}} \cdot \nabla l_{t t} l_{t t}+\frac{1}{2} \int(w^{-\nu} h^{-\frac{1}{2}})_t|l_{t t}|^2=\sum_{i=13}^{20} I_i,
					\end{aligned}
				\end{equation}
				where  
			\begin{align*}
						 I_{13}=&-\int( v \cdot \nabla l)_{t t}w^{-\nu}h^{-\frac{1}{2}} l_{t t}\\
						 \leq& C|w^{-\frac\nu2}|_\infty|\varphi|^{\frac{1}{4}}_{\infty}(|\nabla l|_\infty|v_{t t}|_2+|\nabla l_t|_4|v_t|_4)|w^{-\frac{\nu}{2}} h^{-\frac{1}{4}} l_{t t}|_2  \\&+C|\varphi|_{\infty}^{\frac{1}{2}}|w^{-\frac{\nu}{2}}|_{\infty}|v|_{\infty}|h^{\frac{1}{4}} \nabla l_{t t}|_2|w^{-\frac{\nu}{2}} h^{-\frac{1}{4}} l_{t t}|_2, \\
						 I_{14}=&\int\Big(\big(2 a_4(w^\nu h)_t \Delta l_t+2 a_4(w^\nu)_th_t \Delta l\big) w^{-\nu} h^{-\frac{1}{2}}l_{t t} \\
						&+a_4(w^\nu)_{t t}h^{\frac{1}{2}} \Delta l w^{-\nu} l_{t t}+a_4  h_{t t} \Delta l h^{-\frac{1}{2}}l_{t t}\Big) \\
						 \leq &C\big(|\varphi|_{\infty}^{\frac{3}{4}}|h_t|_{\infty}|w^{\frac{\nu}{2}}|_{\infty}|h^\frac12 \nabla^2 l_t|_2+|\varphi|_\infty^\frac14|w^{\frac{\nu}{2}-1}|_{\infty}|w_t|_4|h_t|_{\infty}|\nabla^2 l|_4 \\
						& +|\varphi|_\infty^\frac14|w^{\frac{\nu}{2}}|_{\infty}|h_{t t}|_4|\nabla^2 l|_4 +|w^{\frac{\nu}{2}-2}|_{\infty}|w_t|_\infty^2|h^\frac34 \nabla^2 l|_2\big)|w^{-\frac{\nu}{2}} h^{-\frac{1}{4}} l_{t t}|_2 \\
						& +C|w^{-1}|_{\infty}(|g h^{-1}|_{\infty}^{\frac{1}{4}}|g^{-\frac{1}{4}} w_{t t}|_2|h \nabla^2 l|_4 +|w_t|_4|h^\frac12\nabla l_t|_2)| l_{t t}|_4,\\
						 I_{15}=&\int a_5\big(w^\nu n g^{2} H(v)\big)_{t t} w^{-\nu} h^{-\frac12}l_{t t} \\
						 \leq& C\big(|n|_{\infty}|\varphi|_{\infty}^{\frac{1}{4}}|g \nabla v|_{\infty}^2(|w^{\frac{\nu}{2}-2}|_{\infty}|w_t|_4^2+|w^{\frac{\nu}{2}-1}|_{\infty}| w_{t t}|_2) \\
						& +|w^{\frac{\nu}{2}}|_{\infty}|\varphi|_{\infty}^{\frac{1}{4}}(|n_{t t}|_2|g \nabla v|_{\infty}^2+|n|_{\infty}|g_t|_{\infty}^2|\nabla v|_4^2) \\
						& +|w^{\frac{\nu}{2}}|_{\infty}|\varphi|_\infty^\frac14|g \nabla v|_4|\nabla v|_\infty|n|_{\infty}|g_{t t}|_4 +|\varphi|_{\infty}^{\frac{1}{4}}|w^{\frac{\nu}{2}-1}|_{\infty}|w_t|_4|n_t|_4|g \nabla v|_{\infty}^2 \\
						& +|w^{\frac{\nu}{2}-1}|_{\infty}| w_t|_4|n|_{\infty}|\varphi|_{\infty}^{\frac{1}{4}}|g_t|_{\infty}| \nabla v|_{\infty}|g\nabla v|_{4} \\
						& +|w^{\frac{\nu}{2}}|_{\infty}|n_t|_2|\varphi|_{\infty}^{\frac{1}{4}}|g_t|_{\infty}|g \nabla v|_{\infty}|\nabla v|_{\infty}\big)|w^{-\frac{\nu}{2}} h^{-\frac{1}{4}} l_{t t}|_2
      +\int n g^{2} H^t(v) h^{-\frac12}l_{t t} \\
						& + C|\varphi|_\infty^\frac14|w^\frac\nu2|_\infty\big(| n | _ { \infty } |g \nabla v_t|_4^2+|g \nabla v_t|_4(| n | _ { \infty } |w^{-1}|_{\infty}| w_t|_4|g \nabla v|_{\infty} \\
						& +|n_t|_4|g \nabla v|_{\infty}+|n|_{\infty}|g_t|_{\infty}|\nabla v|_4)\big)|w^{-\frac\nu2}h^{-\frac14}l_{t t}|_2, \\
						 I_{16}=&\int a_6(w^{\nu+1}  \text{div} \psi)_{t t} w^{-\nu} h^{-\frac{1}{2}}l_{t t} \\
						 \leq&\int a_6 w  \text{div} \psi_{t t} h^{-\frac{1}{2}}l_{t t} +C|w^{\frac{\nu}{2}-1}|_{\infty}|\varphi|_{\infty}^{\frac{1}{4}}|w_t|_4|w_t|_\infty|\nabla \psi|_4 |w^{-\frac{\nu}{2}} h^{-\frac{1}{4}} l_{t t}|_2 \\
						& +C\big(|g h^{-1}|_{\infty}^{\frac{1}{4}}|\varphi|_{\infty}^{\frac{1}{4}}|g^{-\frac{1}{4}} w_{t t}|_2|\nabla \psi|_4+|\varphi|_{\infty}^{\frac{1}{2}}|w_t|_4|\nabla \psi_t|_2\big)|l_{t t}|_4, \\
						 I_{17}=&a_7 \int(w^{\nu+1} h^{-1} \psi \cdot \psi)_{t t} w^{-\nu}h^{-\frac12} l_{t t} \\
						 \leq& C\big(|\psi|_{\infty}^2(|g h^{-1}|_{\infty}^{\frac{1}{4}}|w^{\frac{\nu}{2}}|_{\infty}|g^{-\frac{1}{4}} w_{t t}|_2|\varphi|_{\infty}+|w^{\frac{\nu}{2}-1}|_{\infty}|\varphi|^\frac54_{\infty}|w_t|_4^2) \\
						& +|w^{\frac{\nu}{2}+1}|_{\infty}(|\varphi|_{\infty}^{\frac{7}{4}}|\psi|_\infty|\nabla h^\frac12|_4|h_{t t}|_4+|\varphi|^\frac94_{\infty}|h_t|_{\infty}^2|\nabla h^{\frac{1}{2}}|_4^2) \\
						& +|\varphi|_{\infty}^{\frac{5}{4}}|w^{\frac{\nu}{2}+1}|_{\infty}(|\psi_t|_4^2+|\psi|_{\infty}|\psi_{t t}|_2) \\
						& +|w^{\frac{\nu}{2}}|_{\infty}(|w_t|_4|\varphi|_{\infty}^{\frac{5}{4}}|\psi|_{\infty}|\psi_t|_4+| w_t|_2|\varphi|_{\infty}^\frac94|h_t|_{\infty}|\psi|_{\infty}^2) \\
						& +|w^{\frac{\nu}{2}+1}|_{\infty}|\varphi|_{\infty}^{\frac{9}{4}}|h_t|_{\infty}|\psi|_{\infty}|\psi_t|_2\big)|w^{-\frac{\nu}{2}} h^{-\frac{1}{4}} l_{t t}|_2, \\
						 I_{18}=&a_8 \int(w^\nu \nabla l \cdot \psi)_{t t} w^{-\nu}  h^{-\frac{1}{2}}l_{t t} \\
						 \leq& C\big(|\nabla l|_{\infty}|\psi|_{\infty}(|g h^{-1}|_{\infty}^{\frac{1}{4}}|w^{\frac{\nu}{2}-1}|_{\infty}|g^{-\frac{1}{4}} w_{t t}|_2+|\varphi|_\infty^\frac14|w^{\frac{\nu}{2}-2}|_{\infty}|w_t|_4^2) \\
						& +|w^{\frac{\nu}{2}}|_{\infty}(|\varphi|_{\infty}^{\frac{1}{2}}|\psi|_{\infty}|h^{\frac{1}{4}} \nabla l_{t t}|_2+|\varphi|_{\infty}^{\frac{1}{4}}|\nabla l|_{\infty}|\psi_{t t}|_2) \\
						& +|w^{\frac{\nu}{2}-1}|_{\infty}| w_t|_4|\varphi|_{\infty}^{\frac{1}{4}}|\nabla l|_{\infty}|\psi_t|_4 \\
						& +|\varphi|_{\infty}|h^{\frac{1}{2}} \nabla l_t|_2(|w^{-1}|_{\infty}| w_t|_4|\psi|_{\infty} +|\psi_t|_4)\big)| l_{t t}|_4, \\
						 I_{19}=&a_9 \int(w^{\nu-1} g \nabla w \cdot \nabla w)_{t t} w^{-\nu}h^{-\frac12} l_{t t} \\
						 \leq& C\big(|g^\frac12 \nabla w|_{\infty}^2(|h g^{-1}|_{\infty}^{\frac{1}{4}}|w^{\frac{\nu}{2}-2}|_{\infty}|g^{-\frac{1}{4}} w_{t t}|_2+|\varphi|_\infty^\frac14|w^{\frac{\nu}{2}-3}|_{\infty}| w_t|_4^2) \\
						& +|w^{\frac{\nu}{2}-1}|_{\infty}|\varphi|_{\infty}^{\frac{1}{4}}|g_{t t}|_4|\nabla w|_4|\nabla w|_\infty  +|w^{\frac{\nu}{2}-2}|_{\infty}|\varphi|_{\infty}^{\frac{1}{4}}(| w_t|_2|g_t|_{\infty}|\nabla w|_{\infty}^2\\
						&+| w_t|_\infty|g^\frac12\nabla w_t|_4|g^\frac12 \nabla w|_4)  +|w^{\frac{\nu}{2}-1}|_{\infty}|\varphi|_{\infty}^{\frac{1}{4}}|g_t|_{\infty}|\nabla w|_{\infty}| \nabla w_t|_2\\
						&+|w^{\frac\nu2-1}|_{\infty}|gh^{-1}|_\infty^\frac14|g^\frac12 \nabla w|_{\infty}|g^{\frac{1}{4}} \nabla w_{t t}|_2\big)|w^{-\frac{\nu}{2}} h^{-\frac{1}{4}} l_{t t}|_2 \\
						& +C|gh^{-1}|_\infty^\frac12|w^{-1}|_{\infty}|g^{\frac{1}{2}} \nabla w_t|_2|\nabla w_t|_4|l_{t t}|_4 ,\\
	I_{20}= & -a_4 \int \nabla h^{\frac{1}{2}} \cdot \nabla l_{t t} l_{t t}+\frac{1}{2} \int(w^{-\nu} h^{-\frac{1}{2}})_t|l_{t t}|^2 \\
						\leq & C(|\varphi|_{\infty}^{\frac{1}{2}}|w^{\frac{\nu}{2}}|_{\infty}|\psi|_{\infty}|h^{\frac{1}{4}} \nabla l_{t t}|_2+|\varphi|_{\infty}|h_t|_{\infty}|w^{-\frac{\nu}{2}} h^{-\frac{1}{4}} l_{t t}|_2 \\
						& +|\varphi|_{\infty}^{\frac{1}{4}}|w^{-\frac{\nu}{2}-1}|_{\infty}| w_t|_4|l_{t t}|_4)|w^{-\frac{\nu}{2}} h^{-\frac{1}{4}} l_{t t}|_2,
					\end{align*}
			in which,	$H^t(v)$ is given by $$
			\begin{aligned}
				H^t(v)= & 4 \alpha \sum_{i=1}^2 \partial_i v_i \partial_{i tt} v_i+2 \beta \text{div} v \text{div} v_{t t}+2 \alpha \sum_{i \neq j}^2 \partial_i v_j \partial_{i tt} v_j \\
				& +2 \alpha \sum_{i>j}\left(\partial_{i tt} v_j \partial_j v_i+\partial_i v_j \partial_{jt t} v_i\right) .
			\end{aligned}
			$$
				
				To estimate $I_{15}$ and $I_{16}$, one needs  Lemma \ref{lem2as} and 
\begin{align*}
|l_{tt}|_4\le& (|h^{-\frac{1}{4}}l_{tt}|_2|h^\frac14 l_{tt}\cdot l_{tt}\cdot l_{tt}|_2)^\frac14
\\\le& C|h^{-\frac{1}{4}}l_{tt}|_2^\frac14|\nabla(h^\frac14l_{tt}\cdot l_{tt}\cdot l_{tt})|_1^\frac14\\
\le&C|h^{-\frac{1}{4}}l_{tt}|_2^\frac14(|h^{-\frac12}\nabla h|_\infty^\frac14|h^{-\frac{1}{4}}l_{tt}|_2^\frac14|l_{tt}|_4^\frac12+|h^\frac14\nabla l_{tt}|_2^\frac14|l_{tt}|_4^\frac12),\\
|l_{tt}|_4\le& C(|w^{-\frac{\nu}{2}}h^{-\frac{1}{4}}l_{tt}|_2|w^{\frac{\nu}{2}}|_{\infty}|h^{-\frac12}\nabla h|_\infty^\frac12\\&+|w^{\frac{\nu}{2}}|_{\infty}^\frac12|w^{-\frac{\nu}{2}}h^{-\frac{1}{4}}l_{tt}|_2^\frac12|h^\frac14\nabla l_{tt}|_2^\frac12),\\
|g^\frac54 \nabla v|_\infty\le& |g^\frac54 \nabla v|_4^\frac23|\nabla^2(g^\frac54 \nabla v)|_2^\frac13\\\le& |g^\frac52\nabla v\cdot \nabla v|_2^\frac13|\nabla^2(g^\frac54 \nabla v)|_2^\frac13\\
\le& |\nabla (g^\frac52\nabla v\cdot \nabla v)|_1^\frac13|\nabla^2(g^\frac54 \nabla v)|_2^\frac13\\\le& (|\nabla g|_\infty|g\nabla v|_2|g^\frac12\nabla v|_2+|g\nabla v|_2|g^\frac32\nabla^2v|_2)^\frac13|\nabla^2(g^\frac54 \nabla v)|_2^\frac13\\\le& M(c_0)c_4^2,\\
\int n g^{2} H^t(v) h^{-\frac12}l_{t t} \leq & C|w^{\frac{\nu}{2}}|_{\infty}|g \nabla v|_{\infty}|v_{t t}|_2\big(|n|_{\infty}|\varphi|_{\infty}^{\frac{1}{4}}(|\nabla g|_{\infty} +|\psi|_\infty|gh^{-1}|_\infty)\\
&+|\psi|_{\infty}|\varphi|_{\infty}^{\frac{1}{4}-b}|g h^{-1}|_{\infty}\big)|w^{-\frac{\nu}{2}} h^{-\frac{1}{4}} l_{t t}|_2  
\\&+C|n|_{\infty}| v_{t t}|_2|g h^{-1}|_{\infty}^{\frac{1}{2}}|g^\frac32 \nabla^2 v|_4| l_{t t}|_4\\&+C|n|_{\infty}| v_{t t}|_2|g h^{-1}|_{\infty}^{\frac{3}{4}}|g^\frac54 \nabla v|_\infty|h^{\frac{1}{4}} \nabla l_{t t}|_2, \\
a_6\int  w  \text{div}\psi_{t t} h^{-\frac{1}{2}} l_{t t}= & -a_6\int(\nabla w h^{-\frac{1}{2}}+w \nabla h^{-\frac{1}{2}}) \cdot \psi_{t t} l_{t t}-a_6\int w h^{-\frac{1}{2}} \psi_{t t} \cdot \nabla l_{t t} \\
\leq & C(|w^{\frac{\nu}{2}}|_{\infty}|\varphi|_{\infty}^{\frac{1}{4}}|\nabla w|_{\infty}+|w^{\frac{\nu}{2}+1}|_{\infty}|\varphi|_{\infty}^{\frac{5}{4}}|\psi|_{\infty})|\psi_{t t}|_2|w^{-\frac{\nu}{2}} h^{-\frac{1}{4}} l_{t t}|_2 \\
& +C|w|_{\infty}|\varphi|_{\infty}^{\frac{3}{4}}|\psi_{t t}|_2|h^{\frac{1}{4}} \nabla l_{t t}|_2.
\end{align*}
				Multiplying \ef{822} by $t$ and integrating over $(\tau, t)$, then \ef{801}, Lemmas \ref{gh}-\ref{l3} and estimates on $I_i(i=13, \ldots, 20)$ imply that
				\begin{equation*}
					\begin{aligned}
						& t|w^{-\frac{\nu}{2}} h^{-\frac{1}{4}} l_{t t}|_2^2+\frac{a_4}{4} \int_\tau^t s|h^{\frac{1}{4}} \nabla l_{s s}|_2^2 \text{d} s \\
						\leq & \tau|w^{-\frac{\nu}{2}} h^{-\frac{1}{4}} l_{t t}(\tau)|_2^2+M(c_0)+M(c_0) c_4^{14+3 \nu} \int_\tau^t s|w^{-\frac{\nu}{2}} h^{-\frac{1}{4}} l_{s s}|_2^2 \text{d} s.
					\end{aligned}
				\end{equation*}
				
				It follows from \ef{823} and Lemma \ref{bjr} that 
				$$
				\exists s_k\ \ \text{s.t.} \ \ s_k \longrightarrow 0\quad \text { and } \quad s_k|w^{-\frac{\nu}{2}} h^{-\frac{1}{4}} l_{t t}(s_k, x)|_2^2 \longrightarrow 0 \quad \text { as } \quad k \longrightarrow \infty .
				$$
				Letting $\tau=s_k$ and $k \rightarrow \infty$ in (3.89), then Gronwall's inequality implies that
				\begin{equation}
					t|w^{-\frac{\nu}{2}} h^{-\frac{1}{4}} l_{t t}|_2^2+\frac{a_4}{4} \int_\tau^t s|h^{\frac{1}{4}} \nabla l_{s s}|_2^2 \text{d} s+\int_0^t s|\nabla l_{s s}|_2^2 \text{d} s \leq M(c_0),
				\end{equation}
				for $0 \leq t \leq T_7$, this, together with \ef{820}, gives
				\begin{equation}\label{824}
					t^{\frac{1}{2}}|h^{-\frac{1}{4}} l_{t t}(t)|_2+t^{\frac{1}{2}}|\nabla^2 l_t(t)|_2+t^{\frac{1}{2}}|h^\frac12 \nabla^2 l_t(t)|_2 \leq M(c_0) c_1^{\frac{\nu}{2}} .
				\end{equation}
				
				Furthermore, \ef{818} and \ef{824} imply
				\begin{equation*}
					\begin{gathered}
						|l|_{\infty}=|l_0+\int_0^t l_s \text{d} s|_{\infty} \leq|l_0|_{\infty}+t|l_t|_{\infty} \leq c_0+C t| l_t|_2^{\frac{1}{4}}|\nabla l_t|_2^{\frac{1}{2}}|\nabla^2 l_t|_2^{\frac{1}{4}} \leq \frac{3}{2} c_0, \\
						l=l_0+\int_0^t l_s \text{d} s \geq l_0-t|l_t|_{\infty} \geq
						c_0^{-1}-C t| l_t|_2^{\frac{1}{4}}|\nabla l_t|_2^{\frac{1}{2}}|\nabla^2 l_t|_2^{\frac{1}{4}}\geq \frac{1}{2} c_0^{-1},
					\end{gathered}
				\end{equation*}
				for $0 \leq t \leq T_7$. 
				The proof of Lemma \ref{lt} is complete.
                
			\end{proof}
			\subsubsection{Estimates on $u$.}

			\begin{lemma}\label{u1}
				For $t \in[0, T_7]$, one has
				\begin{equation}\label{u123}
					\begin{aligned}
						|h^\frac12u|_2^2+\int_0^t(|\nabla u|_2^2+|h\nabla u|_2^2) \mathrm{d} s \leq M(c_0), \\
						|h\nabla u|_2^2+\int_0^t(|h^\frac32\nabla^2 u|_2^2+|\nabla^2u|_2^2+|h^\frac12u_s|_2^2) \mathrm{d} s \leq M(c_0),\\
						|h^\frac32\nabla^2u|_2^2+|h^\frac12u_t|_2^2+\int_0^t(|h^\frac32\nabla^3u|_2^2+|h\nabla u_s|_{2}^2) \mathrm{d} s \leq M(c_0) .
					\end{aligned}
				\end{equation}
			\end{lemma}
			\begin{proof}
				First, 
				multiplying $\ef{ln}_2$ by $hu$, integrating over $\mathbb{R}^2$, using \eqref{801} and Lemmas \ref{phiphii}-\ref{lt}, one has
				%
					\begin{align*}
						&\fr{1}{2}\frac{\text{d}}{\text{d}t}|h^{\frac12}u|_2^2+a_2\alpha|l^{\frac{\nu}{2}}h\nabla u|^2_2+a_2(\alpha+\beta)|l^{\frac{\nu}{2}}h\text{div} u|^2_2 \\
						=&-\int  \big(v \cdot \nabla v+a_1 \phi \nabla l+l \nabla \phi-a_2 \nabla l^\nu \cdot g Q(v)-a_3 l^\nu \psi \cdot Q(v)\big)\cdot h u\\
						& +\frac12\int h_t |u|^2 -a_2\int   \nabla  ( l^{\nu}  h^2 ) \cdot \hat{Q}(u)\cdot u\\
						\le & C \big( |g^{-1} h|_\infty|\varphi|_\infty^\frac12|g\nabla v|_2(|v|_\infty+|l^{\nu}|_\infty |\psi|_\infty) +|h^{\frac12}\nabla l|_2|\phi|_\infty+|h^{\frac12}\nabla \phi|_2|l|_\infty\\
						&+|h^{\frac12}\nabla l|_\infty |l^{\nu-1}|_\infty|g\nabla v|_2+|h^{-\frac12}h_t|_\infty|\varphi|_\infty^{\frac12} |h^{\frac12}u|_2\\
						&+|l^{\fr{\nu}{2}-1}|_\infty |h^{\fr12} \nabla l|_\infty | l^{\fr{\nu}{2}}h\nabla u|_2+ |\psi|_\infty |\varphi|_\infty^{\frac12} |l^{\frac{\nu}{2}}|_\infty |l^{\frac{\nu}{2}}h\nabla u|_2 \big)|h^{\frac12}u|_2\\
						\leq& M(c_0)c_4^{4}|h^{\fr12}u|^2_2+M(c_0)+\fr{1}{2}a_2\alpha|l^{\frac{\nu}{2}}h\nabla u|^2_2,
					\end{align*}
                    where $	\hat{Q}(u)=\alpha\nabla u+(\alpha+\beta)\text{div}u\mathbb{I}_2$, together with  \ef{2.14}, 
				\begin{equation*}
					\begin{aligned}
						|h^\frac12 \nabla l|_\infty\le& C|h^\frac12\nabla l|_2^\frac14|\nabla(h^\frac12\nabla l)|_2^\frac12|\nabla^2(h^\frac12\nabla l)|_2^\frac14
			\le M(c_0)c_1^{4\nu+3},\\
					\end{aligned}
				\end{equation*}
				and $\text{Gronwall's} $ inequalty   implies that 	for $0\leq t\leq T_7$,
				\begin{equation}\label{ul22}
					\begin{split}
						&| h^{\frac12}u|_2^2 +\int^t_0 \big(|\nabla u|^2_2+|h\nabla u|^2_2+|l^{\frac{\nu}{2}}h \nabla u|^2_2 \big)\text{d}s
						\le  M(c_0).
					\end{split}
				\end{equation}

				Second, 
				  $\ef{ln}_2$ also implies 
				\begin{equation}\label{825}
					\begin{aligned}
						 &l^{-\nu}\left(u_t+v \cdot \nabla v+a_1 \phi \nabla l+l \nabla \phi\right)+a_2 h L u 
						\\= & a_2 g l^{-\nu} \nabla l^\nu \cdot Q(v)+a_3 \psi \cdot Q(v),
					\end{aligned}
				\end{equation}
thus      
                multiplying \ef{825} by $hu_t$ and integrating over $\mathbb{R}^2$, one has 
				$$
				\begin{aligned}
					& \frac{1}{2} \frac{\text{d}}{\text{d}t}\big(a_2 \alpha|h \nabla u|_2^2+a_2(\alpha+\beta)|h \text{div} u|_2^2\big)+|l^{-\frac{\nu}{2}}h^\frac12 u_t|_2^2 \\
					= & -\int l^{-\nu}\big(v \cdot \nabla v+a_1 \phi \nabla l+l \nabla \phi-a_2 g \nabla l^\nu \cdot Q(v)-a_3 l^\nu \psi \cdot Q(v)\big) \cdot hu_t \\
					& + \int a_2 h h_t\big(\alpha|\nabla u|^2+(\alpha+\beta)|\text{div} u|^2\big) 
					-\int 2a_2 h\nabla h \cdot \hat{Q}(u) \cdot u_t \\
					\leq & C|l^{-\frac{\nu}{2}}|_{\infty}(|v|_{\infty}|g^{-1}h|_\infty|\varphi|^\frac12_\infty|g\nabla v|_2+|h^\frac12\nabla l|_2|\phi|_{\infty}+|l|_{\infty}|h^\frac12\nabla \phi|_2 \\
					& +|g \nabla v|_{4}|l^{\nu-1}|_{\infty}|h^\frac12\nabla l|_4+|g^{-1}h|_\infty|\varphi|^\frac12_\infty|\psi|_{\infty}|l^\nu|_{\infty}|g\nabla v|_2)|l^{-\frac{\nu}{2}} h^\frac12u_t|_2 \\
					& +C|h^{-\frac12}h_t|_{\infty}|\varphi|^\frac12_{\infty}|h \nabla u|_2^2+C|l^{\frac{\nu}{2}}|_{\infty}|\psi|_{\infty}|l^{-\frac{\nu}{2}} u_t|_2|h \nabla u|_2 \\
					\leq & M(c_0) c_4^6|h \nabla u|_2^2+M(c_0) +\frac{1}{2}|l^{-\frac{\nu}{2}} h^\frac12u_t|_2^2,
				\end{aligned}
				$$
				this, together with \eqref{2.14} and  Gronwall's inequality, give that for $0 \leq t \leq T_7$,
				\begin{equation}\label{828}
					\begin{aligned}
						& |h \nabla u|_2^2+|\nabla u|_2^2+\int_0^t(|l^{-\frac{\nu}{2}}h^\frac12 u_s|_2^2+|h^\frac12u_s|_2^2) \mathrm{d} s \\
						\leq & M(c_0)(1+t) \exp \big(M(c_0) c_4^6 t\big) \leq M(c_0) .
					\end{aligned}
				\end{equation}

				Since 
				\begin{equation}\label{720}
					a_2 L(h^\frac32 u)=l^{-\nu}h^\frac12 \mathcal{H}-a_2 G(\nabla h^\frac32, u),
				\end{equation}
                where 
                		\begin{equation}\label{kg}
				\begin{aligned}
					\mathcal{H}= & -u_t-v \cdot \nabla v-l \nabla \phi-a_1 \phi \nabla l+a_2 g \nabla l^\nu \cdot Q(v)+a_3 l^\nu \psi \cdot Q(v),\\
                    \widetilde{G}=& G(\nabla h^\frac32, u)=(\alpha+\beta)\big( \nabla h^\frac32{\rm{div}} u+\nabla h^\frac32\cdot\nabla u+u\cdot\nabla( \nabla h^\frac32)\big)\\&+\alpha\nabla h^{\frac32}\cdot\nabla u+\alpha {\rm{div}}(u\otimes \nabla h^{\frac32}).
				\end{aligned} 
			\end{equation}
				Thus \ef{Gdingyi}, \ef{801}, \ef{g32},  \ef{kg}, \ef{828} and Lemmas \ref{phiphii}-\ref{l3} imply
				\begin{equation}\label{HG1}
					\begin{aligned}
						|h^\frac12\mathcal{H}|_2 \leq & C(|h^\frac12u_t|_2+|g^{-1}h|_\infty^\frac12|v|_4|g^\frac12\nabla v|_4+|l|_{\infty}|h^\frac12\nabla \phi|_2+|\phi|_{\infty}|h^\frac12\nabla l|_2\\&+|h^\frac12\nabla l|_4| l^{\nu-1}|_{\infty}|g \nabla v|_4 
			+|g^{-1}h|_\infty^\frac12|l^\nu|_{\infty}|\psi|_{\infty}|g^\frac12\nabla v|_2)\\ \leq& M(c_0)(|h^\frac12u_t|_2+1), \\
						|\widetilde{G}|_2 \leq & C(|\psi|_{\infty}|h^\frac12\nabla u|_2+|\nabla \psi|_4|h^\frac12u|_4) \leq M(c_0),
					\end{aligned}
				\end{equation}
				where one also used 
				\begin{equation}\label{ld2}
					\begin{aligned}
						|h^\frac12\nabla l|_{2}\le|h_0^\frac12\nabla l_0|_2+t^\frac12\left(\int_0^t|(h^\frac12\nabla l)_t|_2^2\right)^\frac12\le M(c_0),\\ 
      |h^\frac12\nabla l|_{4}\le|h_0^\frac12\nabla l_0|_4+t^\frac12\left(\int_0^t|(h^\frac12\nabla l)_t|_4^2\right)^\frac12\le M(c_0),\\
      |g\nabla v|_2\le|h_0\nabla u_0|_2+t^\frac12\left(\int_0^t|(g\nabla v)_t|_2^2\right)^\frac12\le M(c_0),\\
						\|l\|_{D^1\cap D^2}\le\|l_0\|_{D^1\cap D^2}+t^\frac12\left(\int_0^t\|l_s\|^2_{D^1\cap D^2}\text{d}s\right)^\frac12\le M(c_0).\\
					\end{aligned}
				\end{equation}

				Thus \ef{ul22},\ef{828}-\ef{HG1} and Lemmas \ref{ell}, \ref{gh}-\ref{psi} imply that
				\begin{equation}\label{hu2}
					\begin{aligned}
						|h^\frac32 u|_{D^2} \leq & C(|l^{-\nu}h^\frac12 \mathcal{H}|_2+|G(\nabla h^\frac32, u)|_2) 
                        \\
                        \leq & M(c_0)(|h^\frac12u_t|_2+1), \\
						|h^\frac32 \nabla^2 u|_2 \leq & C(|h^\frac32 u|_{D^2}+|\nabla \psi|_4|h^\frac12u|_4+|\psi|_{\infty}|h^\frac12\nabla u|_2 ) \\ \leq & C|h^\frac32 u|_{D^2}+M(c_0),
					\end{aligned}
				\end{equation}
				which, along with \ef{ul22},\ef{828}, yield $\ef{u123}_2$.
				
				Next, 
				applying $\partial_t$ to $(3.1)_2$, one has
				\begin{equation}\label{829}
					\begin{aligned}
						& u_{t t}+a_2 l^\nu h L u_t+(v \cdot \nabla v)_t+(l \nabla \phi)_t+a_1(\phi \nabla l)_t \\
						= & -a_2(l^\nu h)_t L u+\big(a_2 g \nabla l^\nu \cdot Q(v)+a_3 l^\nu \psi \cdot Q(v)\big)_t .
					\end{aligned}
				\end{equation}
				Multiplying \ef{829} by $l^{-\nu} h u_t$ and  integrating over $\mathbb{R}^2$, one has 
					\begin{align*}
						& \frac{1}{2} \frac{\text{d}}{\text{d} t}|l^{-\frac{\nu}{2}}h^\frac12 u_t|_2^2+a_2 \alpha|h \nabla u_t|_2^2+a_2(\alpha+\beta)|h \text{div} u_t|_2^2 \\
						= & \int l^{-\nu}\Big(-(v \cdot \nabla v)_t-(l \nabla \phi)_t-a_1(\phi \nabla l)_t-a_2(l^\nu h)_t L u \\
						& +\big(a_2 g \nabla l^\nu \cdot Q(v)+a_3 l^\nu \psi \cdot Q(v)\big)_t\Big) \cdot hu_t \\
						& -\int 2a_2 h\nabla h \cdot \hat{Q}(u_t) \cdot u_t+\frac{1}{2} \int(l^{-\nu}h)_t|u_t|^2 \\
						\leq & C|l^{-\frac{\nu}{2}}|_{\infty}\big(|g^{-1}h|_\infty|\varphi|_\infty^\frac12(|v|_{\infty}|g\nabla v_t|_2+|v_t|_2|g\nabla v|_{\infty})
      +|l_t|_\infty|h^\frac12\nabla \phi|_2 \\
						&
     +|l|_{\infty}|h^\frac12\nabla \phi_t|_2 
      +|h^\frac12\nabla l_t|_2|\phi|_{\infty}+|\phi_t|_{\infty}|h^\frac12\nabla l|_2\big)|l^{-\frac{\nu}{2}}h^\frac12 u_t|_2\\
						&+C|l^{-1}|_{\infty}|h^\frac14l_t|_4|h \nabla^2 u|_2|h^\frac34u_t|_4  +C|l^{\frac{\nu}{2}}|_{\infty}|h^{-\frac12}h_t|_{\infty}|h\nabla^2 u|_2|l^{-\frac{\nu}{2}} h^\frac12u_t|_2\\
						&+C\big(|l^{\frac{\nu}{2}-2}|_{\infty}|g \nabla v|_{\infty}|l_t|_4|h^\frac12\nabla l|_4  +|l^{\frac{\nu}{2}-1}|_{\infty}(|g^{-\frac12}g_t|_{\infty}|g^\frac12\nabla v|_{\infty}|h^\frac12\nabla l|_2\\
						& +|g \nabla v|_{\infty}|h^\frac12\nabla l_t|_2+|g^{-1}h|_\infty^\frac12|\psi|_{\infty}|l_t|_4|g^\frac12\nabla v|_4) \\
						&+|l^{\frac{\nu}{2}}|_{\infty}|g^{-1}h|_\infty|\varphi|_\infty^\frac12(|\psi_t|_2|g\nabla v|_{\infty}+|\psi|_{\infty}|g\nabla v_t|_2)\big)|l^{-\frac{\nu}{2}}h^\frac12 u_t|_2 \\& +C|l^{-1}|_{\infty}|g \nabla v_t|_2|h^\frac12\nabla l|_\infty|l^{-\frac{\nu}{2}}h^\frac12 u_t|_4 \\&+C|l^{\frac{\nu}{2}}|_{\infty}|\varphi|_{\infty}^{\frac{1}{2}}|\psi|_{\infty}|h \nabla u_t|_2|l^{-\frac{\nu}{2}} h^\frac12u_t|_2\\
						& +C(|l^{-1}|_{\infty}|l_t|_\infty+|h^{-\frac12}h_t|_\infty|\varphi|_\infty^\frac12)|l^{-\frac{\nu}{2}} h^\frac12u_t|^2_2.
					\end{align*}
				Integrating it over $(\tau, t)\big(\tau \in(0, t)\big)$, then \ef{801} and Lemmas \ref{phiphii}-\ref{lt} imply that 
				\begin{equation}\label{831}
					\begin{aligned}
						& \frac{1}{2}|l^{-\frac{\nu}{2}}h^\frac12 u_t(t)|_2^2+\frac{a_2 \alpha}{2} \int_\tau^t|h \nabla u_s|_2^2 \mathrm{~d} s \\
						\leq & \frac{1}{2}|l^{-\frac{\nu}{2}} h^\frac12u_t(\tau)|_2^2+M(c_0) c_4^2 \int_0^t|l^{-\frac{\nu}{2}} h^\frac12u_s|^2 \mathrm{~d} s+M(c_0) c_4^{4+2 \nu} t+M(c_0) .
					\end{aligned}
				\end{equation}
				Letting $\tau \rightarrow 0$ in \ef{831}, by Lemmas \ref{inin}, \ref{lt} and Gronwall's inequality, for $0 \leq t \leq T_7$, one has
				\begin{equation}\label{ut1}
					\begin{aligned}
						& |h^\frac12u_t(t)|_2^2+\int_0^t(|h \nabla u_s|_2^2+|\nabla u_s|_2^2) \mathrm{d} s \\
						\leq & \big(M(c_0) c_4^{4+2 \nu} t+M(c_0)\big) \exp \big(M(c_0) c_4^2 t\big) \leq M(c_0),
					\end{aligned}
				\end{equation}
				this, together with \ef{hu2}, imply that 
				\begin{equation}\label{832}
					|h^\frac32 u(t)|_{D^2}+|h^\frac32 \nabla^2 u(t)|_2+|u(t)|_{D^2} \leq M(c_0) .
				\end{equation}

				Similarly,
				 \ef{Gdingyi}, \ef{801}, \ef{828},  \ef{kg}, \ef{ld2}, \ef{832} and Lemmas \ref{phiphii}-\ref{lt} give
				\begin{equation}\label{833}
					\begin{aligned}
						|h^\frac12\mathcal{H}|_{D^1} \leq & C|\nabla h^\frac12|_4|\mathcal{H}|_4+C(|h^\frac12\nabla u_t|_{2}+|v|_{\infty}|h^\frac12\nabla^2 v|_2\\
						&+|h^\frac14\nabla v|_4^2+|l|_{\infty}|h^\frac12\nabla^2 \phi|_2+|\nabla \phi|_4|h^\frac12\nabla l|_4  \\&+|\phi|_{\infty}|h^\frac12\nabla^2 l|_2+|\nabla g|_{\infty}|\nabla l^\nu|_{\infty}|h^\frac12\nabla v|_2\\
						&+|l^{\nu-1}|_\infty|h^\frac12\nabla^2 l|_4|g \nabla v|_4+|l^{\nu-2}|_\infty|\nabla l|_\infty|h^\frac12\nabla l|_4|g \nabla v|_4  \\
						& +|h^\frac12\nabla l^\nu|_{\infty}|g \nabla^2 v|_2+|\nabla l^\nu|_{\infty}|\psi|_{\infty}|h^\frac12\nabla v|_2\\&+|l^\nu|_{\infty}|\nabla \psi|_4|h^\frac12\nabla v|_4 +|l^\nu|_{\infty}|\psi|_{\infty}|h^\frac12\nabla^2 v|_2)  \\\leq&M(c_0)(|h^\frac12\nabla u_t|_{2}+c_3^{2 \nu+3}), \\
						|\widetilde{G}|_{D^1} \leq & C\big(|\psi|_{\infty}|\nabla(h^\frac12\nabla u)|_2+|\nabla \psi|_4(|h^\frac12\nabla u|_4+|\varphi|_\infty^\frac12|\psi|_\infty|u|_4) \\& +|\nabla^2 \psi|_2|h^\frac12u|_{\infty}\big) \leq M(c_0).
					\end{aligned}
				\end{equation}
				Thus, it follows from \ef{828}-\ef{HG1}, \ef{832}-\ef{833} and Lemmas \ref{ell}, \ref{gh}-\ref{psi}  that 
				\begin{equation}\label{834}
					\begin{aligned}
						|h^\frac32 u(t)|_{D^3} \leq & C|l^{-\nu} \mathcal{H}|_{D^1}+C|\widetilde{G}|_{D^1} 
						\leq M(c_0)(|h^\frac12\nabla u_t|_{2}+c_3^{2 \nu+3}), \\
						|h^\frac32 \nabla^3 u(t)|_2 \leq & C\big(|h^\frac32 u(t)|_{D^3}+|\psi|_{\infty}|\nabla(h^\frac12\nabla u)|_2\\&+|\nabla \psi|_4(|h^\frac12\nabla u|_4+|\varphi|_\infty^\frac12|\psi|_\infty|u|_4)  +|\nabla^2 \psi|_2|h^\frac12u|_{\infty}\big) \\
						\leq& M(c_0)(|h^\frac32 u(t)|_{D^3}+c_3^{2 \nu+3}),
					\end{aligned}
				\end{equation}
				this, together with \ef{ut1} and Lemma \ref{h3438n}, gives
				\begin{equation*}\label{835}
					\int_0^t(|h^\frac32 \nabla^3 u|_2^2+|h^\frac32 \nabla^2 u|_{D^1}^2+|u|_{D^3}^2) \mathrm{d} s \leq M(c_0) \quad \text { for } \quad 0 \leq t \leq T_7 .
				\end{equation*}
				
				The proof of Lemma \ref{u1} is complete.
				
			\end{proof}

			\begin{lemma}\label{u2}
				For $t \in[0, T_7]$, one has
				\begin{equation}\label{721}
					\begin{aligned}
						\big(|u|_{D^3}+|h^\frac32 \nabla^2 u|_{D^1}\big)(t) \leq M(c_0) c_3^{2 \nu+3}, \\
						|h \nabla u_t|_2+|u_t|_{D^1}+\int_0^t(|h^\frac12u_{s s}|_2^2+|u_s|_{D^2}^2) \mathrm{d} s \leq M(c_0), \\
						\int_0^t\big(|h \nabla^2 u_s|_2^2+|u|_{D^4}^2+|h \nabla^2 u|_{D^2}^2+|(h \nabla^2 u)_s|_2^2\big) \mathrm{d} s \leq M(c_0) .
					\end{aligned}
				\end{equation}
			\end{lemma}
			\begin{proof}
First, multiplying \ef{829} by $l^{-\nu} hu_{t t}$ and integrating over $\mathbb{R}^2$, one has
				\begin{equation}\label{836}
					\frac{a_2}{2} \frac{\text{d}}{\text{d} t}\big(\alpha|h \nabla u_t|_2^2+(\alpha+\beta)|\text{div} u_t|_2^2\big)+|l^{-\frac{\nu}{2}} h^\frac12u_{t t}|_2^2=\sum_{i=21}^{24} I_i,
				\end{equation}
		where 
					\begin{align*}
						I_{21}= & \int l^{-\nu}\big(-(v \cdot \nabla v)_t-(l \nabla \phi)_t-a_1(\phi \nabla l)_t \\
						& -a_2(l^\nu)_t h L u-a_2 l^\nu  h_t L u\big) \cdot hu_{t t} \\
						\leq & C|l^{-\frac{\nu}{2}}|_{\infty}\big(|g^{-1}h|_\infty|\varphi|_\infty^\frac12(|v|_{\infty}|g\nabla v_t|_2+|v_t|_2|g\nabla v|_{\infty})\\
						& +|l_t|_\infty|h^\frac12\nabla \phi|_2+|h^\frac12\nabla l_t|_2|\phi|_{\infty} +|\phi_t|_{\infty}|h^\frac12\nabla l|_2\\&+|l|_{\infty}|h^\frac12\nabla \phi_t|_2
						+|l^{\nu-1}|_{\infty}|l_t|_4|h^\frac32 \nabla^2 u|_4\\&+|l^\nu|_{\infty}|\varphi|_\infty^\frac12|h^{-\frac12}h_t|_{\infty}|h^\frac32\nabla^2 u|_2\big)|l^{-\frac{\nu}{2}} h^\frac12u_{t t}|_2, \\
						I_{22}= & \int l^{-\nu}\big(a_2 g \nabla l^\nu \cdot Q(v)+a_3 l^\nu \psi \cdot Q(v)\big)_t \cdot hu_{t t} \\
						\leq & C|l^{-\frac{\nu}{2}}|_{\infty}\big(|h^\frac12(\nabla l^\nu)_t|_2|g \nabla v|_{\infty}+|g^{-\frac12}g_t|_{\infty}|h^\frac12\nabla l^\nu|_4|g^\frac12\nabla v|_4. \\
						& +|h^\frac12 \nabla l^\nu|_{\infty}|g \nabla v_t|_2+|g^{-1}h|_\infty|\varphi|_\infty^\frac12(|l^\nu|_{\infty}|\psi|_{\infty}|g\nabla v_t|_2 \\
						& +|l^\nu|_{\infty}|\psi_t|_2|g\nabla v|_{\infty}+|(l^\nu)_t|_4|\psi|_{\infty}|g\nabla v|_4)\big)|l^{-\frac{\nu}{2}} h^\frac12u_{t t}|_2, \\
						I_{23}+I_{24}= & -\int 2a_2 h\nabla h \hat{Q}(u_t) \cdot u_{t t} +a_2 \int   hh_t\big(\alpha|\nabla u_t|^2+(\alpha+\beta)|\text{div} u_t|^2\big) \\
						\leq & C|\varphi|^\frac12_{\infty}(|l^{\frac{\nu}{2}}|_{\infty}|\psi|_{\infty}|h \nabla u_t|_2|l^{-\frac{\nu}{2}} h^\frac12u_{t t}|_2+|h^{-\frac12}h_t|_{\infty}|h \nabla u_t|_2^2) .
					\end{align*}

				Integrating \ef{836} over $(\tau, t)$,  by \ef{801} and Lemmas \ref{phiphii}-\ref{psi}, \ref{l1}, one has  
				\begin{equation}\label{838}
					\begin{aligned}
						&|h \nabla u_t(t)|_2^2+\int_\tau^t|l^{-\frac{\nu}{2}}h^\frac12 u_{s s}|_2^2 \mathrm{~d} s \\
						\leq& C|h \nabla u_t(\tau)|_2^2+M(c_0) c_4^2 \int_0^t|h \nabla u_s|_2^2 \mathrm{~d} s+M(c_0)(c_4^{7 \nu+6} t+1).
					\end{aligned}
				\end{equation}
		 %
				%
				Letting $\tau \rightarrow 0$ in \ef{838}, then Lemma \ref{inin} and Gronwall's inequality  give that 
				\begin{equation}\label{840}
					\begin{gathered}
						|h \nabla u_t(t)|_2^2+|\nabla u_t(t)|_2^2+\int_0^t|h^\frac12u_{s s}|_2^2 \mathrm{~d} s \\
						\leq M(c_0)(1+c_4^{7 \nu+6} t) \exp \big(M(c_0) c_4^2 t\big) \leq M(c_0),\quad \text{for}\ \ 0 \leq t \leq T_7.
					\end{gathered}
				\end{equation}
				This, together with \ef{834}, gives
				\begin{equation}\label{841}
					|h^\frac32u|_{D^3}+|h^\frac32 \nabla^3 u|_2+|h^\frac32 \nabla^2 u|_{D^1}+|\nabla^3 u|_2 \leq M(c_0) c_3^{2 \nu+3} .
				\end{equation}
				
				Now,  \ef{829} implies
				\begin{equation}\label{842}
					\begin{aligned}
						a_2 L(h u_t) & =a_2 h L u_t-a_2 G(\psi, u_t) 
						=l^{-\nu} \mathcal{G}-a_2 G(\psi, u_t),
					\end{aligned}
				\end{equation}
				with
				\begin{equation}\label{843}
					\begin{aligned}
						\mathcal{G}= & -u_{t t}-(v \cdot \nabla v)_t-(l \nabla \phi)_t-a_1(\phi \nabla l)_t-a_2(l^\nu)_t h L u \\
						& -a_2  h_t l^\nu L u+\big(a_2 g \nabla l^\nu \cdot Q(v)+a_3 l^\nu \psi \cdot Q(v)\big)_t,\\
                        \widehat{G}=&G(\psi, u_t).
					\end{aligned}
				\end{equation}
				%
				where \ef{Gdingyi}, \ef{801}, \ef{kg}, \ef{ld2}, \ef{840}-\ef{841}, \ef{843} and Lemmas \ref{phiphii}-\ref{u1} imply that
				\begin{equation}\label{844}
					\begin{aligned}
						|\mathcal{G}|_2 \leq & C(|u_{t t}|_2+\|v\|_2|\nabla v_t|_2+\|l\|_{L^{\infty} \cap D^1 \cap D^2}\|\phi_t\|_1\\&+\|\phi\|_2\|l_t\|_{L^4\cap D^1} +|(l^\nu)_t|_4|h L u|_4\\&+|l^\nu|_{\infty}|h^{-\frac12}h_t|_{\infty}|h^\frac12\nabla^2 u|_2+|g^{-\frac12}g_t|_{\infty}|\nabla l^\nu|_2|g^\frac12\nabla v|_{\infty} \\
						& +|g \nabla v|_{\infty}|\nabla(l^\nu)_t|_2+|l^{\nu-1}|_{\infty}|\sqrt{h} \nabla l|_{\infty}|g h^{-1}|_{\infty}^{\frac{1}{2}}|\sqrt{g} \nabla v_t|_2 \\
						& +|(l^\nu)_t|_4|\psi|_{\infty}|\nabla v|_4+|l^\nu|_{\infty}|\psi_t|_2|\nabla v|_{\infty}+|l^\nu|_{\infty}|\psi|_{\infty}|\nabla v_t|_2) \\
						\leq & M(c_0)(|u_{t t}|_2+c_4^{3 \nu+3}),\\
						|\mathcal{H}|_{D^2} \leq & C\big(|u_t|_{D^2}+\|v\|_2\|\nabla v\|_2+\|l\|_{L^{\infty} \cap D^1 \cap D^3}(|\phi|_\infty+\|\nabla \phi\|_2) \\
						& +|h^\frac12\nabla^3 l^\nu|_2|g^\frac12 \nabla v|_\infty|gh^{-1}|_\infty^\frac12
						+|h^\frac12\nabla^2 l^\nu|_4|g^\frac12 \nabla^2 v|_4|gh^{-1}|_\infty^\frac12\\
						&+|\nabla l^\nu|_\infty|g \nabla^3 v|_2+\|\nabla l^\nu\|_2\|\nabla g\|_{L^{\infty} \cap D^2}\|\nabla v\|_2 \\
						& +\|l^\nu\|_{L^{\infty} \cap D^1 \cap D^2}\|\psi\|_{L^q \cap D^{1} \cap D^2}\|\nabla v\|_2\big) \\
						\leq & M(c_0)(|u_t|_{D^2}+c_4^{6 \nu+5}), \\
						|\widehat{G}|_2 \leq & C(|\psi|_{\infty}|\nabla u_t|_2+|\nabla \psi|_4|u_t|_4) \leq M(c_0) .
					\end{aligned}
				\end{equation}
				
				Thus \ef{720}, \ef{HG1}, \ef{833}, \ef{840}-\ef{842}, \ef{844} and Lemmas \ref{ell}, \ref{phiphii}-\ref{u1} give that
				\begin{equation}\label{845}
					\begin{aligned}
						|h u_t|_{D^2} \leq & C|l^{-\nu} \mathcal{G}|_2+C|G(\psi, u_t)|_2 
						\le  M(c_0)(|u_{t t}|_2+c_4^{3 \nu+3}), \\
						|h \nabla^2 u_t|_2 \leq & C(|h u_t|_{D^2}+|\nabla u_t|_2|\psi|_{\infty} +|\nabla\psi|_{4}|u_t|_4) \\
						 \leq& M(c_0)(|u_{t t}|_2+c_4^{3 \nu+3}), \\
						|(h \nabla^2 u)_t|_2 \leq & C(|h \nabla^2 u_t|_2+|h^{-\frac12}h_t|_{\infty}|h^\frac12\nabla^2 u|_2) \leq M(c_0)(|u_{t t}|_2+c_4^{3 \nu+3}), \\
						|u|_{D^4} \leq & C|h^{-1} l^{-\nu} \mathcal{H}|_{D^2} \leq M(c_0)(|u_t|_{D^2}+c_4^{6 \nu+5}) \\
						\leq & M(c_0)(|u_{t t}|_2+c_4^{6 \nu+5}) .
					\end{aligned}
				\end{equation}

				It follows from $\ef{ln}_2$  that 
				\begin{equation*}\label{846}
					\begin{aligned}
						 a_2 L(h \nabla^{\varsigma} u)=a_2 h \nabla^{\varsigma} L u-a_2 G(\psi, \nabla^{\varsigma} u) 
						=  h \nabla^{\varsigma}(h^{-1} l^{-\nu} \mathcal{H})-a_2 G(\psi, \nabla^{\varsigma} u),
					\end{aligned}
				\end{equation*}
				where  $\varsigma \in \mathbb{Z}_{+}^2$ is the multi-index with $|\varsigma|=2$. This,  together with \ef{HG1}-\ef{ld2}, \ef{833}, \ef{840}-\ef{841}, \ef{844}-\ef{845} and Lemmas \ref{phiphii}-\ref{u1}, give
				\begin{equation}\label{847}
					\begin{aligned}
						|h \nabla^2 u(t)|_{D^2} \leq & C|h \nabla^{\varsigma}(h^{-1} l^{-\nu} \mathcal{H})|_2 +C(|\psi|_{\infty}|u|_{D^3}+|\nabla \psi|_4|\nabla^2 u|_4) \\
						\leq & M\left(c_0\right)(|u_{t t}|_2+c_4^{6 \nu+5}) .
					\end{aligned}
				\end{equation}
				Finally,  \ef{801}, \ef{840}, \ef{845}-\ef{847} and Lemma \ref{h3438n} imply that
				\begin{equation*}\label{848}
					\int_0^{T_7}(|h \nabla^2 u_t|_2^2+|u_t|_{D^2}^2+|u|_{D^4}^2+|h \nabla^2 u|_{D^2}^2+|(h \nabla^2 u)_t|_2^2) \mathrm{d} t \leq M(c_0) .
				\end{equation*}
				
				The proof of Lemma \ref{u2} is complete.
				
			\end{proof}
			
			\begin{lemma}\label{ut}
				For $t \in\left[0, T_7\right]$, one has
				\begin{equation}\label{849}
					\begin{aligned}
						t|u_t(t)|_{D^2}^2+t|h \nabla^2 u_t(t)|_2^2+t|u_{t t}(t)|_2^2+t|u(t)|_{D^4}^2(t) & \leq M(c_0) c_4^{6 \nu+4}, \\
						\int_0^t s(|u_{s s}|_{D^1}^2+|h\nabla^3u_s|_{2}^2+|h^\frac12 u_{s s}|_{D^1}^2) {\rm{d}} s & \leq M(c_0) c_4^{6 \nu+4} .
					\end{aligned}
				\end{equation}
			\end{lemma}			
			\begin{proof}
				Applying $\partial_t$ to \ef{829},
				%
				 multiplying the resulting equation by $l^{-\nu} u_{t t}$ and integrating over $\mathbb{R}^2$, one has 
				\begin{equation}\label{851}
					\frac{1}{2} \frac{\text{d}}{\text{d} t}|l^{-\frac{\nu}{2}} u_{t t}|_2^2+a_2 \big(\alpha|h^\frac{1}{2} \nabla u_{t t}|_2^2+(\alpha+\beta)|h^{\frac{1}{2}} \text{div} u_{t t}|_2^2\big)=\sum_{i=25}^{28} I_i,
				\end{equation}
				where 
		\begin{align*}
						I_{25}=&\int l^{-\nu}\big(-(v \cdot \nabla v)_{t t}-a_1(\phi \nabla l)_{t t}-(l \nabla \phi)_{t t}\big) \cdot u_{t t} \\
						\leq& C|l^{-\frac{\nu}{2}}|_{\infty}(|\nabla v_t|_4|v_t|_4+|\nabla v|_{\infty}|v_{t t}|_2+|v|_{\infty}|\nabla v_{t t}|_2+|\phi|_{\infty}|\nabla l_{t t}|_2 \\
						& +|\phi_{t t}|_2|\nabla l|_{\infty}+|\phi_t|_{\infty}|\nabla l_t|_2+|l_t|_4|\nabla \phi_t|_4 \\
						& +|l_{t t}|_4|\nabla \phi|_4+|l|_{\infty}|\nabla \phi_{t t}|_2)|l^{-\frac{\nu}{2}} u_{t t}|_2, \\
						I_{26}=&a_3 \int l^{-\nu}\big(l^\nu \psi \cdot Q(v)\big)_{t t} \cdot u_{t t} \\
						\leq& C|l^{-\frac{\nu}{2}}|_{\infty}(|l^\nu|_{\infty}|\psi_{t t}|_2|\nabla v|_{\infty}+|l^{\nu-1}|_{\infty}|l_{t t}|_4|\psi|_{\infty}|\nabla v|_4 \\
						& +|l^{\nu-2}|_{\infty}|l_t|_4^2|\psi|_{\infty}|\nabla v|_\infty+|l^\nu|_{\infty}|\psi|_{\infty}|\nabla v_{t t}|_2+|\psi_t|_4|(l^\nu)_t|_4|\nabla v|_{\infty} \\
						& +|\psi|_{\infty}|(l^\nu)_t|_4|\nabla v_t|_4+|l^\nu|_{\infty}|\psi_t|_4|\nabla v_t|_4)|l^{-\frac{\nu}{2}} u_{t t}|_2, \\
						I_{27}=&a_2 \int l^{-\nu}\big(g \nabla l^\nu \cdot Q(v)\big)_{t t} \cdot u_{t t} \\
						\leq &C|l^{-\frac{\nu}{2}}|_{\infty}\big(|l^{\nu-3}|_{\infty}|g \nabla v|_{\infty}|l_t|_4^2|\nabla l|_\infty+|l^{\nu-2}|_{\infty}|g \nabla v|_{\infty}|l_{t t}|_4|\nabla l|_4 \\
						& +|l^{\nu-2}|_{\infty}|g \nabla v|_{\infty}|l_t|_4|\nabla l_t|_4+|l^{\nu-1}|_{\infty}|\nabla l_t|_2|g^{-\frac12}g_t|_{\infty}|g^{\frac12}\nabla v|_{\infty} \\
						& +|l^{\nu-2}|_{\infty}|l_t|_4(|\nabla l|_4|g^{-\frac12}g_t|_{\infty}|g^{\frac12}\nabla v|_{\infty}\\
						&+|gh^{-1}|_\infty^\frac14|h^\frac12\nabla l|_\infty|g^\frac12 \nabla v_t|_2)\big)|l^{-\frac{\nu}{2}} u_{t t}|_2 \\
						& +C|l^{\frac{\nu}{2}-1}|_{\infty}|gh^{-1}|_\infty^\frac14|h^\frac14\nabla l_t|_4|g^\frac34 \nabla v_t|_4|l^{-\frac{\nu}{2}}u_{t t}|_2\\
						&+C|l^{\frac{\nu}{2}-1}|_{\infty}|\nabla l|_4|g^{-\frac12}g_{t t}|_4|g^\frac12\nabla v|_{\infty}|l^{-\frac{\nu}{2}} u_{t t}|_2 \\
						& +C|l^{\frac{\nu}{2}-1}|_{\infty}|g h^{-1}|_{\infty}^{\frac{1}{2}}|\sqrt{h} \nabla l|_{\infty}|\sqrt{g} \nabla v_{t t}|_2|l^{-\frac{\nu}{2}} u_{t t}|_2 \\
						& +C|l^{\frac{\nu}{2}-1}|_{\infty}(|\nabla l|_{\infty}|g^{-\frac12}g_t|_{\infty}|g^\frac12\nabla v_t|_2+|g \nabla v|_{\infty}|\nabla l_{t t}|_2)|l^{-\frac{\nu}{2}} u_{t t}|_2, \\
						I_{28}=&-a_2 \int l^{-\nu}\big((h l^\nu)_{t t} L u+2(l^\nu)_t h L u_t+2 h_t l^\nu L u_t-l^\nu \nabla h \cdot \hat{Q}(u_{t t})\big) \cdot u_{t t} \\
						& +\frac{1}{2} \int(l^{-\nu})_t|u_{t t}|^2 \\
						\leq& C|l^{-\frac{\nu}{2}}|_{\infty}(|(l^\nu)_t|_4|h^{-\frac12}h_t|_{\infty}|h^\frac12\nabla^2 u|_4+|l^{\nu-2}|_{\infty}|l_t|_4^2|h \nabla^2 u|_\infty \\
						& +|l^\nu|_{\infty}|h^{-\frac12}h_{t t}|_4|h^\frac12\nabla^2 u|_4+|l^\nu|_{\infty}|h^{-\frac12}h_t|_{\infty}^2|\nabla^2 u|_2 \\
						& +|l^\nu|_{\infty}|h^{-\frac12}h_t|_{\infty}|\varphi|^\frac12_{\infty}|h \nabla^2 u_t|_2+|l^\nu|_{\infty}|\psi|_{\infty}|\sqrt{h} \nabla u_{t t}|_2|\varphi|_{\infty}^{\frac{1}{2}})|l^{-\frac{\nu}{2}} u_{t t}|_2 \\
						&+C|l^{-1}|_{\infty}|l_t|_4|h \nabla^2 u_t|_2|u_{t t}|_4 \\
						& +C|l^{-1}|_{\infty}|l_{t t}|_4|h \nabla^2 u|_2|u_{t t}|_4+C|l^{-\frac{\nu}{2}-1}|_{\infty}|l_t|_4|l^{-\frac{\nu}{2}} u_{t t}|_2|u_{t t}|_4 .
					\end{align*}
				Multiplying \ef{851} by $t$ and integrating over $(\tau, t)$, then \ef{801}, Lemmas \ref{phiphii}-\ref{u2} and the estimates of $I_i(i=25, \ldots, 28)$ give that
				\begin{equation}\label{853}
					\begin{aligned}
						& t|l^{-\frac{\nu}{2}} u_{t t}(t)|_2^2+\frac{a_2 \alpha}{4} \int_\tau^t s|h^\frac12 \nabla u_{s s}|_2^2 \mathrm{~d} s \\
						\leq & \tau|l^{-\frac{\nu}{2}} u_{t t}(\tau)|_2^2+M(c_0) c_4^{6 \nu+4}(1+t)+M(c_0) c_5^{2 \nu+8} \int_\tau^t s|l^{-\frac{\nu}{2}} u_{s s}|_2^2 \mathrm{~d} s .
					\end{aligned}
				\end{equation}
				
				It follows from \ef{840} and Lemma \ref{bjr} that 
				$$
				\exists s_k \quad \text{s.t.} \quad s_k \longrightarrow 0\quad \text { and } \quad s_k\left|u_{t t}\left(s_k, x\right)\right|_2^2 \longrightarrow 0\quad \text { as } \quad k \longrightarrow \infty .
				$$  
				Letting $\tau=s_k$ and  $k \rightarrow \infty$ in \ef{853}, then Gronwall's inequality  implies that
				\begin{equation}\label{854}
					t|u_{t t}(t)|_2^2+\int_0^t s|h^\frac12 \nabla u_{s s}|_2^2 \mathrm{~d} s+\int_0^t s|\nabla u_{s s}|_2^2 \mathrm{~d} s \leq M(c_0) c_4^{6 \nu+4}, 
				\end{equation}
				for $0 \leq t \leq T_7$.
				Moreover,  \ef{845} and \ef{854} imply
				\begin{equation}\label{722}
					t^{\frac{1}{2}}|\nabla^2 u_t(t)|_2+t^{\frac{1}{2}}|h \nabla^2 u_t(t)|_2+t^{\frac{1}{2}}|\nabla^4 u(t)|_2 \leq M(c_0) c_4^{3 \nu+2} .
				\end{equation}
				
				Now,
				  \ef{Gdingyi}, \ef{ld2}, \ef{843} and Lemmas \ref{phiphii}-\ref{u2} give that
				\begin{equation}\label{855}
					\begin{aligned}
						|\mathcal{G}|_{D^1} \leq & C\big(|u_{t t}|_{D^1}+\|\nabla v\|_2|\nabla v_t|_2+|v|_{\infty}|\nabla^2 v_t|_2 \\
						& +\|l\|_{L^{\infty} \cap D^1 \cap D^3}\|\phi_t\|_2+\|l_t\|_{D^1 \cap D^2}(|\phi|_\infty+\|\nabla \phi\|_2) \\
						& +\|l^{\nu-1}\|_{1, \infty}\|l_t\|_{L^{\infty} \cap D^2}(\|h L u\|_1+|\psi|_{\infty}|\nabla^2 u|_2) \\
						& +(1+|\psi|_{\infty})(1+|\varphi|_{\infty})(|h^{-\frac12}h_t|_{L^{\infty}}+|h_t|_{  D^2})\|l^\nu\|_{1, \infty}\|h^\frac12\nabla^2 u\|_1 \\
						& +(|g^{-\frac12}g_t|_{L^{\infty}}+|g_t|_{ D^1})\|\nabla l^\nu\|_2\|g^\frac12\nabla v\|_2\\&+\|\nabla l^\nu\|_2(|\nabla g|_{\infty}|\nabla v_t|_2+|g \nabla^2 v_t|_2) \\
						& +(|g \nabla v|_{\infty}+|\nabla g|_{\infty}\|\nabla v\|_2+\|g \nabla^2 v\|_1)\|l_t\|_{D^1 \cap D^2}\|l^{\nu-1}\|_{L^{\infty} \cap D^1 \cap D^3} \\
						& +\|l^{\nu-1}\|_{1, \infty}\|l_t\|_{D^1}\|\psi\|_{L^{\infty} \cap D^{1}}\|\nabla v\|_2 \\
						& +\|l^\nu\|_{1, \infty}\|\psi_t\|_1\|\nabla v\|_2+\|l^\nu\|_{1, \infty}\|\psi\|_{L^{\infty} \cap D^{1}}\|\nabla v_t\|_1\big) \\
						\leq & M(c_0)(|\nabla u_{t t}|_2+c_4^{4 \nu+3}|g \nabla^2 v_t|_2+c_4^{5 \nu+5}|l_t|_{D^2}+c_4^{5 \nu+7}), \\
						|\widehat{G}|_{D^1} \leq & C(|\psi|_{\infty}|\nabla^2 u_t|_2+|\nabla \psi|_4|\nabla u_t|_4 
						+|\nabla^2 \psi |_2|u_t|_{\infty}) \\\leq& M(c_0)(|u_t|_{D^2}+c_4^{2 \nu+3}) .
					\end{aligned}
				\end{equation}
				Thus \ef{842}, \ef{855} and Lemmas \ref{ell}, \ref{phiphii}-\ref{u2} imply that for $0 \leq t \leq T_7$,
				\begin{equation*}\label{856}
					\begin{aligned}
						|h u_t|_{D^3} & \leq C|l^{-\nu} \mathcal{G}|_{D^1}+C|\widehat{G}|_{D^1} \\
						& \leq M(c_0)\big(\| u_{t t}\|_1+|u_t|_{D^2}+c_4^{6 \nu+7}(|g \nabla^2 v_t|_2+|l_t|_{D^2}+1)\big), \\
						|h \nabla^3 u_t(t)|_2 &\leq  C(|h u_t|_{D^3}+|u_t|_{\infty}|\nabla^2 \psi|_2+|\nabla u_t|_4|\nabla \psi|_4 +|\nabla^2 u_t|_2|\psi|_{\infty}) \\
						& \leq C|h u_t|_{D^3}+M(c_0)(|u_t|_{D^2}+c_4^{2 \nu+3}),
					\end{aligned}
				\end{equation*}
				this, together with \ef{721}, \ef{854}-\ef{722} and Lemma \ref{h3438n}, yield $\ef{849}_2$.
				
				The proof of Lemma \ref{ut} is complete.
                
			\end{proof}
			
			Thus  for $0 \leq t \leq T_7=\min \{T^*,\big(1+M(c_0) c_5\big)^{-40-10 \nu}\}$, Lemmas \ref{gh}-\ref{ut} imply
			\begin{align*}
				\|(\phi-\epsilon)(t)\|_{L^p\cap D^1 \cap D^3}^2+|h^\frac12\nabla\phi(t)|_2^2+|h\nabla^2\phi(t)|_2^2\leq M(c_0),\\
    |h^\frac12\phi_t(t)|^2_2+|h\nabla\phi_t(t)|^2_2+\|\phi_t(t)\|_2^2+|\phi_{t t}(t)|_2^2+\int_0^t\|\phi_{s s}\|_1^2 \text{d} s \leq C c_4^6, \\
				\|\psi(t)\|_{L^q \cap D^{1} \cap D^2}^2 \leq M(c_0), \quad |\psi_t(t)|_2^2 \leq C c_3^4, \quad|h_t(t)|_{\infty}^2 \leq M(c_0) c_3^\frac{3}{2} c_4^\frac12, \\
|\psi_t(t)|_{D^1}^2+\int_0^t(|\psi_{s s}|_2^2+|h_{s s}|_4^2) \text{d} s \leq M(c_0) c_4^4, \\
				h(t, x)>\frac{1}{2 c_0}, \quad \frac{2}{3} \epsilon^{-2 \iota}<\varphi<2c_0,\quad \widetilde{C}^{-1} \leq g h^{-1}(t, x) \leq \widetilde{C},\\ |\zeta(t)|_4^2+\|n(t)\|^2_{L^{\infty} \cap D^{1, q} \cap D^{1,4} \cap D^2 \cap D^3} \leq M(c_0), \quad|n_t(t)|^2_2 \leq M(c_0) c_2^2, \\
				|n_t(t)|^2_{\infty}+|\nabla n_t(t)|^2_2+|\nabla n_t(t)|_4^2 \leq M(c_0) c_4^2, \quad|n_{t t}(t)|_2^2 \leq M(c_0) c_4^6, \\
				| (l-\bar{l})(t)|_2^2+|h^{\frac{1}{2}} \nabla l(t)|_2^2+\int_0^t(l_s|_2^2+|h \nabla^2 l|_2^2+|\nabla^2 l|_2^2) \text{d} s \leq M(c_0) c_1^{3 \nu}, \\
		|l_t(t)|_2^2+|h \nabla^2 l(t)|_2^2+\int_0^t(|h^{\frac{1}{2}} \nabla l_s|_2^2+|h \nabla^3 l|_2^2) \text{d} s \leq M(c_0) c_1^{4 \nu+2}, \\
				|h^{\frac{1}{2}} \nabla l_t(t)|_2^2+|h\nabla^3 l(t)|_2^2+\int_0^t(| l_{s s}|_2^2+|h^\frac12\nabla^2 l_s|_2^2) \text{d} s \leq M(c_0) c_1^{8 \nu+6}, \\
t|l_t(t)|_{D^2}^2+t|h^\frac12 \nabla^2 l_t(t)|_2^2+t|h^{-\frac{1}{4}} l_{t t}(t)|_2^2 \leq M(c_0) c_1^\nu, \\\int_0^t s(|l_{s s}|_{D^1}^2+|h^{\frac{1}{4}} l_{s s}|_{D^1}^2) \text{d} s \leq M(c_0), \quad \frac{1}{2} c_0^{-1} \leq l(x, t) \leq \frac{3}{2} c_0,\\
				|h^\frac12u|_{2}^2+|h \nabla u(t)|_2^2+\|u(t)\|_1^2+\int_0^t(\|\nabla u \|_1^2+|u_s|_2^2) \text{d} s \leq M(c_0),\\
    (|u|_{D^2}^2+|h^\frac32 \nabla^2 u|_2^2+|h^\frac12u_t|_2^2)(t)+\int_0^t(|u|_{D^3}^2+|h^\frac32 \nabla^2 u|_{D^1}^2+|u_s|_{D^1}^2) \text{d} s \leq M(c_0), \\
		(|u_t|_{D^1}^2+|h \nabla u_t|_2^2+|u|_{D^3}^2+|h^\frac32 \nabla^2 u|_{D^1}^2)(t)+\int_0^t|u_s|_{D^2}^2 \text{d} s \leq M(c_0) c_3^{4 \nu+6}, \\
				\int_0^t(|h^\frac12u_{s s}|_2^2+|u|_{D^4}^2+|h \nabla^2 u|_{D^2}^2+|(h \nabla^2 u)_s|_2^2) \text{d} s \leq M(c_0), \\
				t|u_t(t)|_{D^2}^2+t|h \nabla^2 u_t(t)|_2^2+t|u_{t t}(t)|_2^2+t|u(t)|_{D^4}^2 \leq M(c_0) c_4^{6 \nu+4},\\
				\int_0^t s(|u_{s s}|_{D^1}^2+|h \nabla^3 u_s|_2^2+|h^\frac12 u_{s s}|_{D^1}^2) \text{d} s \leq M(c_0) c_4^{6 \nu+4} .
			\end{align*}
Let
		\begin{align*}
				T^* & =\min \left\{T,\big(1+M(c_0)^{36 \nu^3+104 \nu^2+102 \nu+36}\big)^{-40-10 \nu}\right\}, \quad c_1^2=M(c_0)^2,  \\
				c_2^2&=c_3^2=M(c_0)^{8 \nu+7}, \quad
				c_4^2  =M(c_0)^{24 \nu^2+45 \nu+23}, \quad c_5^2=M(c_0)^{72 \nu^3+207 \nu^2+204 \nu+70},
			\end{align*}
	for $0 \leq t \leq T^*$, one has the following uniform energy estimates with respect to   $\epsilon$: 
				\begin{equation}\label{857}
			\begin{aligned}
					\|(\phi-\epsilon)(t)\|_{L^p\cap D^1 \cap D^3}^2+|h^\frac12\nabla\phi(t)|_2^2+|h\nabla^2\phi(t)|_2^2\leq c_5^2,&\\
    |h^\frac12\phi_t(t)|^2_2+|h\nabla\phi_t(t)|^2_2+\|\phi_t(t)\|_2^2+|\phi_{t t}(t)|_2^2+\int_0^t\|\phi_{s s}\|_1^2 \text{d} s \leq c_5^2, &\\
    \|\psi(t)\|_{L^q \cap D^{1} \cap D^2}^2 \leq c_1, \quad |\psi_t(t)|_2^2 \leq c_4^2, \quad|h_t(t)|_{\infty}^2 \leq c_4^2, &\\
|\psi_t(t)|_{D^1}^2+\int_0^t(|\psi_{s s}|_2^2+|h_{s s}|_4^2) \text{d} s \leq  c_5^2, &\\
				h(t, x)>\frac{1}{2 c_0}, \frac{2}{3} \epsilon^{-2 \iota}<\varphi<2c_0,\widetilde{C}^{-1} \leq g h^{-1}(t, x) \leq \widetilde{C},&\\ |\zeta(t)|_4^2+\|n(t)\|^2_{L^{\infty} \cap D^{1, q} \cap D^{1,4} \cap D^2 \cap D^3} \leq c_1, \quad|n_t(t)|^2_2 \leq  c_4^2, &\\
				|n_t(t)|^2_{\infty}+|\nabla n_t(t)|^2_2+|\nabla n_t(t)|_4^2 +|n_{t t}(t)|_2^2 \leq c_5^2, &\\
				| (l-\bar{l})(t)|_2^2+|h^{\frac{1}{2}} \nabla l(t)|_2^2+\int_0^t(l_s|_2^2+|h \nabla^2 l|_2^2+|\nabla^2 l|_2^2) \text{d} s \leq c_2^{2},&
                \end{aligned}	    
				\end{equation}
                \begin{equation}\label{310857}
			\begin{aligned}
		|l_t(t)|_2^2+|h \nabla^2 l(t)|_2^2+\int_0^t(|h^{\frac{1}{2}} \nabla l_s|_2^2+|h \nabla^3 l|_2^2) \text{d} s \leq  c_2^{2}, &\\
				|h^{\frac{1}{2}} \nabla l_t(t)|_2^2+|h\nabla^3 l(t)|_2^2+\int_0^t(| l_{s s}|_2^2+|h^\frac12\nabla^2 l_s|_2^2) \text{d} s \leq  c_2^{2}, &\\
    t|l_t(t)|_{D^2}^2+t|h^\frac12 \nabla^2 l_t(t)|_2^2+t|h^{-\frac{1}{4}} l_{t t}(t)|_2^2 \leq  c_2^2, &\\
    \int_0^t s(|l_{s s}|_{D^1}^2+|h^{\frac{1}{4}} l_{s s}|_{D^1}^2) \text{d} s \leq c_1^2, \quad  c_1^{-1} \leq l(x, t) \leq  c_1,&\\
				|h^\frac12u|_{2}^2+|h \nabla u(t)|_2^2+\|u(t)\|_1^2+\int_0^t(\|\nabla u \|_1^2+|u_s|_2^2) \text{d} s \leq c_1^2,&\\
    (|u|_{D^2}^2+|h^\frac32 \nabla^2 u|_2^2+|h^\frac12u_t|_2^2)(t)\qquad\quad\quad&\\+\int_0^t(|u|_{D^3}^2+|h^\frac32 \nabla^2 u|_{D^1}^2+|u_s|_{D^1}^2) \text{d} s \leq c_3^2, &\\
		(|u_t|_{D^1}^2+|h \nabla u_t|_2^2+|u|_{D^3}^2+|h^\frac32 \nabla^2 u|_{D^1}^2)(t)+\int_0^t|u_s|_{D^2}^2 \text{d} s \leq  c_4^{2}, &\\
				\int_0^t(|h^\frac12u_{s s}|_2^2+|u|_{D^4}^2+|h \nabla^2 u|_{D^2}^2+|(h \nabla^2 u)_s|_2^2) \text{d} s \leq c_4^2, &\\
				t|u_t(t)|_{D^2}^2+t|h \nabla^2 u_t(t)|_2^2+t|u_{t t}(t)|_2^2+t|u(t)|_{D^4}^2 \leq  c_5^{2},&\\
				\int_0^t s(|u_{s s}|_{D^1}^2+|h \nabla^3 u_s|_2^2+|h^\frac12 u_{s s}|_{D^1}^2) \text{d} s \leq  c_5^{2}.&
				\end{aligned}	    
				\end{equation}
            
            In fact, we can  reset
            $$
			\begin{aligned}
				T^* & =\min \left\{T,\Big(1+10\big(10^6M(c_0)\big)^{36 \nu^3+104 \nu^2+102 \nu+36}\Big)^{-40-10 \nu}\right\}, \\
				c_1^2 & =10^2M(c_0)^2, \quad c_2^2=c_3^2=10^2\big(10^2M(c_0)\big)^{8 \nu+7}, \\
				c_4^2 & =10^2\big(10^4M(c_0)\big)^{24 \nu^2+45 \nu+23}, \quad c_5^2=10^2\big(10^6M(c_0)\big)^{72 \nu^3+207 \nu^2+204 \nu+70},
			\end{aligned}
			$$
        which ensures that \ef{857}-\ef{310857} holds with the new settings  and leads to the following desired estimates:    
           \begin{align*}\label{251801}
					\sup_{0\leq t\leq T^*}(\|\nabla h\|^2_{L^\infty\cap L^q\cap D^{1}\cap D^2}+|\nabla h^\frac{1}{2}|_4^2)(t)\leq c_1^2,\qquad&\\
					\inf_{[0,T^*]\times \mathbb{R}^2} l(t,x)\geq c^{-1}_1,\quad   \inf_{[0, T^{*}]\times \mathbb{R}^2} h(t,x)\ge c_1^{-1},\qquad&\\
					\sup_{0\leq t\leq T^*}  (|l|^2_{\infty}+|u|_\infty^2)(t)+\int^{T^*}_0(|u|^2_{D^2}+|u_t|^2_2)\text{d}t\leq c_1^2,\qquad&\\
					\sup_{0\leq t\leq T^*}  (\|u\|^2_{1}+|h\nabla u|_2^2+|h^{\frac12}\nabla l|_2^2)(t)+\int_{0}^{T^*}(|l_t|_2^2+|h\nabla^2l|_2^2)\text{d}t\le c_2^2,\qquad&\\
					\sup_{0\leq t\leq T^*}(|l_t|_2^2+|h\nabla^2l|_2^2)(t)+\int_{0}^{T^*}(|h^\frac12\nabla l_t|_2^2+|h\nabla^3 l|_2^2)\text{d}t\le c_2^2,\qquad&\\
					\sup_{0\leq t\leq T^*}(|h^{\frac12}\nabla l_t|_2^2+|h\nabla^3l|_2^2)(t)+\int_{0}^{T^*}(| l_{tt}|_2^2+|h^\frac12\nabla^2 l_t|_2^2)\text{d}t\le c_2^2,\qquad&\\
					\text{ess}\sup_{0\leq t\leq T^*}t(|h^{-\frac14} l_{tt}|_2^2+|h^\frac12\nabla^2 l_t|_2^2)(t)+\int_{0}^{T^*}t|h^\frac14l_{tt}|_{D^1}^2\text{d}t\le c_2^2,\qquad&\\
					\sup_{0\leq t\leq T^*}(|u|^2_{D^2}+|u_t|^2_2+|h^\frac32\nabla^2u|^2_2)(t)+\int^{T^*}_0(|u|_{D^3}^2+|u_t|^2_{D^1})\text{d}t\leq c_3^2,\qquad&\\
					\sup_{0\leq t\leq T^*}(|u|^2_{D^3}+|u_t|_{D^1}^2+|h\nabla u_t|^2_2+|h_t|^2_{D^1})(t)+\int^{T^*}_0(|u|^2_{D^4}+|u_t|^2_{D^2}+|u_{tt}|^2_2)\text{d}t\leq c_4^2,\qquad&\\
					\sup_{0\leq t\leq T^*}(|h^\frac32\nabla^3u|^2_2+|h_t|^2_\infty)(t)+\int^{T^*}_0(|(h\nabla^2u)_t|^2_2+|h\nabla^4u|^2_2)\text{d}t\leq c_4^2,\qquad&\\
					\text{ess}\sup_{0\leq t\leq T^*}t(|u|^2_{D^4}+|h\nabla^2u_t|^2_2)(t)+\int^{T^*}_0 (|h_{tt}|_4^2+ |h_{tt}|^2_{D^1})\text{d}t\leq c^2_5,\qquad&\\
					\text{ess}\sup_{0\leq t\leq T^*}t|u_{tt}(t)|^2_2+\int^{T^*}_0t(|u_{tt}|^2_{D^1}+|h^\frac12u_{tt}|_{D^1}^2+|u_t|^2_{D^3})\text{d}t\leq c_5^2.\qquad&
				\end{align*}

   \section{Local existence in Theorems \ref{th21}}
			\subsection{The existence without  vacuum}
This subsection aim at proving the local-in-time existence of the classical solution to the Cauchy problem  under $\phi^{\epsilon}_0\ge \epsilon$:
			\begin{equation}\label{858}\left\{\begin{aligned}
					&\phi^{\epsilon}_t+u^{\epsilon}\cdot\nabla\phi^{\epsilon}+(\gamma-1)\phi^{\epsilon} \text{div}u^{\epsilon}=0,\\[2pt]
					&u^{\epsilon}_t+u^{\epsilon}\cdot \nabla u^{\epsilon}+a_1\phi^{\epsilon}\nabla l^{\epsilon}+l^{\epsilon}\nabla\phi^{\epsilon}+a_2(l^\epsilon)^\nu h^{\epsilon} Lu^{\epsilon}\\[2pt]
					=& a_2 h^{\epsilon} \nabla (l^\epsilon)^\nu  \cdot Q(u^{\epsilon})+a_3(l^\epsilon)^\nu  \psi^{\epsilon} \cdot Q(u^{\epsilon}),\\[3pt]
					&l^{\epsilon}_t+u^{\epsilon}\cdot\nabla l^{\epsilon}-a_4(\phi^\epsilon)^{2\iota}(l^\epsilon)^{\nu}\Delta l^\epsilon\\
					=&a_5(l^\epsilon)^\nu  n^{\epsilon}(\phi^\epsilon)^{4\iota} H(u^{\epsilon})+a_6(l^\epsilon)^{\nu+1}\text{div}\psi^\epsilon+\Theta(\phi^\epsilon,l^\epsilon,\psi^\epsilon),\\[3pt]
					&(\phi^{\epsilon},u^{\epsilon},l^{\epsilon})|_{t=0}=(\phi^\epsilon_0,u^\epsilon_0,l^\epsilon_0)
					=(\phi_0+\epsilon,u_0,l_0)\quad\quad x\in\mathbb{R}^2,\\[3pt]
					&(\phi^{\epsilon},u^{\epsilon},l^{\epsilon})\rightarrow (\epsilon,0,\bar{l}) \quad \text{as} \hspace{2mm}|x|\rightarrow \infty \quad {\rm for}\quad t\geq 0.
				\end{aligned}\right.\end{equation}
			\begin{theorem}\label{ddgs}
				Assume that \ef{can1}, $\epsilon>0$, \ef{a}-\ef{2.8*} and \ef{2.14} hold, where constant $c_0>0$ independent of $\epsilon$. Then there exist  a unique strong solution $(\phi^\epsilon, u^\epsilon, l^\epsilon)$ in $\left[0, T_*\right] \times \mathbb{R}^2$ to \ef{858} satisfying \ef{2.13}, for some  $T_*>0$. Moreover,  \ef{857}-\ef{310857} hold for $\left(\phi^\epsilon, u^\epsilon, l^\epsilon\right)$ with $T^*$ replaced by $T_*$.
			\end{theorem}
            
			Now, we will use iteration scheme to prove Theorem \ref{th21}. 
We start with construct $\left(\phi^0, u^0, l^0\right)$, which is taken as the unique  solution to
			\begin{equation}\label{859}
				\left\{\begin{aligned}&
					U_t+u_0 \cdot \nabla U=0, \text { in }(0, \infty) \times \mathbb{R}^2, \\
					&Y_t-U^{2 \iota} \Delta Y=0, \text { in }(0, \infty) \times \mathbb{R}^2, \\
					& Z_t-U^{ 2\iota} \Delta Z=0, \text { in }(0, \infty) \times \mathbb{R}^2, \\
					&(U, Y, Z)|_{t=0}=(\phi_0^\epsilon, u_0^\epsilon, l_0^\epsilon)=(\phi_0+\epsilon, u_0, l_0) \text { in } \mathbb{R}^2, \\
					&(U, Y, Z) \rightarrow\left(\epsilon, 0, \bar{l}\right) \text { as }|x| \rightarrow \infty  \text { for }t \geq 0 .
				\end{aligned}\right.
			\end{equation}
Then choose $\bar{T} \in(0, T^*]$ sufficiently small such that  \ef{857}-\ef{310857} hold for $(\phi^0, u^0, l^0, h^0, \psi^0)$ with $T^*$ replaced by $\bar{T}$.
			
			\begin{proof}
				First, let $(v, w, g)=\big(u^0, l^0, (\phi^0)^{2\iota}\big)$, and one has a unique solution  $(\phi^1, u^1, l^1)$ to \ef{ln}. By induction, one can construct  $(\phi^k, u^k, l^k)$ for $k \geq 1$, and obtain $(\phi^{k+1}, u^{k+1}, l^{k+1})$ by solving
				\begin{equation}\label{860}
					\left\{\begin{aligned}
						& \phi_t^{k+1}+u^k \cdot \nabla \phi^{k+1}+(\gamma-1) \phi^{k+1} \text{div} u^k=0, \\
						& (l^{k+1})^{-\nu}(u_t^{k+1}+u^k \cdot \nabla u^k+a_1 \phi^{k+1} \nabla l^{k+1}+l^{k+1} \nabla \phi^{k+1}) \\
						& +a_2 h^{k+1} L u^{k+1}=a_2(l^{k+1})^{-\nu} h^k \nabla(l^{k+1})^\nu \cdot Q(u^k)+a_3 \psi^{k+1} \cdot Q(u^k), \\
						& l_t^{k+1}+u^k \cdot \nabla l^{k+1}-a_4h^{k+1}(l^k)^\nu \Delta l^{k+1} \\
						= & a_5(l^k)^\nu n^{k+1}(h^k)^{2} H(u^k)+a_6(l^k)^{\nu+1} \text{div} \psi^{k+1}+\Pi^{k+1}, \\
						& (\phi^{k+1}, u^{k+1}, l^{k+1})|_{t=0}=(\phi_0^\epsilon, u_0^\epsilon, l_0^\epsilon) 
						=  (\phi_0+\epsilon, u_0, l_0) \text { in } \mathbb{R}^2, \\
						& (\phi^{k+1}, u^{k+1}, l^{k+1}) \longrightarrow(\epsilon, 0, \bar{l}) \quad \text { as }\ |x| \rightarrow \infty \ \text { for }\  t \geq 0,
					\end{aligned}\right.
				\end{equation}
			where 
$$
						\Pi^{k+1}=  a_7(l^k)^{\nu+1}(h^{k+1})^{-1} \psi^{k+1} \cdot \psi^{k+1}+a_8(l^k)^\nu \nabla l^{k+1} \cdot \psi^{k+1} 
						 +a_9(l^k)^{\nu-1}h^k\nabla l^k \cdot \nabla l^k.$$
				Moreover,  $(\phi^{k+1}, u^{k+1}, l^{k+1})$ satisfies  \ef{857}-\ef{310857} uniformly in $k$.
				
				Now, we will show the compactness of $\left(\phi^k, u^k, l^k\right)$. Denote 
				$$
				\begin{aligned}
					\bar{f}^{k+1}=f^{k+1}-f^k,\quad \text{for}\quad f=\phi, u, l, \psi, h, n.
				\end{aligned}
				$$
				Then \ef{860} yield
				\begin{equation}\label{861}
					\left\{\begin{aligned}
						&\bar{\phi}_t^{k+1}+u^k \cdot \nabla \bar{\phi}^{k+1}+\bar{u}^k \cdot \nabla \phi^k+(\gamma-1)(\bar{\phi}^{k+1} \text{div} u^k+\phi^k \text{div} \bar{u}^k)=0, \\
						&(l^{k+1})^{-\nu} \bar{u}_t^{k+1}+a_2 h^{k+1} L \bar{u}^{k+1}+a_2 \bar{h}^{k+1} L u^k=\sum_{i=1}^4 \bar{\mathbb{U}}_i^{k+1}, \\
						& \bar{l}_t^{k+1}-a_4 h^{k+1}(l^k)^\nu \Delta \bar{l}^{k+1}=\sum_{i=1}^4 \bar{\mathbb{L}}_i^{k+1}+\bar{\Pi}^{k+1}, \\
						&(\bar{\phi}^{k+1}, \bar{u}^{k+1}, \bar{l}^{k+1})|_{t=0}=(0,0,0) \text { in } \mathbb{R}^2, \\
						&(\bar{\phi}^{k+1}, \bar{u}^{k+1}, \bar{l}^{k+1}) \longrightarrow(0,0,0) \quad \text { as } \quad|x| \rightarrow \infty \quad \text { for } \quad t \geq 0,
					\end{aligned}\right.
				\end{equation}
				where
				\begin{align*}
					\bar{\mathbb{U}}_1^{k+1}= & -(l^{k+1})^{-\nu}(u^k \cdot \nabla \bar{u}^k+\bar{u}^k \cdot \nabla u^{k-1}) \\
					& -\big((l^{k+1})^{-\nu}-(l^k)^{-\nu}\big)(u_t^k+u^{k-1} \cdot \nabla u^{k-1}), \\
					\bar{\mathbb{U}}_2^{k+1}= & -(l^{k+1})^{-\nu}(a_1 \bar{\phi}^{k+1} \nabla l^{k+1}+a_1 \phi^k \nabla \bar{l}^{k+1}+\bar{l}^{k+1} \nabla \phi^{k+1}+l^k \nabla \bar{\phi}^{k+1}) \\
					& -\big((l^{k+1})^{-\nu}-(l^k)^{-\nu}\big)(a_1 \phi^k \nabla l^k+l^k \nabla \phi^k), \\
					\bar{\mathbb{U}}_3^{k+1}= & a_2(l^{k+1})^{-\nu}\Big(h^k\big(\nabla(l^{k+1})^\nu-\nabla(l^k)^\nu\big) \cdot Q(u^k)+h^k \nabla(l^k)^\nu \cdot Q(\bar{u}^k) \\
					& +\bar{h}^k \nabla(l^k)^\nu \cdot Q(u^{k-1})\Big)+a_3 \bar{\psi}^{k+1} \cdot Q(u^k)+a_3 \psi^k \cdot Q(\bar{u}^k), \\
					\bar{\mathbb{U}}_4^{k+1}= & a_2\big((l^{k+1})^{-\nu}-(l^k)^{-\nu}\big) h^{k-1} \nabla(l^k)^\nu \cdot Q(u^{k-1}), \\
					\bar{\mathbb{L}}_1^{k+1}= & -(u^k \cdot \nabla \bar{l}^{k+1}+\bar{u}^k \cdot \nabla l^k),\\
					\bar{\mathbb{L}}_2^{k+1}= & a_4\Big(h^{k+1}\big((l^k)^\nu-(l^{k-1})^\nu\big)+\bar{h}^{k+1}(l^{k-1})^\nu\Big) \Delta l^k, \\
					\bar{\mathbb{L}}_3^{k+1}= & a_5(l^k)^\nu n^{k+1}\Big((h^k)^{2}\big(H(u^k)-H(u^{k-1})\big)+\big((h^k)^{2}-(h^{k-1})^{2}\big) H(u^{k-1})\Big) \\
					& +a_5(h^{k-1})^{2} H(u^{k-1})\Big((l^k)^\nu \bar{n}^{k+1}+\big((l^k)^\nu-(l^{k-1})^\nu\big) n^k\Big), \\
					\bar{\mathbb{L}}_4^{k+1}= & a_6(l^k)^{\nu+1} \text{div} \bar{\psi}^{k+1}+a_6\big((l^k)^{\nu+1}-(l^{k-1})^{\nu+1}\big)\text{div} \psi^k,\\
					\bar{\Pi}^{k+1}=& a_7(l^k)^{\nu+1}\Big((h^{k+1})^{-1} \bar{\psi}^{k+1} \cdot(\psi^{k+1}+\psi^k)+\big((h^{k+1})^{-1}-(h^k)^{-1}\big) \psi^k \cdot \psi^k\Big) \\
					& +a_7\big((l^k)^{\nu+1}-(l^{k-1})^{\nu+1}\big)(h^k)^{-1} \psi^k \cdot \psi^k \\
					& +a_8(l^k)^\nu(\nabla l^{k+1} \cdot \bar{\psi}^{k+1}+\nabla \bar{l}^{k+1} \cdot \psi^k) +a_8\big((l^k)^\nu-(l^{k-1})^\nu\big)\nabla l^k \cdot \psi^k \\
					& +a_9(l^k)^{\nu-1} h^k \nabla\bar{ l}^k \cdot(\nabla l^k+\nabla l^{k-1}) \\
					& +a_9\Big((l^k)^{\nu-1}\bar{h}^k+h^{k-1}\big((l^k)^{\nu-1}-(l^{k-1})^{\nu-1}\big)\Big)|\nabla l^{k-1}|^2 .
				\end{align*}
				
Moreover,  $\ef{860}_1$ implies
    \begin{equation}\label{nlpsiii}
    \begin{aligned}
       & \bar{h}_t^{k+1}+u^k \cdot \nabla \bar{h}^{k+1}+\bar{u}^k \cdot \nabla h^k+(\delta-1)(\bar{h}^{k+1} \text{div} u^k+h^k \text{div} \bar{u}^k)=0,\\
						&\bar{\psi}^{k+1}_t+\nabla(u^k\cdot\bar{\psi}^{k+1}+\bar{u}^k\cdot\psi^k)+(\delta-1)(\bar{\psi}^{k+1}\text{div}u^k+\psi^k\text{div}\bar{u}^k)\\
						&\qquad+a\delta(\bar{h}^{k+1}\nabla\text{div}u^k+h^k\nabla\text{div}\bar{u}^k)=0.
					\end{aligned}
				\end{equation}
                
				To show the compactness of the approximate solutions $\left(\phi^k, u^k, l^k\right)$, we start with the following estimates.
				\begin{lemma}
					$$
					(\bar{h}^{k+1}, \bar{\phi}^{k+1}) \in L^{\infty}([0, \bar{T}] ; H^3) \ \ \text {and} \ \ \bar{\psi}^{k+1} \in L^{\infty}([0, \bar{T}] ; H^2) \ \ \text { for } \quad k=1,2, \ldots
					$$
				\end{lemma} 
				\begin{proof}
				The proof is omitted for simplicity, see the proof of Lemma 3.11 in \cite{DXZ1} for details. 
				\end{proof}
				
				First, the classical energy estimates on $\eqref{861}_1$ gives
				\begin{equation}\label{862}
					\frac{\text{d}}{\text{d} t}|\bar{\phi}^{k+1}|_2^2 \leq C\big(|\nabla u^k|_{\infty}|\bar{\phi}^{k+1}|_2^2+(|\phi^k|_{\infty}|\nabla \bar{u}^k|_2+|\nabla\phi^k |_\infty |\bar{u}^k|_2)|\bar{\phi}^{k+1}|_2\big),
				\end{equation}
                %
				\begin{equation}\label{863}
    \begin{aligned}
        \frac{\text{d}}{\text{d} t}|\nabla \bar{\phi}^{k+1}|_2^2 \leq& C\|u^k\|_3|\nabla\bar{\phi}^{k+1}|_2^2+C(\|\phi^k\|_{L^{\infty} \cap D^1\cap D^{3}}\| \bar{u}^k\|_2\\&+|\bar{\phi}^{k+1}|_4|\nabla^2u^k|_4)|\nabla\bar{\phi}^{k+1}|_2,
    \end{aligned}	
				\end{equation}
                and
				\begin{equation}\label{864}
					\frac{\text{d}}{\text{d} t}|\nabla^2\bar{\phi}^{k+1}|_2^2 \leqslant C(\|u^k\|_3\| \bar{\phi}^{k+1}\|_2^2+\|\phi^k\|_{L^\infty\cap D^1 \cap D^3}\| \bar{u}^k\|_3|\bar{\phi}^{k+1}|_{D^2}).
				\end{equation}
                
Similarly, by $\eqref{nlpsiii}_1$, for $\bar{h}^{k+1}$, one has
\begin{equation}\label{1163}
					\frac{\text{d}}{\text{d} t}|\bar{h}^{k+1}|_2^2 \leq C\big(|\nabla u^k|_{\infty}|\bar{h}^{k+1}|_2^2+(|h^k\nabla \bar{u}^k|_2+|\nabla h^k |_\infty |\bar{u}^k|_2)|\bar{h}^{k+1}|_2\big),
				\end{equation}
by $\eqref{nlpsiii}_2$, for $\bar{\psi}^{k+1}$, one has 
				\begin{equation}\label{1164}
					\begin{aligned}
						 \frac{\text{d}}{\text{d} t}|\bar{\psi}^{k+1}|_2^2 \leq &C\|u^k\|_3|\bar{\psi}^{k+1}|_2^2+C( \|\psi^k\|_{L^{\infty} \cap D^1\cap D^{2}}\| \bar{u}^k\|_1\\&+|\bar{h}^{k+1}|_4|\nabla^2u^k|_4+|h^k \nabla^2 \bar{u}^k|_2)|\bar{\psi}^{k+1}|_2,&
					\end{aligned}
				\end{equation}

				and
				\begin{equation}\label{867}
					\begin{aligned}
						\frac{\text{d}}{\text{d} t}|\nabla \bar{\psi}^{k+1}|_2^2 \leq & C\|u^k\|_3\| \bar{\psi}^{k+1}\|_1^2+C(\|\psi^k\|_{L^{\infty} \cap D^{1}\cap D^2}\|\bar{u}^k\|_2
						\\&+|\bar{h}^{k+1}|_\infty|\nabla^3 u^k|_2+|h^k \nabla^3 \bar{u}^k|_2)|\nabla \bar{\psi}^{k+1}|_2.
					\end{aligned}
				\end{equation}

				For $\bar{l}^{k+1}$, by  $\ef{861}_3$, one has 
    \begin{equation}\label{1165}
				 \frac{\text{d}}{\text{d} t}|\bar{l}^{k+1}|_2^2+2a_4|(h^{k+1})^{\frac{1}{2}}(l^k)^\frac{\nu}{2} \nabla \bar{l}^{k+1}|_2^2=\sum_{i=1}^5 J_i, 
				\end{equation}
    where 
				
				\begin{align*}
					J_1= & \int2\Big(\bar{\mathbb{L}}_1^{k+1}-a_4\big((l^k)^\nu \nabla h^{k+1}+h^{k+1} \nabla(l^k)^\nu\big) \cdot \nabla \bar{l}^{k+1}\Big) \bar{l}^{k+1} \\
					\leq & C\big((|u^k|_\infty+|(l^k)^\nu|_\infty |\nabla h^{k+1}|_\infty) |\nabla \bar{l}^{k+1}|_2| \bar{l}^{k+1}|_2+|\nabla l^k|_\infty|\bar{u}^k|_2| \bar{l}^{k+1}|_2\\
					& +
|(h^{k+1})^\frac12\nabla(l^k)^\nu|_\infty|(l^k)^{-\frac{\nu}{2}}|_\infty|(h^{k+1})^{\frac{1}{2}}(l^k)^\frac{\nu}{2} \nabla \bar{l}^{k+1}|_2| \bar{l}^{k+1}|_2\big), \\
					J_2= & \int 2\bar{\mathbb{L}}_2^{k+1} \bar{l}^{k+1} 
                    \\
                    \leq & C(|h^{k+1}\nabla^2l^k|_2|\bar{l}^{k+1}|_4^2+|(l^{k-1})^{\nu}|_\infty|\bar{h}^{k+1}|_4|\nabla^2l^k|_4|\bar{l}^{k+1}|_2), \\
					J_3= & \int 2\bar{\mathbb{L}}_3^{k+1} \bar{l}^{k+1} \\ \leq &C\Big(|(l^k)^\nu|_\infty |n^{k+1}|_\infty\big(|h^k\nabla \bar{u}^{k}|_2(|h^k\nabla u^k|_\infty+|h^k\nabla u^{k-1}|_\infty)\\
					&+|\bar{h}^{k}|_2|\nabla u^{k-1}|_\infty(|h^k\nabla u^{k-1}|_\infty+|h^{k-1}\nabla u^{k-1}|_\infty) \big)  \\&+|h^{k-1}\nabla u^{k-1} |^2_\infty(|(l^k)^\nu|_\infty |\bar{h}^{k+1}|_2+|n^k|_\infty|\bar{l}^k|_2) \Big)|\bar{l}^{k+1}|_2, \\
					J_4= & \int 2\bar{\mathbb{L}}_4^{k+1} \bar{l}^{k+1} \\\leq& C(|(l^k)^{\nu}|_\infty|\nabla l^k|_\infty | \bar{\psi}^{k+1}|_2+|\bar{l}^k|_4|\nabla \psi^k|_4)|\bar{l}^{k+1}|_2\\&+C|(l^k)^{\nu+1}|_\infty | \bar{\psi}^{k+1}|_2|\nabla\bar{l}^{k+1}|_2,\\
					J_5= & \int2 \bar{\Pi}^{k+1} \cdot \bar{l}^{k+1} \\\leq& C\Big(|(l^k)^{\nu+1}|_\infty\big(|(h^{k+1})^{-1}|_\infty |\bar{\psi}^{k+1}|_2 (|\psi^{k+1}|_\infty+|\psi^k|_\infty)+|\bar{h}^{k+1}|_2 |\psi^k|_\infty^2\big) \\
					& +|\bar{l}^k|_2|(h^k)^{-1}|_\infty |\psi^k|_\infty^2  +|(l^k)^\nu|_\infty(|\nabla l^{k+1}|_\infty |\bar{\psi}^{k+1}|_2\\&+|\nabla \bar{l}^{k+1}|_2 |\psi^k|_\infty)  +|\bar{l}^k|_2|\nabla l^k|_\infty | \psi^k|_\infty \\&+|(l^k)^{\nu-1}|_\infty |(h^k)^\frac12 \nabla\bar{ l}^k|_2 (|(h^k)^\frac12\nabla l^k|_\infty+|(h^k)^\frac12\nabla l^{k-1}|_\infty) \\
					& +|(l^k)^{\nu-1}|_\infty|\bar{h}^k|_2|\nabla l^{k-1}|_\infty^2+|\bar{l}^{k}|_2|(h^{k-1})^\frac12\nabla l^{k-1}|_\infty^2\Big)|\bar{l}^{k+1}|_2,
				\end{align*}
			in which, we have used integration by parts and  
				\begin{equation}\label{871}
				  \begin{aligned}
					& |(h^{k+1})^{-1}-(h^k)^{-1}|_2=|(h^k-h^{k+1})(h^{k+1})^{-1}(h^k)^{-1}|_2\leq C|\bar{h}^{k+1}|_2, \\
					& |h^{k+1}(h^k)^{-1}|_{\infty}+|h^k(h^{k+1})^{-1}|_{\infty} \leq C, \\
					&  (h^{k})^{2}-(h^{k-1})^{2}=(h^{k}-h^{k-1})(h^{k}+h^{k-1}), \;|\bar{l}^{k+1}|_2 \leq C|\bar{l}^{k+1}|_2^{\frac{1}{2}}| \nabla\bar{l}^{k+1}|_2^{\frac{1}{2}},\\
					& a^{-b}|\bar{n}^{k+1}|_2=|(h^{k+1})^b-(h^k)^b|_2 \\&\qquad\qquad\;\;\leq C(|(h^{k+1})^{b-1}|_\infty+|(h^k)^{b-1}|_\infty)|\bar{h}^{k+1}|_2\leq C|\bar{h}^{k+1}|_2, \\
					& |(l^k)^{\nu+1}-(l^{k-1})^{\nu+1}|_2\le C(|(l^k)^{\nu}|_\infty+|(l^{k-1})^{\nu}|_\infty)|\bar{l}^k|_2\leq C|\bar{l}^k|_2 , \\
					& |(l^k)^{\nu}-(l^{k-1})^{\nu}|_2\le C(|(l^k)^{\nu-1}|_\infty+|(l^{k-1})^{\nu-1}|_\infty)|\bar{l}^k|_2\leq C|\bar{l}^k|_2 .
				\end{aligned}  
				\end{equation}


		For convenience, in the following estimation, the quantity that can be derived from \ef{857}-\ef{310857} is directly controlled by the constant $C$.	
      Now, the classical energy estimates on $\ef{861}_3$ gives
				\begin{equation}\label{869}
				a_4 \frac{\text{d}}{\text{d} t}|(h^{k+1})^{\frac{1}{2}}(l^k)^{\frac{\nu}{2}} \nabla \bar{l}^{k+1}|_2^2+| \bar{l}_t^{k+1}|_2^2=\sum_{i=6}^{10} J_i, 
				\end{equation}
                where 
			\begin{align*}
					J_6= & \int2\Big(\bar{\mathbb{L}}_1^{k+1}-a_4\big((l^k)^\nu \nabla h^{k+1}+h^{k+1}\nabla(l^k)^\nu\big) \cdot \nabla \bar{l}^{k+1}\Big) \bar{l}_t^{k+1} \\
					& +\int a_4\Big(h^{k+1}_t(l^k)^\nu+h^{k+1}\big((l^k)^\nu\big)_t\Big) \nabla \bar{l}^{k+1} \cdot \nabla \bar{l}^{k+1} \\
					\leq & C(|(h^{k+1})^{\frac{1}{2}}(l^k)^{\frac{\nu}{2}} \nabla \bar{l}^{k+1}|_2+| \bar{u}^k|_2)| \bar{l}_t^{k+1}|_2 \\&+C(1+|l_t^k|_{D^2}^{\frac{1}{4}})|(h^{k+1})^{\frac{1}{2}}(l^k)^{\frac{\nu}{2}} \nabla \bar{l}^{k+1}|_2^2, \\
					J_7= & \int 2\bar{\mathbb{L}}_2^{k+1} \bar{l}_t^{k+1} \leq C(|\bar{l}^k|_4+|\bar{h}^{k+1}|_4)| \bar{l}_t^{k+1}|_2, \\
					J_8= & \int2 \bar{\mathbb{L}}_3^{k+1} \bar{l}_t^{k+1} \leq C(|h^k \nabla \bar{u}^k|_2+|\bar{h}^k|_2 +|\bar{h}^{k+1}|_2+| \bar{l}^k|_2)|\bar{l}_t^{k+1}|_2, \\
					J_9= & \int 2\bar{\mathbb{L}}_4^{k+1} \bar{l}_t^{k+1} \leq C(\|\bar{\psi}^{k+1}\|_1+| \bar{l}^k|_4)| \bar{l}_t^{k+1}|_2, \\
					J_{10}= & \int 2\bar{\Pi}^{k+1} \cdot \bar{l}_t^{k+1} \\\leq& C| \bar{l}_t^{k+1}|_2(|(l^{k-1})^{\frac{\nu}{2}}(h^k)^{\frac{1}{2}} \nabla \bar{l}^k|_2+|\bar{\psi}^{k+1}|_2 \\
					& +|(l^k)^{\frac{\nu}{2}}(h^{k+1})^{\frac{1}{2}} \nabla \bar{l}^{k+1}|_2+|  \bar{l}^k|_2+|\bar{h}^{k+1}|_2+|\bar{h}^{k}|_2).
				\end{align*}              
Next, applying $\partial_t$ to $\ef{861}_3$,
 multiplying the resulting equation by $2\bar{l}_t^{k+1}$, and integrating over $\mathbb{R}^2$, one gets
\begin{equation}\label{874}
 \frac{\text{d}}{\text{d} t}| \bar{l}_t^{k+1}|_2^2+2a_4|(h^{k+1})^{\frac{1}{2}}(l^k)^{\frac{\nu}{2}} \nabla \bar{l}_t^{k+1}|_2^2=\sum_{i=11}^{15} J_i,   
\end{equation}
where, by using  \ef{871}, one has
\begin{align*}
J_{11}= & \int2(\bar{\mathbb{L}}_1^{k+1})_t \bar{l}_t^{k+1} \\
	\leq & C(\| \bar{u}^k\|_2 +|(h^{k+1})^{\frac{1}{2}}(l^k)^{\frac{\nu}{2}} \nabla \bar{l}_t^{k+1}|_2+| \bar{u}_t^k|_2)| \bar{l}_t^{k+1}|_2\\
&+C|(h^{k+1})^{\frac{1}{2}}(l^k)^{\frac{\nu}{2}} \nabla \bar{l}^{k+1}|_2| \bar{l}_t^{k+1}|_4, \\
J_{12}= & \int2\Big((\bar{\mathbb{L}}_2^{k+1})_t+a_4\big(h^{k+1}(l^k)^\nu\big)_t \Delta \bar{l}^{k+1} \\&-a_4 \big(\nabla h^{k+1}(l^k)^\nu \cdot \nabla \bar{l}_t^{k+1}+h^{k+1} \nabla(l^k)^\nu \nabla \bar{l}_t^{k+1}\big) \Big) \bar{l}_t^{k+1} \\
	\leq & C(| \bar{l}^k|_4 +|\bar{h}^{k+1}|_\infty+\|\bar{h}_t^{k+1}\|_1 +|h^{k+1} \nabla^2 \bar{l}^{k+1}|_2\\
	& +|(h^{k+1})^{\frac{1}{2}}(l^k)^{\frac{\nu}{2}} \nabla \bar{l}_t^{k+1}|_2)| \bar{l}_t^{k+1}|_2\\
	&+C|(h^k)^\frac12 \nabla^2 l_t^k|_2| (h^k)^\frac14\bar{l}^k|_4|(h^{k+1})^\frac14\bar{l}_t^{k+1}|_4\\
	& +C(| \nabla^2 l_t^k|_2| \bar{h}^{k+1}|_4+| \bar{l}_t^k|_4+|h^{k+1} \nabla^2 \bar{l}^{k+1}|_2) |\bar{l}_t^{k+1}|_4, \\
		J_{13}= & \int2(\bar{\mathbb{L}}_3^{k+1})_t \bar{l}_t^{k+1} \\\leq& C(|h^k \nabla \bar{u}^k|_2+|h^k\nabla \bar{u}_t^k|_2+|\bar{h}_t^k|_2+|\bar{h}_t^{k+1}|_2\\
				&+| \bar{l}^k_t|_2+| \bar{l}^k|_\infty+|\bar{h}^{k+1}|_\infty+|\bar{h}^k|_\infty)| \bar{l}_t^{k+1}|_2 +C|h^k \nabla \bar{u}^k|_4|\bar{l}_t^{k+1}|_4,\\
J_{14}= & \int2(\bar{\mathbb{L}}_4^{k+1})_t \bar{l}_t^{k+1} \\\leq& C(| \bar{l}^k|_\infty +|\bar{l}^k_t|_4)| \bar{l}_t^{k+1}|_2 +|\nabla \bar{\psi}^{k+1}|_2| \bar{l}_t^{k+1}|_4+J_*, \\
	J_{15}= & \int 2\bar{\Pi}_t^{k+1} l_t^{k+1} \\\leq& C(\|\bar{\psi}^{k+1}\|_1+|\bar{\psi}_t^{k+1}|_2+\|\bar{h}_t^k\|_1+\|\bar{h}^k\|_2 \\&+\|\bar{h}_t^{k+1}\|_1+|\bar{h}^{k+1}|_2 +|  \bar{l}^k|_\infty+|\bar{l}^k_t|_2\\&+|(h^k)^{\frac{1}{2}}(l^{k-1})^{\frac{\nu}{2}} \nabla \bar{l}^k|_2 +|(h^k)^{\frac{1}{2}} \nabla \bar{l}_t^k|_2+|(h^{k+1})^{\frac{1}{2}} \nabla \bar{l}_t^{k+1}|_2)|\bar{l}_t^{k+1}|_2\\
 &+C(|\bar{\psi}^{k+1}|_2+|(h^k)^{\frac{1}{2}} \nabla \bar{l}^k|_4+|(h^{k+1})^{\frac{1}{2}}(l^k)^{\frac{\nu}{2}} \nabla \bar{l}^{k+1}|_2)| \bar{l}_t^{k+1}|_4.
\end{align*}		    
 In which, via using $\eqref{nlpsiii}_2$ and integration by parts, the term $J_*$ in $J_{14}$ can be controlled  as		

    \begin{align*}
J_*= & 2a_6 \int(l^k)^{\nu+1} \text{div} \bar{\psi}_t^{k+1} \bar{l}_t^{k+1} \\
	= & -2a_6 \int(l^k)^{\nu+1} \bar{l}_t^{k+1} \text{div}\big(\nabla(u^k \cdot \bar{\psi}^{k+1})+\nabla(\bar{u}^k \cdot \psi^k) \\
	& +(\delta-1)(\bar{\psi}^{k+1}\text{div} u^k+\psi^{k} \text{div} \bar{u}^k)+a \delta(\bar{h}^{k+1} \nabla \text{div} u^k+h^k \nabla \text{div} \bar{u}^{k})\big) \\
		\leq & C(\|\bar{\psi}^{k+1}\|_1+|\bar{h}^{k+1}|_\infty+\|\bar{u}^k\|_1+|h^k\nabla \bar{u}^k|_2+|h^k \nabla^2 \bar{u}^k|_2 \\
	& +|h^k \nabla^3 \bar{u}^k|_2)| \bar{l}_t^{k+1}|_2+C\|\bar{\psi}^{k+1}\|_1| \nabla \bar{l}_t^{k+1}|_2 .
\end{align*}

Furthermore,  $\eqref{nlpsiii}$ gives that
\begin{equation}\label{878}
\begin{aligned}
 |\bar{h}_t^{k+1}|_2 &\leq C(|\bar{\psi}^{k+1}|_2+|\bar{h}^{k+1}|_2+\|\bar{u}^k\|_1), \\
    |\bar{\psi}_t^{k+1}|_2 &\leq C(\|\bar{\psi}^{k+1}\|_1+|\bar{h}^{k+1}|_\infty+\|\bar{u}^k\|_1+|h^k \nabla^2 \bar{u}^k|_2). 
\end{aligned}
\end{equation}

    Now, 
			by	$\ef{861}_2$, one has
		\begin{equation}\label{1167}
		    \begin{aligned}
					& \frac{\text{d}}{\text{d} t}|(l^{k+1})^{-\frac{\nu}{2}}(h^{k+1})^\frac12 \bar{u}^{k+1}|_2^2+2a_2 \alpha|h^{k+1} \nabla \bar{u}^{k+1}|_2^2\\&+2a_2(\alpha+\beta)|h^{k+1} \text{div} \bar{u}^{k+1}|_2^2 \\
					\leq&  C \big(|(l^{k+1})^{-\frac{\nu}{2}}(h^{k+1})^\frac12 \bar{u}^{k+1}|_2^2+(|(l^{k+1})^{-\frac{\nu}{2}}(h^{k+1})^\frac12 \bar{u}^{k+1}|_4+|h^{k+1} \nabla \bar{u}^{k+1}|_2\\&+|\bar{h}^{k+1}|_2+|\bar{l}^{k+1}|_2+\|\bar{\phi}^{k+1}\|_1+|(h^{k+1})^\frac12\nabla\bar{l}^{k+1}|_2+|h^k\nabla\bar{u}^{k}|_2\\&+|\bar{h}^{k}|_2+|\bar{\psi}^{k+1}|_2)|(l^{k+1})^{-\frac{\nu}{2}}(h^{k+1})^\frac12 \bar{u}^{k+1}|_2+|h^{k+1} \nabla \bar{u}^{k+1}|_2\|\bar{\phi}^{k+1}\|_1\big),
				\end{aligned}
		\end{equation}
			%
				and
				\begin{equation}\label{1168}
				   \begin{aligned}
					& \frac{\text{d}}{\text{d} t}\big(a_2 \alpha|h^{k+1} \nabla \bar{u}^{k+1}|_2^2 +a_2(\alpha+\beta)|h^{k+1} \text{div} \bar{u}^{k+1}|_2^2\big)\\&+2|(l^{k+1})^{-\frac{\nu}{2}}(h^{k+1})^\frac12 \bar{u}_t^{k+1}|_2^2=\sum_{i=16}^{21} J_i,
				\end{aligned} 
				\end{equation}
				where 
				\begin{align*}
					J_{16}= & - 2a_2 \int \bar{h}^{k+1} L u^k \cdot h^{k+1}\bar{u}_t^{k+1} \leq C|\bar{h}^{k+1}|_4|(l^{k+1})^{-\frac{\nu}{2}} (h^{k+1})^\frac12\bar{u}_t^{k+1}|_2, \\
					J_{17}= &  \int2\big(\overline{\mathbb{U}}_1^{k+1}-2a_2\frac{\delta-1}{a \delta}  \psi^{k+1} \cdot Q(\bar{u}^{k+1})\big) \cdot h^{k+1}\bar{u}_t^{k+1} \\
					\leq & C(| \nabla \bar{u}^k|_2+| \bar{u}^k|_2+| \bar{l}^{k+1}|_2
					 +|h^{k+1} \nabla \bar{u}^{k+1}|_2)|(l^{k+1})^{-\frac{\nu}{2}}(h^{k+1})^\frac12 \bar{u}_t^{k+1}|_2, \\
					J_{18}= &  \int2 \overline{\mathbb{U}}_2^{k+1} \cdot h^{k+1}\bar{u}_t^{k+1} \\
					\leq & C(\|\bar{\phi}^{k+1}\|_1+|(h^{k+1})^{\frac{1}{2}} \nabla \bar{l}^{k+1}|_2+| \bar{l}^{k+1}|_4)|(l^{k+1})^{-\frac{\nu}{2}}(h^{k+1})^\frac12 \bar{u}_t^{k+1}|_2\\&+C\|\bar{\phi}^{k+1}\|_1|h^{k+1} \nabla \bar{u}_t^{k+1}|_2, \\
					J_{19}= &  \int2 \overline{\mathbb{U}}_3^{k+1} \cdot h^{k+1}\bar{u}_t^{k+1} \\\leq& C(|(h^{k+1})^{\frac{1}{2}} \nabla \bar{l}^{k+1}|_2+|\bar{h}^k|_2 
					 +|\bar{\psi}^{k+1}|_2\\
                     &+|h^k \nabla \bar{u}^k|_2)|(l^{k+1})^{-\frac{\nu}{2}} (h^{k+1})^\frac12\bar{u}_t^{k+1}|_2, \\
					J_{20}= & \int2 \overline{\mathbb{U}}_4^{k+1} \cdot h^{k+1}\bar{u}_t^{k+1} \leq C|\bar{l}^{k+1}|_2|(l^{k+1})^{-\frac{\nu}{2}}(h^{k+1})^\frac12 \bar{u}_t^{k+1}|_2, \\
					J_{21}= & 2a_2 \int h^{k+1}h_t^{k+1}\big(\alpha|\nabla \bar{u}^{k+1}|^2+(\alpha+\beta)|\text{div} \bar{u}^{k+1}|^2\big) \leq C|h^{k+1} \nabla \bar{u}^{k+1}|_2^2 .
				\end{align*}	
				Next, applying $\partial_t$ to $\ef{861}_2$, 
		 multiplying the resulting  system by $2 h^{k+1}\bar{u}_t^{k+1}$, and integrating over $\mathbb{R}^2$, one gets
			\begin{equation}\label{1570}
			    \begin{aligned}
					& \frac{\text{d}}{\text{d} t}|(l^{k+1})^{-\frac{\nu}{2}} (h^{k+1})^{\frac{1}{2}}\bar{u}_t^{k+1}|_2^2+2 a_2 \alpha|h^{k+1} \nabla \bar{u}_t^{k+1}|_2^2 \\
					& +2 a_2(\alpha+\beta)|h^{k+1} \text{div} \bar{u}_t^{k+1}|_2^2=\sum_{i=22}^{26} J_i,
				\end{aligned}
			\end{equation}
				where 
				    \begin{align*}
					 J_{22}=&\int\Big(-2\big((l^{k+1})^{-\nu}\big)_th^{k+1}(\bar{u}_t^{k+1})^2+2( \overline{\mathbb{U}}_1^{k+1})_t \cdot h^{k+1}\bar{u}_t^{k+1}\\&-(l^{k+1})^{-\nu}(h^{k+1})_t(\bar{u}_t^{k+1})^2\Big) \\
					 \leq& C(1+|l_t^{k+1}|_{D^2}^{\frac{1}{4}})|(l^{k+1})^{-\frac{\nu}{2}}(h^{k+1})^{\frac{1}{2}} \bar{u}_t^{k+1}|_2^2+C(|h^k \nabla \bar{u}_t^k|_2\\
					&+| \bar{l}_t^{k+1}|_2+| \bar{u}_t^k|_2  +\| \bar{u}^k\|_2)|(l^{k+1})^{-\frac{\nu}{2}} (h^{k+1})^{\frac{1}{2}}\bar{u}_t^{k+1}|_2 \\
					& +C(1+|u_{t t}^k|_2)(|\bar{l}^{k+1}|_2+|(h^{k+1})^{\frac{1}{2}}\nabla\bar{l}^{k+1}|_2)(|(l^{k+1})^{-\frac{\nu}{2}} (h^{k+1})^{\frac{1}{2}}\bar{u}_t^{k+1}|_2 \\
					&+|h^{k+1} \nabla \bar{u}_t^{k+1}|_2)  +C| \bar{l}_t^{k+1}|_2|h^{k+1} \nabla \bar{u}_t^{k+1}|_2, \\
					 J_{23}=&\int-2 a_2\big(2\nabla h^{k+1} \cdot \hat{Q}(\bar{u}_t^{k+1})+h_t^{k+1} L \bar{u}^{k+1}+(\bar{h}^{k+1} L u^k)_t\big) \cdot h^{k+1}\bar{u}_t^{k+1} \\
					 \leq &C(|h^{k+1} \nabla \bar{u}_t^{k+1}|_2+|h^{k+1}\nabla^2 \bar{u}^{k+1}|_2+|\bar{h}_t^{k+1}|_2+|\bar{\psi}_t^{k+1}|_2 \\
					& +|h^k\nabla^2 u_t^k|_2\|\bar{h}^{k+1}\|_2)|(l^{k+1})^{-\frac{\nu}{2}}(h^{k+1})^{\frac{1}{2}} \bar{u}_t^{k+1}|_2, \\
					 J_{24}=&\int 2(\overline{\mathbb{U}}_2^{k+1})_t \cdot h^{k+1} \bar{u}_t^{k+1} \\\leq& C\big(\|\bar{\phi}^{k+1}\|_2+|\bar{\phi}_t^{k+1}|_2+|(h^{k+1})^\frac12\nabla\bar{\phi}_t^{k+1}|_2\\
					& +|\bar{l}_t^{k+1}|_2 +(1+|l_t^k|_{D^2}^{\frac{1}{2}}+|l_t^{k+1}|_{D^2}^{\frac{1}{2}})(| \bar{l}^{k+1}|_2\\
					&+|(h^{k+1})^{\frac{1}{2}} \nabla \bar{l}^{k+1}|_2)\big)|(l^{k+1})^{-\frac{\nu}{2}}(h^{k+1})^{\frac{1}{2}} \bar{u}_t^{k+1}|_2 \\
					& +C(| \bar{l}_t^{k+1}|_2+\| \bar{l}^{k+1}\|_1)|h^{k+1} \nabla \bar{u}_t^{k+1}|_2,\\ 
					 J_{25}=&\int 2(\overline{\mathbb{U}}_3^{k+1})_t \cdot h^{k+1}\bar{u}_t^{k+1} \\\leq& C\big((1+|h^k \nabla^2 u_t^k|_2)|(h^{k+1})^{\frac{1}{2}}(l^k)^{\frac{\nu}{2}} \nabla \bar{l}^{k+1}|_2\\
					& +|\bar{ l}_t^{k+1}|_2 +\|\bar{h}^k\|_2+|h^k\nabla \bar{u}_t^k|_2 +|\bar{h}_t^k|_2+|\bar{\psi}_t^{k+1}|_2\\
					& +\|\bar{\psi}^{k+1}\|_1\big)|(l^{k+1})^{-\frac{\nu}{2}} (h^{k+1})^{\frac{1}{2}}\bar{u}_t^{k+1}|_2+C|h^k\nabla \bar{u}^k|_4|(h^{k+1})^{\frac{1}{2}}\bar{u}_t^{k+1}|_4 \\
					& +C\big((1+|h^k \nabla^2 u_t^k|_2)|(h^{k+1})^{\frac{1}{2}}(l^k)^{\frac{\nu}{2}} \nabla \bar{l}^{k+1}|_2 \\
					& +| \bar{l}_t^{k+1}|_2+\|\bar{\psi}^{k+1}\|_1\big)|h^{k+1} \nabla \bar{u}_t^{k+1}|_2, \\
					 J_{26}=&\int 2(\overline{\mathbb{U}}_4^{k+1})_t \cdot h^{k+1}\bar{u}_t^{k+1} \\\leq& C\|  \bar{l}^{k+1}\|_1|h^{k+1} \nabla \bar{u}_t^{k+1}|_2 \\
                     &+C(|\bar{l}_t^{k+1}|_2+\| \bar{l}^{k+1}\|_1)|(l^{k+1})^{-\frac{\nu}{2}} (h^{k+1})^{\frac{1}{2}}\bar{u}_t^{k+1}|_2,
				\end{align*}
				in which, we have used integration by parts in $J_{24}$ and $J_{25}$ to deal with the terms related to $\nabla \bar{l}_t^{k+1}$, and
    \begin{equation}\label{2678}
\begin{aligned}
 |\bar{\phi}_t^{k+1}|_2 &\leq C(|\nabla\bar{\phi}^{k+1}|_2+|\bar{\phi}^{k+1}|_2+\|\bar{u}^k\|_1), \\
    |(h^{k+1})^\frac12\nabla\bar{\phi}_t^{k+1}|_2 &\leq C(\|\nabla\bar{\phi}^{k+1}\|_1+|\bar{\phi}^{k+1}|_\infty+\|\bar{u}^k\|_1+|h^k \nabla^2 \bar{u}^k|_2). 
\end{aligned}
\end{equation}
				
				
				Moreover, according to  
				$\ef{861}_2$-$\ef{861}_3$ and Lemma \ref{ell}, one has 
			    \begin{align*}
      |(h^{k+1})^\frac12 \nabla^2 \bar{l}^{k+1}|_2 \leq & C(|\bar{l}_t^{k+1}|_2+| \bar{u}^k|_2+|\bar{l}^{k+1}|_2+|\bar{l}^{k}|_2+|\bar{h}^{k}|_2\\&+|\bar{h}^{k+1}|_2+|h^k \nabla \bar{u}^k|_2+|(l^k)^{\frac{\nu}{2}}(h^{k+1})^{\frac{1}{2}} \nabla \bar{l}^{k+1}|_2 \\
& +|(l^{k-1})^{\frac{\nu}{2}}(h^k)^{\frac{1}{2}} \nabla \bar{l}^k|_2+|\bar{\psi}^k|_2+\|\bar{\psi}^{k+1}\|_1), \\
					|h^{k+1} \nabla^2 \bar{u}^{k+1}|_2 \leq & C(|(l^{k+1})^{-\frac{\nu}{2}} \bar{u}_t^{k+1}|_2+|h^{k+1} \nabla \bar{u}^{k+1}|_2+|h^k \nabla \bar{u}^k|_2\\&+|\bar{u}^{k+1}|_2+|\bar{u}^k|_2 +|| \bar{\phi}^{k+1} \|_1+|\bar{h}^k|_2+|\bar{h}^{k+1}|_2\\
					&+|\bar{\psi}^{k+1}|_2 +|(l^k)^{\frac{\nu}{2}}(h^{k+1})^{\frac{1}{2}} \nabla \bar{l}^{k+1}|_2+|\bar{l}^{k+1}|_2),\\
|h^{k+1} \nabla^3 \bar{u}^{k+1}|_2 \leq & C(|(l^{k+1})^{-\frac{\nu}{2}} \bar{u}_t^{k+1}|_2+|h^{k+1} \nabla \bar{u}_t^{k+1}|_2+|h^{k+1} \nabla \bar{u}^{k+1}|_2 \\
& +|h^k \nabla \bar{u}^k|_2+|\bar{u}^{k+1}|_2+|\bar{u}^k|_2+|\bar{l}^{k+1}|_2+|\bar{l}^{k}|_2\\&+| \bar{l}_t^{k+1}|_2+|(l^{k-1})^{\frac{\nu}{2}}(h^k)^{\frac{1}{2}} \nabla \bar{l}^k|_2  +\|\bar{\phi}^{k+1}\|_2+\|\bar{h}^k\|_2\\
&+\|\bar{h}^{k+1}\|_2+|(l^k)^{\frac{\nu}{2}}(h^{k+1})^{\frac{1}{2}} \nabla \bar{l}^{k+1}|_2  +|(l^k)^{-\frac{\nu}{2}} \bar{u}_t^k|_2\\
&+|h^{k-1} \nabla \bar{u}^{k-1}|_2+|\bar{u}^{k-1}|_2+\|\bar{\phi}^k\|_1+|\bar{h}^{k-1}|_2),
				\end{align*}
		which, along with \ef{862}-\ef{1165}, \ef{869}-\ef{874} and \ef{1167}-\ef{1570}, yields that
				\begin{equation}\label{1171}
				  \begin{aligned}
					& \frac{\text{d}}{\text{d} t}\big(\|\bar{\phi}^{k+1}\|_2^2+|\bar{h}^{k+1}|_2^2+\|\bar{\psi}^{k+1}\|_1^2+| \bar{l}^{k+1}|_2^2+|(h^{k+1})^{\frac{1}{2}}(l^k)^{\frac{\nu}{2}} \nabla \bar{l}^{k+1}|_2^2\\&+| \bar{l}_t^{k+1}|_2^2+|(l^{k+1})^{-\frac{\nu}{2}}(h^{k+1})^\frac12 \bar{u}^{k+1}|_2^2+ \alpha|h^{k+1} \nabla \bar{u}^{k+1}|_2^2\\
					&  +(\alpha+\beta)|h^{k+1} \text{div} \bar{u}^{k+1}|_2^2+|(l^{k+1})^{-\frac{\nu}{2}} (h^{k+1})^{\frac{1}{2}}\bar{u}_t^{k+1}|_2^2\big)  \\&+a_4|(h^{k+1})^{\frac{1}{2}}(l^k)^{\frac{\nu}{2}} \nabla \bar{l}_t^{k+1}|_2^2+ a_2 \alpha|h^{k+1} \nabla \bar{u}_t^{k+1}|_2^2 \\
					\leq & \mathcal{C}^k(t)\big(\|\bar{\phi}^{k+1}\|_2^2+|\bar{h}^{k+1}|_2^2+\|\bar{\psi}^{k+1}\|_1^2+| \bar{l}^{k+1}|_2^2+|(h^{k+1})^{\frac{1}{2}}(l^k)^{\frac{\nu}{2}} \nabla \bar{l}^{k+1}|_2^2\\&+| \bar{l}_t^{k+1}|_2^2+|(l^{k+1})^{-\frac{\nu}{2}}(h^{k+1})^\frac12 \bar{u}^{k+1}|_2^2+ |h^{k+1} \nabla \bar{u}^{k+1}|_2^2\\
					&  +|(l^{k+1})^{-\frac{\nu}{2}} (h^{k+1})^{\frac{1}{2}}\bar{u}_t^{k+1}|_2^2\big)+\widetilde{\sigma}^{-2}\big(\|\bar{\phi}^{k+1}\|_2^2+\|\bar{\psi}^{k+1}\|_1^2\\&+| \bar{l}_t^{k+1}|_2^2+|(l^{k+1})^{-\frac{\nu}{2}} (h^{k+1})^{\frac{1}{2}}\bar{u}_t^{k+1}|_2^2\big) +\sigma\big(|h^k \nabla \bar{u}^k|_2^2+\|\bar{\phi}^k\|_1^2\\
     &+\|\bar{\psi}^k\|_1^2+|\bar{u}^k|_2^2+|h^{k-1} \nabla \bar{u}^{k-1}|_2^2  +(1+|(h^k)^\frac12 \nabla^2 l_t^k|_2)(|\bar{l}^{k}|_2^2\\
&+|(l^{k-1})^{\frac{\nu}{2}}(h^k)^{\frac{1}{2}} \nabla \bar{l}^k|_2^2)+|\bar{u}_t^k|_2^2+|\bar{l}_t^k|_2^2+|\bar{l}^{k-1}|_2^2+|\bar{h}^{k}|_2^2\\&+|\bar{h}^{k-1}|_2^2  +\|\bar{\psi}^{k-1}\|_1^2+\|\bar{u}^{k-1}\|_1^2+|(l^{k-2})^{\frac{\nu}{2}}(h^{k-1})^{\frac{1}{2}} \nabla \bar{l}^{k-1}|_2^2\\
&+|(l^{k-1})^{-\frac{\nu}{2}} \bar{u}_t^{k-1}|_2^2  +|h^{k-2} \nabla \bar{u}^{k-2}|_2^2+|\bar{u}^{k-2}|_2^2+\|\bar{\phi}^{k-1}\|_1^2\\
&+|\bar{h}^{k-2}|_2^2\big)+\widetilde{\sigma}\big(|(h^k)^{\frac{1}{2}} (l^{k-1})^{\frac{\nu}{2}}\nabla \bar{l}_t^k|_2^2+|h^k \nabla \bar{u}_t^k|_2^2\big),\\
				\end{aligned}  
				\end{equation}
				where constants $\sigma$ and $\widetilde{\sigma}$  will be determined later, and
				$$
				\begin{aligned}			\mathcal{C}^k(t)= & C \sigma^{-3}(1+|(h^k)^\frac12 \nabla^2 l_t^k|_2^2+|(h^{k+1})^\frac12 \nabla^2 l_t^{k+1}|_2^2+|h^k \nabla^2 u_t^k|_2^2+|u_{t t}^k|_2^2).
				\end{aligned}
				$$

				Now, define
				$$
				\begin{aligned}
					\Gamma^{k+1}(t)= & \sup _{0 \leq s \leq t}\|\bar{\phi}^{k+1}\|_2^2+\sup _{0 \leq s \leq t}|\bar{h}^{k+1}|_2^2+\sup _{0 \leq s \leq t}\|\bar{\psi}^{k+1}\|_1^2+\sup _{0 \leq s \leq t}|\bar{l}^{k+1}|_2^2\\&+\sup _{0 \leq s \leq t}|(h^{k+1})^{\frac{1}{2}}(l^k)^{\frac{\nu}{2}} \nabla \bar{l}^{k+1}|_2^2  +\sup _{0 \leq s \leq t}|(l^{k+1})^{-\frac{\nu}{2}} (h^{k+1})^{\frac{1}{2}}\bar{u}^{k+1}|_2^2 \\
					& +\sup _{0 \leq s \leq t}|h^{k+1} \nabla \bar{u}^{k+1}|_2^2+\sup _{0 \leq s \leq t}| \bar{l}_t^{k+1}|_2^2+\sup _{0 \leq s \leq t}|(l^{k+1})^{-\frac{\nu}{2}} (h^{k+1})^{\frac{1}{2}}\bar{u}_t^{k+1}|_2^2 .
				\end{aligned}
				$$
				Then it follows from \ef{1171} and Gronwall's inequality that
				$$
				\begin{aligned}
					& \quad \Gamma^{k+1}(t)+\int_0^t(|(h^{k+1})^{\frac{1}{2}}(l^k)^{\frac{\nu}{2}} \nabla \bar{l}_s^{k+1}|_2^2+|h^{k+1} \nabla \bar{u}_s^{k+1}|_2^2) \text{d} s \\
					& \leq C\Big(\int_0^t \widetilde{\sigma}(|(h^k)^{\frac{1}{2}} (l^{k-1})^{\frac{\nu}{2}}\nabla \bar{l}_s^k|_2^2+|h^k \nabla \bar{u}_s^k|_2^2) \text{d} s+(t+t^\frac12) \sigma \Gamma^k(t) \\
					& \quad+t \sigma \Gamma^{k-1}(t)+t \sigma \Gamma^{k-2}(t)\Big) \exp (C \sigma^{-3} t+C \sigma^{-3}+\widetilde{\sigma}^{-2} t) .
				\end{aligned}
				$$
			Thus choose $\sigma \in(0, \min \{1, a_4, a_2 \alpha\}), T_* \in(0, \min \{1, \bar{T}\}]$ and $\widetilde{\sigma} \in(0, 1)$ satisfy
			\begin{equation*}
			    \begin{aligned}
		C \widetilde{\sigma} \exp (C \sigma^{-3} T_*+C\sigma^{-3}+\widetilde{\sigma}^{-2} T_*) \leq \frac{1}{32},\\
     C \sigma T_*^{\frac12} \exp (C \sigma^{-3} T_*+C \sigma^{-3}+\widetilde{\sigma}^{-2} T_*) \leq \frac{1}{32}.
				\end{aligned}
			\end{equation*}
One can obtain that
				$$
				\begin{aligned}
					& \sum_{k=1}^{\infty}\Big(\Gamma^{k+1}(T_*)+\int_0^{T_*}(|(h^{k+1})^{\frac{1}{2}}(l^k)^{\frac{\nu}{2}} \nabla \bar{l}_t^{k+1}|_2^2+|h^{k+1}\nabla \bar{u}_t^{k+1}|_2^2) \text{~d} t\Big)<\infty,
				\end{aligned}
				$$
				which, together  with  \ef{857}-\ef{310857}, implies 
				$$
				\begin{aligned}
					\lim _{k \rightarrow \infty}(\|\bar{\phi}^{k+1}\|_{s^{\prime}}+\|\bar{h}^{k+1}\|_{s^{\prime}}+\|\bar{u}^{k+1}\|_{s^{\prime}}+\|\bar{l}^{k+1}\|_{s^{\prime}}+|\bar{u}_t^{k+1}|_2+|\bar{l}_t^{k+1}|_2) & =0, 
				\end{aligned}
				$$
				for any $s^{\prime} \in[1,3)$. And there has a subsequence (still denoted by $\left(\phi^k, h^k, l^k, u^k\right)$ ) satisfies
		\begin{equation}\label{880}
			    \begin{aligned}
					& (\phi^k-\epsilon,h^k, l^k-\bar{l},u^k) \overset{L^{\infty}([0, T_*] ; H^{s^{\prime}})}{\longrightarrow}(\phi^\epsilon-\epsilon,h^\epsilon,l^\epsilon-\bar{l} , u^\epsilon), \\
		& (u_t^k, l_t^k) \overset{ L^{\infty}([0, T_*] ; L^2)}{\longrightarrow}(u_t^\epsilon, l_t^\epsilon) ,
				\end{aligned}
			\end{equation}
			for limit $\left(\phi^\epsilon, h^\epsilon, l^\epsilon, u^\epsilon\right)$.
			
Moreover, it follows from  \ef{857}-\ef{310857} that there has a subsequence (still denoted by $(\phi^k, h^k, l^k, u^k)$) converges to $\left(\phi^\epsilon, h^\epsilon, l^\epsilon, u^\epsilon\right)$, which also satisfies
the estimates in 
\ef{857}-\ef{310857} 
except those weighted estimates on $\phi^\epsilon$, $u^\epsilon$ and $l^\epsilon$. 

We claim that
				\begin{align*}
				(h^k)^{\frac{1}{2}} \nabla l^k, (h^k)^\frac32\nabla^3 u^k, t^\frac12(h^k)^{-\frac{1}{4}}l_{tt}^k & \overset{L^{\infty}([0, T_*] ; L^2)}{\rightharpoonup }(h^\epsilon)^{\frac{1}{2}} \nabla l^\epsilon, (h^\epsilon)^\frac32 \nabla^3 u^\epsilon, t^\frac12(h^\epsilon)^{-\frac{1}{4}}l_{tt}^\epsilon;\\
				t^\frac12(h^k)^{\frac{1}{4}} \nabla l_{tt}^k &\overset{L^{2}([0, T_*] ; L^2)}{\rightharpoonup} t^\frac12(h^\epsilon)^{\frac{1}{4}} \nabla l_{tt}^\epsilon.
				\end{align*}
    Indeed, since
				$$
				\begin{aligned}
					(h^k)^\frac12-(h^\epsilon)^\frac12 & =\frac{h^k-h^\epsilon}{(h^k)^\frac12+(h^\epsilon)^\frac12}, \quad
				(h^k)^{\frac{1}{4}}-(h^\epsilon)^{\frac{1}{4}}  =\frac{(h^k)^\frac12-(h^\epsilon)^\frac12}{(h^k)^{\frac{1}{4}}+(h^\epsilon)^{\frac{1}{4}}}, \\
					(h^k)^{-\frac{1}{4}}-(h^\epsilon)^{-\frac{1}{4}} & =\frac{-((h^k)^{\frac{1}{4}}-(h^\epsilon)^{\frac{1}{4}})}{(h^k)^{\frac{1}{4}}(h^\epsilon)^{\frac{1}{4}}},
				\end{aligned}
				$$
				and $h^k$ and $h^\epsilon$ have positive lower bounds independent of $k$, one has
	\begin{equation}\label{879}
				\|\big((h^k)^\frac12-(h^\epsilon)^\frac12,(h^k)^{\frac{1}{4}}-(h^\epsilon)^{\frac{1}{4}},(h^k)^{-\frac{1}{4}}-(h^\epsilon)^{-\frac{1}{4}} \big)\|_{L^{\infty}([0, T_*] ; L^{\infty})} \rightarrow 0,
\end{equation}
				as $k \rightarrow \infty$. Thus it follows from \ef{880}-\ef{879} and \ef{857}-\ef{310857} for $\left(\phi^k, h^k, l^k, u^k\right)$, the estimates for $\left(\phi^\epsilon, h^\epsilon, l^\epsilon, u^\epsilon\right)$ obtained above that
				\begin{align*}
					& \int_0^{T_*} \int_{\mathbb{R}^2}\big((h^k)^\frac12 \nabla l^k-(h^\epsilon)^\frac12 \nabla l^\epsilon\big) \mathcal{X} \text{d} x \text{d} t \\
					\leq & C(\|(h^k)^\frac12-(h^\epsilon)^\frac12\|_{L^{\infty}([0, T_*] ; L^{\infty})}+\|\nabla l^k-\nabla l^\epsilon\|_{L^{\infty}([0, T_*] ; L^2)}) T_* \rightarrow 0 \text { as } k \rightarrow \infty, \\\\
     &  \int_0^{T_*} \int_{\mathbb{R}^2}\big((h^k)^\frac32\nabla^3 u^k-(h^\epsilon )^\frac32\nabla^3 u^\epsilon\big) \mathcal{X} \text{d} x \text{d} t \\
					 \leq& C\|(h^k)^\frac12-(h^\epsilon)^\frac12\|_{L^{\infty}([0, T_*] ; L^{\infty})}T_*\\&
      + C\int_0^{T_*} \int_{\mathbb{R}^2}(h^k\nabla^3 u^k-h^\epsilon \nabla^3 u^\epsilon) \mathcal{X} \text{d} x \text{d} t \\\leq& C(\|(h^k)^\frac12-(h^\epsilon)^\frac12\|_{L^{\infty}([0, T_*] ; L^{\infty})}+\|h^k-h^\epsilon\|_{L^{\infty}([0, T_*] ; L^{\infty})})T_*\\&+C \int_0^{T_*} \int_{\mathbb{R}^2}(\nabla^3 u^k- \nabla^3 u^\epsilon) \mathcal{X} \text{d} x \text{d} t \rightarrow 0 \text { as } k \rightarrow \infty,\\\\
					& \int_0^{T_*} \int_{\mathbb{R}^2}\big(t^\frac12(h^k)^{-\frac{1}{4}}l_{tt}^k-t^\frac12(h^\epsilon)^{-\frac{1}{4}}l_{tt}^\epsilon\big) \mathcal{X} \text{d} x \text{d} t \\
					\leq & C\|(h^k)^{-\frac{1}{4}}-(h^\epsilon)^{-\frac{1}{4}} )\|_{L^{\infty}([0, T_*] ; L^{\infty})} T_* \\&+C\int_0^{T_*} \int_{\mathbb{R}^2}(t^\frac12l_{tt}^k-t^\frac12l_{tt}^\epsilon) \mathcal{X} \text{d} x \text{d} t \rightarrow 0 \text { as } k \rightarrow \infty, \\\\
					& \int_0^{T_*} \int_{\mathbb{R}^2}\big(t^\frac12(h^k)^{\frac{1}{4}} \nabla l_{tt}^k-t^\frac12(h^\epsilon)^{\frac{1}{4}} \nabla l_{tt}^\epsilon\big) \mathcal{X} \text{d} x \text{d} t \\
						\leq & C\|(h^k)^{\frac{1}{4}}-(h^\epsilon)^{\frac{1}{4}} )\|_{L^{\infty}([0, T_*] ; L^{\infty})} T_* \\&+C\int_0^{T_*} \int_{\mathbb{R}^2}(t^\frac12\nabla l_{tt}^k-t^\frac12\nabla l_{tt}^\epsilon) \mathcal{X} \text{d} x \text{d} t \rightarrow 0 \text { as } k \rightarrow \infty,
				\end{align*}
				for any test function $\mathcal{X}(t, x) \in C_c^{\infty}([0, T_*] \times \mathbb{R}^2)$, which, along with the lower semicontinuity of norms, yield the uniform boundedness of $(h^\epsilon)^{\frac{1}{2}} \nabla l^\epsilon, (h^\epsilon)^\frac32 \nabla^3 u^\epsilon, t^\frac12(h^\epsilon)^{-\frac{1}{4}}l_{tt}^\epsilon$   in $L^{\infty}([0, T_*] ; L^2)$ and $t^\frac12(h^\epsilon)^{\frac{1}{4}} l_{tt}^\epsilon$  in $L^{2}([0, T_*] ; L^2)$ with respect to $\epsilon$. Similarly, one can also obtain the other desired weighed estimates. Hence the corresponding weighted estimates for $\left(\phi^\epsilon, u^\epsilon, l^\epsilon\right)$ in \ef{857}-\ef{310857} hold also for the limit.
    
    Next we show that $(\phi^\epsilon, u^\epsilon, l^\epsilon)$ is a weak solution in the sense of distributions to \ef{858}. First, multiplying  $\ef{860}_3$ by any given $\mathcal{X}(t, x) \in C_c^{\infty}\left(\left[0, T_*\right] \times \mathbb{R}^2\right)$ on both sides, and integrating over $(0, t) \times \mathbb{R}^2$ for $t \in\left(0, T_*\right]$, one has
\begin{equation}\label{881}
\begin{aligned}
& \int_0^t \int_{\mathbb{R}^2}(l^{ k+1} \mathcal{X}_s-u^k \cdot \nabla l^{k+1} \mathcal{X}) \text{d} x \text{d} s =  \int l^{k+1} \mathcal{X}(t, x)-\int l_0 \mathcal{X}(0, x) \\& -\int_0^t \int_{\mathbb{R}^2}\big(a_4 (l^{k})^\nu h^{k+1} \Delta l^{ k+1}+a_5 (l^k)^\nu n^{k+1} (h^{k})^{2} H(u^{k})\big) \mathcal{X} \text{d} x \text{d} s \\
& -\int_0^t \int_{\mathbb{R}^2}\big(a_6 (l^k)^{\nu+1}\text{div} \psi^{k+1}+\Pi^{k+1}\big) \mathcal{X} \text{d} x \text{d} s.
\end{aligned}
\end{equation}

It follows from the uniform estimates obtained above that one can take the limit $k \rightarrow \infty$ in \ef{881} to get
\begin{equation}\label{882}
\begin{aligned}
& \int_0^t \int_{\mathbb{R}^2}(l^{ \epsilon} \mathcal{X}_s-u^\epsilon \cdot \nabla l^{\epsilon} \mathcal{X}) \text{d} x \text{d} s 
=  \int l^{\epsilon} \mathcal{X}(t, x)-\int l_0 \mathcal{X}(0, x)  \\&-\int_0^t \int_{\mathbb{R}^2}\big(a_4 (l^\epsilon)^\nu h^{\epsilon} \Delta l^{ \epsilon}+a_5 (l^\epsilon)^\nu n^{\epsilon} (h^{\epsilon})^{2} H(u^{\epsilon})\big) \mathcal{X} \text{d} x \text{d} s \\
& -\int_0^t \int_{\mathbb{R}^2}\big(a_6 (l^\epsilon)^{\nu+1}\text{div} \psi^{\epsilon}+\Theta(\phi^\epsilon,l^\epsilon,\psi^\epsilon)\big) \mathcal{X} \text{d} x \text{d} s.
\end{aligned}
\end{equation}

Similarly, one can show that ($\phi^\epsilon, u^\epsilon, l^\epsilon$) also satisfies the equations $\ef{858}_1$, $\ef{858}_2$ and the initial data in the sense of distributions. Thus, $(\phi^\epsilon, u^\epsilon, l^\epsilon)$ is a weak solution in the sense of distributions to the Cauchy problem \ef{858}.
				
				For the uniqueness. Let $(\phi_1, u_1, l_1)$ and $(\phi_2, u_2, l_2)$ be two  solutions to \ef{858} that satisfy  \ef{857}-\ef{310857}. Denote
				$$
				\begin{aligned}
					\bar{f}  =f_1-f_2, \quad \text{for}\quad f=h, \phi, u, l,\psi. 
				\end{aligned}
				$$
				Then \ef{858} gives
			\begin{equation}\label{wyx}
				\left\{\begin{aligned}
					& \bar{\phi}_t+u_1 \cdot \nabla \bar{\phi}+\bar{u} \cdot \nabla \phi_2+(\gamma-1)(\bar{\phi} \text{div} u_1+\phi_2 \text{div} \bar{u})=0, \\
					& \bar{u}_t+u_1 \cdot \nabla \bar{u}+l_1 \nabla \bar{\phi}+a_1 \phi_1 \nabla \bar{l}+a_2 l_1^\nu h_1 L \bar{u} \\
					=&-\bar{u} \cdot \nabla u_2-a_1 \bar{\phi} \nabla l_2-\bar{l} \nabla \phi_2-a_2(l_1^\nu h_1-l_2^\nu h_2) L u_2 \\
					&+a_2\big(h_1 \nabla l_1^\nu \cdot Q(u_1)-h_2 \nabla l_2^\nu \cdot Q(u_2)\big) \\
					&+a_3\big(l_1^\nu \psi_1 \cdot Q\left(u_1\right)-l_2^\nu \psi_2 \cdot Q\left(u_2\right)\big), \\
					& \bar{l}_t+u_1 \cdot \nabla \bar{l}+\bar{u} \cdot \nabla l_2-a_4 \phi_1^{2\iota} l_1^\nu \Delta \bar{l} \\
					=&a_4(\phi_1^{2 \iota}l_1^\nu \Delta l_2-\phi_2^{2 \iota} l_2^\nu \Delta l_2) +a_5\big(l_1^\nu n_1 \phi_1^{4 \iota} H(u_1)-l_2^\nu n_2 \phi_2^{4 \iota} H(u_2)\big) \\
					&+a_6(l_1^{\nu+1}  \text{div} \psi_1-l_2^{\nu+1}  \text{div} \psi_2)+\Theta(\phi_1, l_1, \psi_1)-\Theta(\phi_2, l_2, \psi_2), \\
					& \bar{h}_t+u_1 \cdot \nabla \bar{h}+\bar{u} \cdot \nabla h_2+(\delta-1)(\bar{h} \text{div} u_1+h_2 \text{div} \bar{u})=0,\\
					& \bar{\psi}_t+\sum_{k=1}^2 A_k(u_1) \partial_k \bar{\psi}+B(u_1) \bar{\psi}+a \delta(\bar{h} \nabla \text{div} u_1+h_2 \nabla \text{div} \bar{u}) \\
					=&-\sum_{k=1}^2 A_k(\bar{u}) \partial_k \psi_2-B(\bar{u}) \psi_2, \\
					&(\bar{\phi}, \bar{u}, \bar{l}, \bar{h}, \bar{\psi})|_{t=0}=(0,0,0,0,0) \quad \text { in } \quad \mathbb{R}^2, \\
					&(\bar{\phi}, \bar{u}, \bar{l}, \bar{h}, \bar{\psi}) \longrightarrow(0,0,0,0,0) \quad \text { as }|x| \rightarrow \infty \quad \text { for } \quad t \geq 0 .
				\end{aligned}\right.
			\end{equation}
				
				Set
				$$
				\begin{aligned}
					\Phi(t)= & \|\bar{\phi}\|_2^2+|\bar{h}|_2+\|\bar{\psi}\|_1^2+|\bar{l}|_2+|h_1^{\frac{1}{2}} l_1^{\frac{\nu}{2}} \nabla \bar{l}|_2^2+| \bar{l}_t|_2^2+|l_1^{-\frac{\nu}{2}}h_1^{\frac12} \bar{u}|_2^2 \\
					& + \alpha|h_1 \nabla \bar{u}|_2^2+(\alpha+\beta)|h_1 \text{div} \bar{u}|_2^2+|l_1^{-\frac{\nu}{2}}h_1^{\frac12} \bar{u}_t|_2^2,
				\end{aligned}
				$$
    %
				 then one has
				$$
				\frac{\text{d}}{\text{d} t} \Phi(t)+C(|h_1^{\frac{1}{2}} l_1^{\frac{\nu}{2}} \nabla \bar{l}_t|_2^2+|h_1 \nabla \bar{u}_t|_2^2) \leq H(t) \Phi(t),
				$$
				where $H(t)$ satisfying
				$$
				\int_0^t H(s) \mathrm{d} s \leq C\quad \text { for }\quad 0 \leq t \leq T_*,
				$$
				and  Gronwall's inequality implies
				$$
		\bar{\phi}=\bar{u}=\bar{l}=0, \quad a.e. 
				$$
				thus the solution is unique.

 The time-continuity can be obtained as the process in proving  Lemma \ref{ls},  and this complete the proof of Theorem \ref{ddgs}.
			\end{proof}
   \subsection{Limit to the case with vacuum.}  With \ef{857}-\ef{310857} at hand, it is time to prove Theorem \ref{3.1}.
\begin{proof}
First,  assume that 
			$
			\phi_0^\epsilon=\phi_0+\epsilon
			$ for all $\epsilon\in (0,1)$, 
   then
	\begin{equation}\label{883}
			\begin{aligned}
				& u_0=(\phi_0+\epsilon)^{-\iota} g_1^\epsilon, \quad  \nabla u_0=(\phi_0+\epsilon)^{-2\iota} g_2^\epsilon,\quad Lu_0=(\phi_0+\epsilon)^{-3 \iota} g_3^\epsilon, \\
				& \nabla\big((\phi_0+\epsilon)^{3 \iota} L u_0\big)=(\phi_0+\epsilon)^{-\iota} g_4^\epsilon, \quad \nabla l_0=(\phi_0+\epsilon)^{-\iota} g_5^\epsilon, \\
				& \Delta l_0=(\phi_0+\epsilon)^{-2 \iota} g_6^\epsilon, \quad \nabla\big((\phi_0+\epsilon)^{2\iota} \Delta l_0\big)=(\phi_0+\epsilon)^{-\iota} g_7^\epsilon,
			\end{aligned} 
	\end{equation}
			where 
			$$
			 \begin{array}{l}
				g_1^\epsilon=\frac{\phi_0^{-\iota}}{(\phi_0+\epsilon)^{-\iota}} g_1, \quad g_2^\epsilon=\frac{\phi_0^{-2 \iota}}{(\phi_0+\epsilon)^{-2 \iota}} g_2, \quad g_3^\epsilon=\frac{\phi_0^{-3 \iota}}{(\phi_0+\epsilon)^{-3 \iota}} g_3, \\
				g_4^\epsilon=\frac{\phi_0^{-4 \iota}}{(\phi_0+\epsilon)^{-4 \iota}}(g_4-\frac{\epsilon \nabla \phi_0^{3 \iota}}{\phi_0+\epsilon} \phi_0^\iota L u_0),\quad 	g_5^\epsilon=\frac{\phi_0^{-\iota}}{(\phi_0+\epsilon)^{-\iota}} g_5,  \\
			 g_6^\epsilon=\frac{\phi_0^{-2 \iota}}{\left(\phi_0+\epsilon\right)^{-2 \iota}} g_6, \quad
				g_7^\epsilon=\frac{\phi_0^{-3 \iota}}{(\phi_0+\epsilon)^{-3 \iota}}(g_7-\frac{\epsilon \nabla \phi_0^{2\iota}}{\phi_0+\epsilon} \phi_0^{ \iota} \Delta l_0) .
			\end{array} 
			$$
			According to \ef{a}-\ef{2.8*}, one has 
			$$
				\begin{aligned}
					&2+\epsilon+\bar{l}+\|\phi^\epsilon_0-\epsilon\|_{L^p\cap D^1\cap D^2}+|\nabla (\phi_0^\epsilon
					)^{\iota+1}|_2+|(\phi_0^\epsilon
					)^{2\iota}\nabla^2 \phi_0^\epsilon
					|_2+\|u^\epsilon_0\|_{3}\\&+\|\nabla h^\epsilon_0\|_{L^q\cap D^{1}}+| (h^\epsilon_0)^{\frac{1}{2}}\nabla^3h^\epsilon_0|_2+|\nabla (h^\epsilon_0)^{\frac{1}{2}}|_4
					+|(h^\epsilon_0)^{-1}|_\infty \\&+\sum_{i=1}^7 |g^\epsilon_i|_2+\|l^\epsilon_0-\bar{l}\|_{3}+|(l^\epsilon_0)^{-1}|_\infty \le\bar{c}_0,
	\end{aligned}$$
	for some constants $\epsilon_1>0$ and $0<\epsilon<\epsilon_1$,
			where constant $\bar{c}_0>0$ is  independent of $\epsilon$. Thus for the initial data $(\phi_0^\epsilon, u_0^\epsilon, l_0^\epsilon)$, Theorem \ref{ddgs} implies that \ef{858} has a unique solution $(\phi^\epsilon, u^\epsilon, l^\epsilon)$ in $[0, T_*] \times \mathbb{R}^2$ satisfying  \ef{857}-\ef{310857} with $c_0$ replaced by $\bar{c}_0$, where $T_*$ does not depend on $\epsilon$.
			Moreover, one has
\begin{lemma}\label{Lyjx}
    For all $R_0>0$, $\epsilon \in(0,1]$, and ball $B_{R}$ which  centered at origin with radius $R$, there exists a constant $a_{R_0}$ independent of $\epsilon$ satisfying
$$
\phi^\epsilon(t, x) \geq a_{R_0}>0, \quad \forall(t, x) \in\left[0, T_*\right] \times B_{R_0}.
$$
\end{lemma}
The proof of  Lemma \ref{Lyjx} is similar to Lemma 3.9 in \cites{XinZ1}, the detail is omitted for simplicity.

Second, following the process in proving Theorem \ref{ddgs},  we can show that Theorem \ref{3.1} holds, i.e. $(\phi, u, l)$ is the unique strong solution  to  \ef{2.3} in $[0, T_*] \times \mathbb{R}^2$.
\end{proof}

\begin{proof}
The proof of Theorem \ref{th21} follows similar process as \cite{DXZ1}, and the details are omitted for simplicity.

\end{proof}

			\bigskip
			
			{\bf Acknowledgement:} The research was supported in part by National Key R\& D Program of China (No. 2022YFA1007300). 
            The research of Yue Cao is also supported in part by  National Natural Science Foundation of China under Grants 12401275. 

			\bigskip
			
			{\bf Conflict of Interest:} The authors declare  that they have no conflict of
			interest. 
			The authors also  declare that this manuscript has not been previously  published, and will not be submitted elsewhere before your decision. 
			
			\bigskip

            {\bf Data availability:}
Data sharing is  not applicable to this article as no data sets were generated or analyzed during the current study.
					\bigskip
			

		\end{document}